\documentclass[11pt,letterpaper]{article}
\usepackage[margin=1in]{geometry}
\usepackage[english]{babel}

\usepackage{amsmath}
\usepackage{amsthm}
\usepackage{amssymb}
\usepackage{enumitem}
\usepackage{pgfplots}
\usepackage{graphicx}
\usepackage{multirow}
\usepackage{physics}
\usepackage{tabularx,booktabs}
\newcolumntype{Y}{>{\raggedright\arraybackslash}X}
\usepackage{subcaption}
\usepackage{tikz,pgfplots}
\pgfplotsset{compat=1.18}

\usepackage[hidelinks]{hyperref}
\usepackage{algorithm}
\usepackage{algpseudocode}
\usepgfplotslibrary{groupplots}
\usetikzlibrary{
  arrows.meta,
  calc,
  positioning,
  fit,
  backgrounds,
  decorations.pathreplacing,
  matrix
}
\usepackage{tikz}
\usetikzlibrary{positioning}
\newtheorem{theorem}{Theorem}[section]
\newtheorem{lemma}[theorem]{Lemma}
\newtheorem{proposition}[theorem]{Proposition}
\newtheorem{corollary}[theorem]{Corollary}

\theoremstyle{definition}
\newtheorem{definition}[theorem]{Definition}
\newtheorem{example}[theorem]{Example}

\theoremstyle{remark}
\newtheorem{remark}[theorem]{Remark}

\newcommand{\Id}{\mathrm{Id}}
\newcommand{\INT}{\mathrm{INT}}
\newcommand{\LP}{\mathrm{LP}}
\newcommand{\SH}{\mathrm{SH}}

\newcommand{\TV}{\mathrm{TV}}
\newcommand{\Vit}{\mathrm{Vit}}

\newcommand{\fit}{\mathrm{fit}}

\newcommand{\Def}{\mathrm{def}}

\newcommand{\aud}{\mathrm{aud}}
\title{Residual Centering and Error Bounds for Convex Approximation of Mixed-Integer Recourse}
\author{
Alban Kryeziu\\
\small Faculty of Economics and Business, University of Groningen\\
\small \texttt{a.kryeziu@rug.nl}\\
\small \href{https://orcid.org/0009-0009-5913-2415}
{ORCID: 0009-0009-5913-2415}
}
\date{}

\begin{document}
\maketitle

\begin{abstract}
We study convex approximations of mixed-integer recourse functions in
two-stage stochastic programming. For first-stage decisions, the relevant approximation error is the signed expected residual between the convex
approximation and the integer-recourse value function. On bounded uncertainty boxes, we identify residual conditions that yield deterministic bounds for this expectation error. The central condition is coordinate slice-centering, under which integration by parts gives an anisotropic total-variation bound and a mixed-derivative bound whose all-coordinate case is expressed through Vitali variation of the density. Exact centering, however, may be incompatible with convexity, even for totally unimodular ceiling recourse. We therefore use commuting coordinate projectors to decompose an arbitrary residual into a centered component and an explicit slice-mean defect. The resulting defect-adjusted bounds provide selectable error certificates and uniform first-stage decision-quality guarantees. We further establish sharp constants, extend the mixed-derivative estimate to nonsmooth densities, and develop tensor-product and geometric extensions of the framework, alongside a max-affine convex fitting and audit methodology.
\end{abstract}


\clearpage
\tableofcontents
\clearpage

\section{Introduction}
\label{sec:introduction}

Consider the two-stage stochastic mixed-integer program
\begin{align}
\min_{x\geq 0}\quad & c^\top x+\mathbb E_\omega[v^{\INT}(x,\omega)]
\label{eq:intro_smip}
\\
\text{s.t.}\quad & Ax=h,
\nonumber
\end{align}
where the second-stage value function is
\begin{equation}
\label{eq:intro_second_stage}
v^{\INT}(x,\omega)
=
\min_y
\left\{
q^\top y:
Wy\geq \omega-Tx,\;
y\in\mathbb Z^{n_2}_+\times\mathbb R^{n_3}_+
\right\}.
\end{equation}
The recourse variables represent recourse second-stage decisions after the
random vector $\omega$ has been observed. Throughout the paper, the first-stage
feasible set is assumed nonempty, and the recourse value is assumed finite and
integrable for every feasible first-stage decision.

Mixed-integer recourse models are useful because they represent genuinely
discrete recourse actions. They are difficult for the same reason. Integer variables destroy the convexity and regularity properties used by
classical decomposition methods. Even simple integer recourse leads to
nonsmooth and nonconvex value functions \cite{vlerk1}, and general
mixed-integer recourse is harder still \cite{schultz2,shapiro}. A common solution approach is to replace
$v^{\INT}$ by a tractable convex approximation and solve the resulting
first-stage problem
\cite{VanderVlerk2004,vanDerLaanEtAl2021}.

This replacement raises a basic question: what error should such an
approximation control? Pointwise error is natural, but it is not the quantity
that enters the first-stage objective. If $\widetilde v$ is a convex
approximation of $v^{\INT}$ and
$$
\varphi_x(\omega):=\widetilde v(x,\omega)-v^{\INT}(x,\omega),
$$
then the error in the expected recourse term is
$$
\widetilde Q(x)-Q(x)=\mathbb E_f[\varphi_x].
$$
Thus the decision-relevant object is the signed expected residual. Positive
and negative residuals may cancel under the density, while a smaller residual
with persistent bias can still perturb the first-stage objective. The question
studied in this paper is therefore not only how large the residual is, but
which structural properties of the residual make its expectation controllable. This expectation-error viewpoint is present in the shifted-LP and convex
approximation literature, where one typically starts from a particular
approximation and then proves a bound for the residual generated by that
construction. The departure in this paper is to reverse that order. We start
from an arbitrary signed residual and ask which residual structures imply
expectation control, independently of how the approximation was produced.

\subsection*{A residual-first view}

The existing convex-approximation literature is largely
approximation-first. One constructs a particular approximation, often through
shifted distributions, shifted LP relaxations, or related perturbations, and
then proves an expectation bound for the residual produced by that
construction. Convex approximations for complete integer recourse are developed
in \cite{vlerkthesis,VanderVlerk2004}. Uniform error bounds for totally
unimodular integer recourse are derived in \cite{ward2015}, with extensions
beyond the totally unimodular case in \cite{ward2015bb}. The periodic
mechanism behind these bounds is clarified in \cite{total1} through
total-variation estimates for periodic residual functions. Further extensions
treat performance assessment, mean-risk functionals, and random second-stage costs
\cite{ward_assessing,vanBeestenRomeijnders2020,vanBeestenRomeijnders2022}. Higher-order total-variation refinements based on higher-order
derivatives of the density, with applications to simple integer recourse,
are developed in \cite{vanDerLaanEtAl2018HigherOrder}. The mixed-derivative formulation developed here differs in being box-based,
multivariate, indexed by arbitrary selected coordinate sets, and coupled
to an explicit projector defect for noncentered residuals.

We take the reverse route. The primitive object is the signed residual
$$
\varphi=\widetilde v-v^{\INT},
$$
not the approximation architecture that generated it. This residual-first view
separates three issues that are usually intertwined: the approximation family,
the cancellation mechanism, and the density-variation certificate. It leads to
a structural question: which residuals have small signed expectation on a
bounded uncertainty box, and which of those residual structures can actually be
induced by a convex approximation of a mixed-integer recourse value function?

The answer developed here is coordinate slice-centering and its
defect-adjusted relaxation. Exact slice-centering identifies the cancellation
mechanism: it removes the boundary terms in integration by parts and gives
expectation bounds in terms of anisotropic total variation or
mixed-derivative/Vitali-type variation of the density. Exact centering,
however, is not a free normalization. It is a convex-realizability condition on
slice-average profiles of $v^{\INT}$. When this condition fails, the relevant
object is the projector decomposition
$$
\varphi=\mathcal P_I\varphi+\mathcal R_I\varphi,
$$
where the centered component is controlled by density variation and the
slice-mean defect is paid for explicitly.

This perspective differs from the prevailing decomposition literature in stochastic mixed-integer programming. Integer L-shaped cuts, disjunctive cuts, Fenchel cuts, Gomory-based cuts, split cuts, strengthened Benders cuts, and Lagrangian cuts are all designed primarily to solve stochastic integer programs or improve their computational performance (see \cite{ChenLuedtke2022LagrangianCuts}). Our goal is different. We do not propose a new class of decomposition cuts. Instead, we identify residual structures—whether from shifted relaxations, max-affine approximations, or other convex surrogates—that certify when a signed expectation is small.

Convex regression, max-affine fitting, input-convex neural networks, and
difference-of-convex methods provide flexible ways to represent or learn convex
functions
\cite{MagnaniBoyd2009ConvexPWL,seijo_sen,input_convex_neural_networks,KAZDA2024493}.
In this paper these tools are secondary. The main question is not how to
parameterize a convex approximation, but which residual structures make such an
approximation reliable in expectation.

The article \cite{alban1} establishes Vitali-variation error bounds for two-dimensional expected value functions under compact-support conditions. Its main object is the expected value function and its control through variation of the probability density. The present paper studies a different structural question: how properties of the signed residual itself determine the expectation error of a convex approximation of mixed-integer recourse. Accordingly, slice-centering, residual projectors, explicit centering defects, and their compatibility with convex approximation are treated here as the primary objects of the analysis.

\subsection*{Cancellation, defects, and convexity.}

The bounds rely on a simple but useful construction rather than on a prescribed
approximation formula. For the anisotropic total-variation bound, we build
one-dimensional primitives of the residual along selected coordinate
directions. When the residual has zero coordinate-slice averages, these
primitives have zero normal trace on the boundary of the box. The divergence
theorem, followed by the $BV$ integration-by-parts formula, then transfers the
problem from the residual to the directional variation of the density. For the
mixed-derivative bound, the analogous construction uses iterated orthant
primitives; the same slice-centering condition removes the boundary terms in
the repeated one-dimensional integrations. Thus slice-centering is the
cancellation mechanism behind both bounds. The coordinate averages,
projectors, residual envelopes, and density-variation quantities used to make
this precise are introduced in Section~\ref{sec:problem_setup}.

This centering condition should not be viewed as a harmless normalization. A
convex approximation is constrained through $\widetilde v$, whereas centering
is a condition on the difference $\widetilde v-v^{\INT}$. Since
$v^{\INT}$ is typically nonconvex, exact centering requires the slice averages
of $v^{\INT}$ to be compatible with a convex function of the remaining
coordinates. We show that this compatibility can already fail for a separable
totally unimodular ceiling-recourse value. Exact centering is therefore best
understood as the ideal cancellation case, not as a condition that can be
imposed generically on convex approximations.

This also clarifies the relation with shifted-LP theory. In one-dimensional
simple integer recourse, the half-shifted LP residual is a centered
unit-periodic sawtooth. On full unit cells, periodic zero-mean cancellation and
one-dimensional slice-centering coincide. That coincidence is special to the one-dimensional case. The
shifted-LP theory obtains bounds from the residual structure of a particular
approximation. Here the residual structure is abstracted first. On bounded
boxes and in dimension $d\geq2$, periodic zero mean over a lattice cell,
coordinate slice-centering on the box, and convex representability of a
centered profile are different requirements. The present theory is therefore
related to shifted-LP bounds through signed residual cancellation, but its
cancellation condition is box-based and approximation-independent rather than
periodic and construction-specific.

The resulting bounds are therefore defect-adjusted: the centered component is
controlled by density variation, while the slice-mean defect is priced
explicitly. Since the bounds are available for every nonempty direction set
$I$ and through two different variation estimates, the direction set and the
certificate can be selected after the residual and density have been specified.
This yields a best-certificate principle rather than a single prescribed bound.

The extended theory also identifies the scope and the limitations of the box
geometry. Tensor-product residuals and densities admit exact projector and
variation formulas, and a single smooth tensorized construction proves that
both variation constants are optimal. The mixed-derivative formulation is stable
under mollification once the classical derivative is replaced by a finite
distributional mixed-derivative measure of an admissible whole-space extension.
This gives a nonsmooth extension seminorm, localization bounds for full-support
laws, and an exact affine transport to parallelotopes. On a general Lipschitz
domain, the same argument yields only a domain-dependent first-order $BV$
substitute; the explicit box constants and the classical Vitali interpretation
are no longer preserved.

\subsection*{Contributions } 

The contributions of this paper are as follows.

\begin{itemize}
\item We develop a residual-first framework for signed expectation error on
bounded boxes. For residuals centered on an arbitrary nonempty direction set
$I$, we derive an anisotropic first-order $BV$ bound and a mixed-derivative
bound whose all-coordinate case is a smooth Vitali-variation bound.

\item We use commuting coordinate-average projectors to decompose every
residual as
$$
\varphi=\mathcal P_I\varphi+\mathcal R_I\varphi.
$$
The centered component is controlled by density variation and the
slice-average component enters as an explicit defect. Optimization over the
direction set and variation formulation gives a best-certificate principle and a
uniform first-stage decision-quality guarantee.

\item We derive exact tensor-product formulas for the projectors, residual
envelopes, coordinate variation, and mixed derivatives. A tensorized
construction proves that the coefficients $1/(2|I|)$ and $2^{-|I|}$ are sharp
for every nonempty $I$, and that the unit coefficient multiplying the defect
is also optimal.

\item We extend the mixed-derivative formulation to nonsmooth densities through a
whole-space extension seminorm. The same framework yields localization bounds
for full-support laws, exact affine transport to parallelotopes, and a
domain-dependent first-order $BV$ substitute on general Lipschitz domains.

\item We show that exact residual centering by a convex approximation is a
convex-realizability condition that may fail even for totally unimodular
ceiling recourse. We then translate the defect decomposition into a max-affine
fitting and audit framework and illustrate the resulting diagnostics
numerically.
\end{itemize}

The paper is organized as follows. Section~\ref{sec:problem_setup} develops the residual projectors and the exact,
defect-adjusted, sharp, and nonsmooth variation bounds.
Section~\ref{sec:exact_centering_obstruction} establishes the
convex-realizability obstruction. Sections~\ref{sec:framework_design_revised}
and~\ref{sec:numerics} present the fitting framework and numerical
illustration. The appendices contain the tensorization, sharpness,
mollification, geometric-extension, cut-generation, obstruction, and periodic
comparison details.

\section{Problem setup and residual-based variation bounds}
\label{sec:problem_setup}
 
This section develops the residual-side machinery used throughout the paper.
All estimates are stated first on the fixed box $S$, because this is the
uncertainty domain on which the residual is integrated and the natural domain
for the integration-by-parts arguments below. We introduce the coordinate
projectors and the residual envelopes, derive the exact and defect-adjusted
variation bounds, and then state the best-certificate, tensorization,
sharpness, and nonsmooth-extension consequences. The final part of the section
explains how the box estimates localize full-support laws and records the
precise geometric scope of the theory.

Let $S=\prod_{i=1}^d[a_i,b_i]\subset\mathbb R^d$ be the uncertainty box, and let
$u:=|S|^{-1}\mathbf 1_S$ denote the uniform density on $S$. 

We write
$[d]:=\{1,\ldots,d\}$. For $I\subseteq[d]$, write
$I=\{i_1<\cdots<i_k\}$ with $k=|I|$, and set
$\omega_I:=(\omega_i)_{i\in I}$, $\omega_{-I}:=(\omega_j)_{j\notin I}$,
$S_I:=\prod_{i\in I}[a_i,b_i]$, and
$S_{-I}:=\prod_{j\notin I}[a_j,b_j]$. If $I=\{i\}$, we write
$S_{-i}$ and $\omega_{-i}$.

Let $X:=\{x\in\mathbb R^{n_1}_+:Ax=h\}$ be the feasible first-stage set. For an
integrable function $\varphi:S\to\mathbb R$ and a probability density $f$ on
$S$, set
$$\mathbb E_f[\varphi]:=\int_S\varphi(\omega)f(\omega)\,d\omega, \text{ and } \mathbb E_u[\varphi]:=|S|^{-1}\int_S\varphi(\omega)\,d\omega$$ 
For $x\in X$,
define
$Q(x):=\mathbb E_f[v^{\INT}(x,\omega)]$ and
$\widetilde Q(x):=\mathbb E_f[\widetilde v(x,\omega)]$, with first-stage
objectives $F(x):=c^\top x+Q(x)$ and
$\widetilde F(x):=c^\top x+\widetilde Q(x)$.

The approximation residual is
\begin{equation}
\label{eq:residual_definition_section}
\varphi_x(\omega):=\widetilde v(x,\omega)-v^{\INT}(x,\omega),
\end{equation}
and therefore
\begin{equation}
\label{eq:expected_residual_identity_section}
\widetilde Q(x)-Q(x)=\mathbb E_f[\varphi_x].
\end{equation}
Thus the quantity that enters the first-stage objective is the signed expected
residual, not only the pointwise approximation error. In particular, if
$\sup_{x\in X}|\mathbb E_f[\varphi_x]|\le\varepsilon$, then the standard
decision-quality estimate gives
$$
0\le F(\widetilde x)-F(x^\star)\le 2\varepsilon,
\qquad
|F(x^\star)-\widetilde F(\widetilde x)|\le \varepsilon,
$$
for minimizers $x^\star\in\arg\min_X F$ and
$\widetilde x\in\arg\min_X\widetilde F$; see, for example,
\cite{ward2015,wardthesis}. We therefore seek bounds for
$\mathbb E_f[\varphi_x]$ in terms of variation properties of the density and
structural properties of the residual.

The two variation objects used below measure different aspects of the density.
The anisotropic total variation measures first-order coordinate variation,
whereas the Vitali variation, in the smooth case, measures the top-order mixed
derivative. The estimates are driven by the same residual-side cancellation
principle: coordinate slice-centering.

\subsection{Coordinate averages, residual projectors, and envelopes } 
\label{subsec:coordinate_projectors}

The cancellation mechanism used below is coordinate slice-centering. For
$i\in[d]$ and $\varphi\in L^1(S)$, define the coordinate-$i$ slice-average
profile by
$$
(\bar\Pi_i\varphi)(\omega_{-i})
:=
\frac{1}{b_i-a_i}
\int_{a_i}^{b_i}\varphi(\omega_{-i},t)\,dt,
\qquad
\text{for a.e. }\omega_{-i}\in S_{-i}.
$$
Here $\varphi(\omega_{-i},t)$ means that $t$ is inserted in the $i$th
coordinate. The lifted coordinate-average operator is
$$
(\Pi_i\varphi)(\omega):=(\bar\Pi_i\varphi)(\omega_{-i}).
$$
Thus $\Pi_i\varphi$ is constant in coordinate $i$ and equals the average of
$\varphi$ on the corresponding coordinate-$i$ slice.

We say that $\varphi$ is centered in direction $i$ if $\Pi_i\varphi=0$ a.e. on
$S$, or equivalently if
$$
\int_{a_i}^{b_i}\varphi(\omega_{-i},t)\,dt=0
\qquad
\text{for a.e. }\omega_{-i}\in S_{-i}.
$$
For a nonempty $I\subseteq[d]$, $\varphi$ is centered on $I$ if it is centered
in every direction $i\in I$.

For $I\subseteq[d]$, define
\begin{equation}
\label{eq:projectors_section}
\mathcal P_I:=\prod_{i\in I}(\Id-\Pi_i),
\qquad
\mathcal P_\varnothing:=\Id,
\qquad
\mathcal R_I:=\Id-\mathcal P_I.
\end{equation}
The product is independent of the ordering because the coordinate averages
commute. If $\Pi_J:=\prod_{j\in J}\Pi_j$ and $\Pi_\varnothing:=\Id$, then
$$
\mathcal P_I
=
\sum_{J\subset I}(-1)^{|J|}\Pi_J,
\qquad
\mathcal R_I
=
\sum_{\emptyset\neq J\subset I}(-1)^{|J|+1}\Pi_J .
$$

Thus $\mathcal P_I$ extracts the component that is annihilated by every
selected coordinate-average operator:
\[
\Pi_i\mathcal P_I\varphi=0,
\qquad i\in I.
\]
Equivalently, $\mathcal P_I\varphi$ is centered on every coordinate slice
indexed by $I$. The complementary projector
$\mathcal R_I=\Id-\mathcal P_I$ collects all residual components detected by
at least one of the selected coordinate averages. Every residual therefore
admits the exact decomposition
\[
\varphi
=
\mathcal P_I\varphi+\mathcal R_I\varphi.
\]

Figure~\ref{fig:projector_decomposition_compact} illustrates this identity
using an exact two-dimensional example. Let $S=[0,1]^2$, let
$I=\{1,2\}$, and define
\[
c(x,y)
:=
\sin(2\pi x)\sin(2\pi y),
\]
\[
d(x,y)
:=
0.35\cos(2\pi x)
+
0.25\sin(2\pi y)
+
0.10,
\]
and
\[
\varphi(x,y):=c(x,y)+d(x,y).
\]
Because both one-dimensional sine factors have zero mean,
\[
\Pi_1c=\Pi_2c=0.
\]
Moreover, $d$ consists only of a coordinate-$1$ component, a coordinate-$2$
component, and a constant component. Hence
\[
\mathcal P_{\{1,2\}}d=0.
\]
Consequently,
\[
\mathcal P_{\{1,2\}}\varphi=c,
\qquad
\mathcal R_{\{1,2\}}\varphi=d.
\]

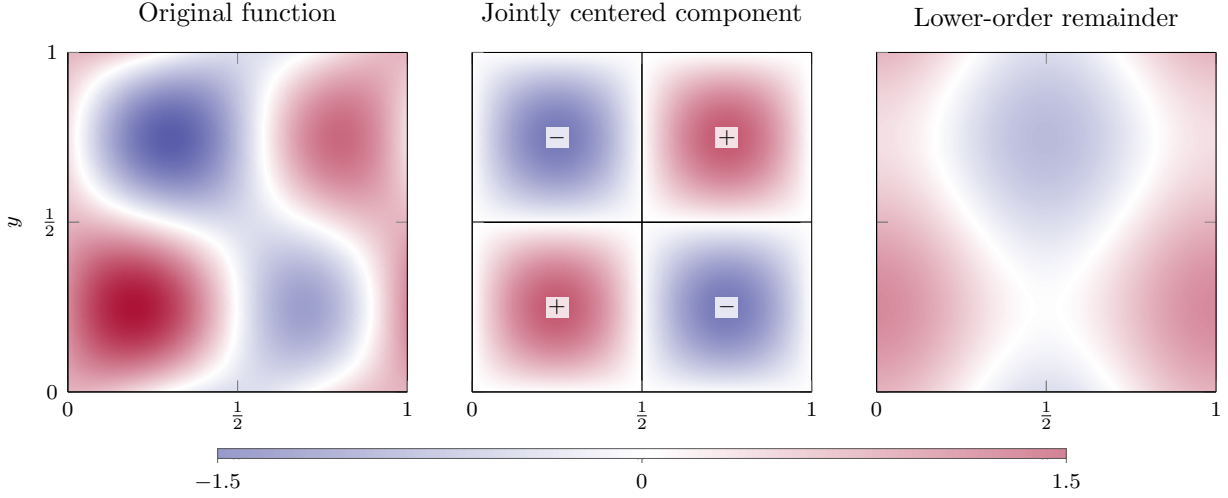
\begin{figure}[t]
\centering

\definecolor{projectorblue}{RGB}{49,54,149}
\definecolor{projectorred}{RGB}{165,0,38}

\pgfplotsset{
  colormap={projectorsignedmap}{
    color=(projectorblue)
    color=(white)
    color=(projectorred)
  }
}

\pgfdeclarehorizontalshading{projectorsignedbar}{100bp}{
  color(0bp)=(projectorblue);
  color(50bp)=(white);
  color(100bp)=(projectorred)
}

\begin{tikzpicture}

\begin{groupplot}[
  group style={
    group size=3 by 1,
    horizontal sep=0.052\textwidth
  },
  width=0.272\textwidth,
  height=0.272\textwidth,
  scale only axis,
  view={0}{90},
  axis equal image,
  xmin=0,
  xmax=1,
  ymin=0,
  ymax=1,
  xtick={0,0.5,1},
  ytick={0,0.5,1},
  xticklabels={$0$,$\tfrac12$,$1$},
  yticklabels={$0$,$\tfrac12$,$1$},
  xlabel={$x$},
  tick label style={font=\scriptsize},
  label style={font=\scriptsize},
  title style={
    font=\small,
    align=center
  },
  point meta min=-1.5,
  point meta max=1.5,
  colormap name=projectorsignedmap,
  enlargelimits=false
]

\nextgroupplot[
  title={Original function},
  ylabel={$y$}
]
\addplot3[
  surf,
  shader=interp,
  domain=0:1,
  y domain=0:1,
  samples=51,
  samples y=51
]
{
  sin(360*x)*sin(360*y)
  +0.35*cos(360*x)
  +0.25*sin(360*y)
  +0.10
};

\nextgroupplot[
  title={Jointly centered component},
  yticklabels={}
]
\addplot3[
  surf,
  shader=interp,
  domain=0:1,
  y domain=0:1,
  samples=51,
  samples y=51
]
{
  sin(360*x)*sin(360*y)
};

\draw[black,semithick]
  (axis cs:0.5,0) -- (axis cs:0.5,1);
\draw[black,semithick]
  (axis cs:0,0.5) -- (axis cs:1,0.5);

\node[
  font=\scriptsize,
  fill=white,
  fill opacity=.82,
  text opacity=1,
  inner sep=1pt
] at (axis cs:0.25,0.25) {$+$};

\node[
  font=\scriptsize,
  fill=white,
  fill opacity=.82,
  text opacity=1,
  inner sep=1pt
] at (axis cs:0.75,0.25) {$-$};

\node[
  font=\scriptsize,
  fill=white,
  fill opacity=.82,
  text opacity=1,
  inner sep=1pt
] at (axis cs:0.25,0.75) {$-$};

\node[
  font=\scriptsize,
  fill=white,
  fill opacity=.82,
  text opacity=1,
  inner sep=1pt
] at (axis cs:0.75,0.75) {$+$};

\nextgroupplot[
  title={Lower-order remainder},
  yticklabels={}
]
\addplot3[
  surf,
  shader=interp,
  domain=0:1,
  y domain=0:1,
  samples=51,
  samples y=51
]
{
  0.35*cos(360*x)
  +0.25*sin(360*y)
  +0.10
};

\end{groupplot}

\begin{scope}[
  shift={($(group c2r1.south)+(0,-0.88cm)$)}
]
  \shade[shading=projectorsignedbar]
    (-0.34\textwidth,0)
    rectangle
    (0.34\textwidth,0.13cm);

  \draw[black!65,line width=0.25pt]
    (-0.34\textwidth,0)
    rectangle
    (0.34\textwidth,0.13cm);

  \draw[black!65,line width=0.25pt]
    (-0.34\textwidth,0)
    --
    (-0.34\textwidth,-0.06cm);

  \draw[black!65,line width=0.25pt]
    (0,0)
    --
    (0,-0.06cm);

  \draw[black!65,line width=0.25pt]
    (0.34\textwidth,0)
    --
    (0.34\textwidth,-0.06cm);

  \node[anchor=north,font=\scriptsize]
    at (-0.34\textwidth,-0.08cm)
    {$-1.5$};

  \node[anchor=north,font=\scriptsize]
    at (0,-0.08cm)
    {$0$};

  \node[anchor=north,font=\scriptsize]
    at (0.34\textwidth,-0.08cm)
    {$1.5$};
\end{scope}

\end{tikzpicture}

\caption{
Exact coordinate-projector decomposition on $S=[0,1]^2$ with
$I=\{1,2\}$. The original function in the left panel satisfies
$\varphi=c+d$. The middle panel is the jointly centered component
$c=\mathcal P_{\{1,2\}}\varphi$, and the right panel is the lower-order
remainder $d=\mathcal R_{\{1,2\}}\varphi$. The black lines mark the interior
zero sets $x=\tfrac12$ and $y=\tfrac12$ of $c$. All three panels use the same
signed scale $[-1.5,1.5]$.
}
\label{fig:projector_decomposition_compact}
\end{figure}
The example also clarifies the inclusion--exclusion role of the projector. The
centered term $c$ is a genuine two-coordinate interaction: averaging it in
either coordinate removes it completely. By contrast, the defect $d$ is built
from lower-order components that survive at least one coordinate average.
Indeed,
\[
\Pi_1d
=
0.25\sin(2\pi y)+0.10,
\qquad
\Pi_2d
=
0.35\cos(2\pi x)+0.10,
\]
and
\[
\Pi_1\Pi_2d=0.10.
\]
Therefore
\[
\mathcal P_{\{1,2\}}d
=
d-\Pi_1d-\Pi_2d+\Pi_1\Pi_2d
=
0.
\]
This exact example replaces a purely schematic interpretation: the middle
panel genuinely has zero coordinate-slice averages, while the right panel
genuinely belongs to the projector-defect space.

\begin{proposition}[Coordinate-average projector algebra]
\label{prop:coordinate_projector_algebra}
Let $I\subseteq[d]$. The coordinate-average operators $\Pi_i$ are bounded linear
contractions on $L^p(S)$ for every $1\le p\le\infty$. They are idempotent and
commute:
$$
\Pi_i^2=\Pi_i,
\qquad
\Pi_i\Pi_j=\Pi_j\Pi_i .
$$
Consequently, $\mathcal P_I$ and $\mathcal R_I$ are bounded linear projections
on $L^p(S)$ for every $1\le p\le\infty$, and
$$
\mathcal P_I\mathcal R_I=\mathcal R_I\mathcal P_I=0.
$$
Moreover, if $i\in I$, then $\Pi_i\mathcal P_I=0$, so
$\mathcal P_I\varphi$ is centered on $I$ for every $\varphi\in L^1(S)$.
For nonempty $I$,
$$
\mathcal R_I\varphi=0
\quad\Longleftrightarrow\quad
\varphi \text{ is centered on } I.
$$
Finally, if $\varphi\in W^{1,\infty}(S)$, then
$\mathcal P_I\varphi,\mathcal R_I\varphi\in W^{1,\infty}(S)$.
\end{proposition}

\begin{proof}
The proof checks the standard properties of the coordinate-average operators in
sequence: contraction, idempotence, commutativity, and the centering identity
for $\mathcal P_I$.

Linearity is immediate. Jensen's inequality gives the $L^p$ contraction for
$1\le p<\infty$, and the $L^\infty$ contraction follows directly from the
definition. Since $\Pi_i\varphi$ is already independent of $\omega_i$, one has
$\Pi_i^2=\Pi_i$. The commutation relation follows from Fubini's theorem:
averaging in coordinates $i$ and $j$ can be performed in either order.

The factors $\Id-\Pi_i$ are commuting projections, hence their product
$\mathcal P_I$ is a projection. Therefore
$\mathcal R_I=\Id-\mathcal P_I$ is also a projection, and
$\mathcal P_I\mathcal R_I=\mathcal R_I\mathcal P_I=0$. Boundedness follows
from boundedness of the factors.

If $i\in I$, then
$$
\Pi_i\mathcal P_I
=
\Pi_i(\Id-\Pi_i)
\prod_{j\in I\setminus\{i\}}(\Id-\Pi_j)
=0,
$$
which proves that $\mathcal P_I\varphi$ is centered on $I$.

If $\mathcal R_I\varphi=0$, then $\varphi=\mathcal P_I\varphi$, and hence
$\varphi$ is centered on $I$. Conversely, if $\Pi_i\varphi=0$ for every
$i\in I$, then each factor $\Id-\Pi_i$ leaves $\varphi$ unchanged. Thus
$\mathcal P_I\varphi=\varphi$ and $\mathcal R_I\varphi=0$.

It remains only to record Sobolev stability. If
$\varphi\in W^{1,\infty}(S)$, then $\Pi_i\varphi$ is independent of
$\omega_i$, so $\partial_i(\Pi_i\varphi)=0$ weakly. For $j\neq i$, weak
differentiation under the integral gives
$$
\partial_j(\Pi_i\varphi)=\Pi_i(\partial_j\varphi).
$$
Since $\Pi_i$ is an $L^\infty$ contraction, $\Pi_i$ maps
$W^{1,\infty}(S)$ into itself. The same holds for finite sums and products of
the $\Pi_i$, and hence for $\mathcal P_I\varphi$ and
$\mathcal R_I\varphi$.
\end{proof}

The estimates below use three different measures of residual size, each adapted
to a different integration-by-parts argument. The quantity
$\Lambda_{\mathrm{pr}}^i$ measures the size of a one-dimensional primitive and
is used in the one-direction bound. The quantity $\Lambda_{\mathrm{sl}}^I$ is
the slice $L^1$ envelope used in the anisotropic $\TV_\infty$ estimate. The
quantity $\Lambda_{\mathrm{mix}}^I$ is the full $I$-slice integral used in the
mixed-derivative estimate. For $i\in[d]$, define
$$
\Lambda_{\mathrm{pr}}^i(\varphi)
:=
\operatorname*{ess\,sup}_{\omega_{-i}\in S_{-i}}
\sup_{t\in[a_i,b_i]}
\left|
\int_{a_i}^{t}\varphi(\omega_{-i},s)\,ds
\right|.
$$
For nonempty $I\subseteq[d]$, the coordinate-slice $L^1$ envelope is
$$
\Lambda_{\mathrm{sl}}^I(\varphi)
:=
\max_{i\in I}
\operatorname*{ess\,sup}_{\omega_{-i}\in S_{-i}}
\int_{a_i}^{b_i}|\varphi(\omega_{-i},t)|\,dt .
$$
When $I=[d]$, we write
$\Lambda_{\mathrm{sl}}(\varphi):=\Lambda_{\mathrm{sl}}^{[d]}(\varphi)$.

For the mixed-derivative estimates, the relevant envelope is the
$I$-coordinate slice integral
$$
\Lambda_{\mathrm{mix}}^I(\varphi)
:=
\operatorname*{ess\,sup}_{\omega_{-I}\in S_{-I}}
\int_{S_I}|\varphi(\omega_{-I},t_I)|\,dt_I .
$$
If $I=[d]$, then $S_{-I}$ is a singleton and
$$
\Lambda_{\mathrm{mix}}^{[d]}(\varphi)=\|\varphi\|_{L^1(S)}.
$$

\subsection{Density variation } 
\label{subsec:density_variation}

For $f\in BV(S)$, the coordinatewise anisotropic variation is
\begin{equation}
\label{eq:def_anisotropic_tv}
\TV_\infty(f;S):=\sum_{i=1}^d |D_i f|(S),
\end{equation}
where $D_i f$ is the $i$th distributional derivative measure of $f$ on $S$.
For $I\subseteq[d]$, we also write
$$
\TV_I(f;S):=\sum_{i\in I}|D_i f|(S).
$$
Throughout, distributional derivatives on the box are understood relative to
its interior. Thus, for $f\in BV(S)$, the notation $D_i f$ means the
distributional derivative on $S^\circ$, and
$$
|D_i f|(S)
\quad\text{is shorthand for}\quad
|D_i f|(S^\circ).
$$
Similarly, for a one-dimensional interval $J=[a,b]$,
$|Df|(J)$ means the relative variation $|Df|((a,b))$ unless a whole-space
extension is explicitly being considered. Consequently, the quantities
$\TV_I(f;S)$ and $\TV_\infty(f;S)$ do not include boundary jump measures.
Boundary-supported derivative mass appears only after a whole-space extension
has been introduced.

Equivalently,
$$
\TV_\infty(f;S)
=
\sup_{\substack{\psi\in C_c^1(\operatorname{int}S;\mathbb R^d)\\
\|\psi(\omega)\|_\infty\le1}}
\int_S f(\omega)\operatorname{div}\psi(\omega)\,d\omega .
$$
If $f\in C^1(\overline S)$, then
$\TV_\infty(f;S)=\int_S\|\nabla f(\omega)\|_1\,d\omega$. Thus, for smooth densities, the computed quantity is the $\ell^1$-gradient
integral, not the more common Euclidean gradient norm. It is equivalent to
the usual total variation of the gradient measure up to dimension-dependent
constants:
$$
|Df|(S)\le \TV_\infty(f;S)\le \sqrt d\,|Df|(S).
$$
Thus $\TV_\infty$ is not a different regularity assumption; it is the
coordinate-adapted form naturally paired with slice primitives.

The second density quantity measures regularity in a different way. Rather than
summing first-order coordinate derivative masses, Vitali variation captures the
mixed, alternating-vertex content of a function over partitions of $S$ into
axis-parallel boxes. For a pointwise-defined function $f:S\to\mathbb R$, let
$\mathfrak P(S)$ be the collection of finite families
$\mathcal Q=\{S_j\}_{j=1}^N$, where
$$
S_j=[a^j,b^j]
:=
\prod_{\ell=1}^d[a_\ell^j,b_\ell^j],
$$
the interiors of the boxes $S_j$ are pairwise disjoint, and
$$
S=\bigcup_{j=1}^N S_j.
$$
The Vitali variation is
\begin{equation}
\label{eq:def_vitali_variation}
V(f;S)
:=
\sup_{\mathcal Q\in\mathfrak P(S)}
\sum_{S_j\in\mathcal Q}|\sigma(f,S_j)|,
\end{equation}
where the quasivolume is the alternating vertex sum
\begin{equation}
\label{eq:def_quasivolume}
\sigma(f,S_j)
:=
\sum_{\nu\in\{0,1\}^d}
(-1)^{\nu_1+\cdots+\nu_d}
f\bigl(b^j+\nu\odot(a^j-b^j)\bigr).
\end{equation}
If $f\in C^d(\overline S)$ and
$\partial_1\cdots\partial_d f\in L^1(S)$, then
\begin{equation}
\label{eq:vitali_smooth_identity}
V(f;S)=\int_S|\partial_1\cdots\partial_d f(\omega)|\,d\omega .
\end{equation}

\begin{remark}[Scope of the Vitali definition]
\label{rem:vitali_scope}
The corner-sum definition of Vitali variation depends on the chosen pointwise
representative and is therefore not an $L^1$ notion. In the main text we use $V(f;S)$ only in the
smooth setting, where \eqref{eq:vitali_smooth_identity} applies.
\end{remark}

For the nonsmooth extension of the mixed-derivative formulation, we use
distributional mixed derivatives. If $U\subset\mathbb R^d$ is open, we write
$\mathcal M(U)$ for the finite signed Radon measures on $U$.

\begin{definition}[Distributional mixed derivative]
\label{def:distributional_mixed_derivative_main}
Let $U\subset\mathbb R^d$ be open, let
$I=\{i_1<\cdots<i_k\}\subset[d]$ be nonempty, and let
$h\in L^1_{\mathrm{loc}}(U)$. The distribution $D_Ih$ is defined by
$$
\langle D_Ih,\zeta\rangle
:=
(-1)^k\int_U h(x)\partial_I\zeta(x)\,dx,
\qquad
\zeta\in C_c^\infty(U).
$$
If this distribution is represented by a finite signed Radon measure, we use
the same notation $D_Ih\in\mathcal M(U)$ and denote its total variation by
$|D_Ih|(U)$. If $h$ is smooth, then
$$
D_Ih=\partial_Ih(x)\,dx.
$$
\end{definition}

If $d=1$ and $f\in C^1([a,b])$, then
$$
V(f;[a,b])
=
\int_a^b|f'(t)|\,dt
=
|Df|((a,b))
=
\TV_\infty(f;[a,b]).
$$
For a general $BV$ equivalence class, the distributional quantities are
representative-independent, whereas classical Jordan or Vitali variation is
a pointwise notion and may change after modifying a representative on a null
set. For $d\ge2$, they measure
different regularity. The anisotropic variation is first-order and additive in
the coordinate directions, while the smooth Vitali variation measures the
top-order mixed derivative. The resulting residual certificates also use
different envelopes, $\Lambda_{\mathrm{sl}}^I$ and
$\Lambda_{\mathrm{mix}}^I$, and neither certificate uniformly dominates the
other. Explicit product-density families exhibiting both possible orderings
are given in Appendix~\ref{subsec:appendix_branch_separation}.

\subsection{Exact centered anisotropic \texorpdfstring{$\TV_\infty$}{TV-infinity} bound } 
\label{subsec:exact_centered_tv_bounds}

The TV estimate uses a one-dimensional primitive bound to construct a vector
field with divergence equal to the centered residual and with zero normal trace
on the boundary.

\begin{lemma}[Half-oscillation inequality]
\label{lem:half_oscillation}
Let $g\in L^1([a,b])$ satisfy $\int_a^b g(t)\,dt=0$, and define
$G(t):=\int_a^t g(s)\,ds$. Then
\begin{equation}
\label{eq:half_oscillation}
\|G\|_{L^\infty([a,b])}
\le
\frac12\int_a^b |g(s)|\,ds.
\end{equation}
Moreover,
\begin{equation}
\label{eq:recentered_half_oscillation}
\inf_{c\in\mathbb R}\|G-c\|_{L^\infty([a,b])}
\le
\frac14\int_a^b |g(s)|\,ds.
\end{equation}
\end{lemma}

\begin{proof}
Write $g=g_+-g_-$, where $g_+=\max\{g,0\}$ and
$g_-=\max\{-g,0\}$. Since $g$ has zero mean,
$\int_a^b g_+=\int_a^b g_-=:A$, and therefore
$\int_a^b |g|=2A$. For every $t\in[a,b]$,
$$
G(t)=\int_a^t g_+(s)\,ds-\int_a^t g_-(s)\,ds,
$$
so $-A\le G(t)\le A$. This proves \eqref{eq:half_oscillation}.

For the recentered estimate, let $M:=\sup_{[a,b]}G$ and
$m:=\inf_{[a,b]}G$. Since $G(a)=G(b)=0$ and $G'=g$ a.e.,
$\TV(G;[a,b])=\int_a^b |g(s)|\,ds$. Any excursion from $m$ to $M$ must be
matched by a return excursion because $G$ starts and ends at the same value.
Hence $\TV(G;[a,b])\ge2(M-m)$. Therefore
$\operatorname{osc}_{[a,b]}G\le \frac12\int_a^b |g(s)|\,ds$. The best
constant approximation to a bounded function in $L^\infty$ has error equal to
one half of its oscillation, and the claim follows.
\end{proof}
The primitive constructed from the centered residual will serve as the vector
field in a Gauss--Green integration by parts. We use the following form of that
identity on the box.
\begin{lemma}[Gauss--Green identity on a box]
\label{lem:gg_bv_box}
Let $f\in BV(S)$ and let $\Psi\in\mathrm{Lip}(\overline S;\mathbb R^d)$. Then
\begin{equation}
\label{eq:gg_bv_box}
\int_S f\,\operatorname{div}\Psi\,d\omega
=
\int_{\partial S}\operatorname{Tr}(f)(\Psi\cdot n)\,d\mathcal H^{d-1}
-
\int_S \Psi\cdot d(Df).
\end{equation}
\end{lemma}

\begin{proof}
Since $S$ is a bounded Lipschitz domain, $BV(S)$ admits a trace on
$\partial S$, and the Gauss--Green formula holds for $BV$ functions paired
with Lipschitz vector fields; see, for example,
\cite[Chapter~9]{giovanni_leoni_sobolev}. Equation
\eqref{eq:gg_bv_box} is the specialization to the box $S$.
\end{proof}

\begin{theorem}[$k$-direction anisotropic $\TV_\infty$ bound]
\label{thm:k_direction_tv_bound}
Let $I\subseteq[d]$ be nonempty and set $k:=|I|$. Let $f\in BV(S)$ be a
probability density on $S$, and let $\varphi\in W^{1,\infty}(S)$ be centered on
$I$. Then
\begin{equation}
\label{eq:k_direction_tv_bound}
|\mathbb E_f[\varphi]|
\le
\frac{1}{2k}
\Lambda_{\mathrm{sl}}^I(\varphi)
\TV_I(f;S).
\end{equation}
Consequently,
\begin{equation}
\label{eq:k_direction_tv_bound_tvinfty}
|\mathbb E_f[\varphi]|
\le
\frac{1}{2k}
\Lambda_{\mathrm{sl}}^I(\varphi)
\TV_\infty(f;S).
\end{equation}
If $\|\varphi\|_{L^\infty(S)}\le h$, then
\begin{equation}
\label{eq:k_direction_tv_bound_sup}
|\mathbb E_f[\varphi]|
\le
\frac{h}{2k}
\left(\max_{i\in I}(b_i-a_i)\right)
\TV_I(f;S).
\end{equation}
\end{theorem}
\begin{proof}
We subtract the uniform density to remove the constant part of $f$. This does
not change the variation inside $S$, because $u$ has no distributional
derivative in the interior of the box. Let $g:=f-u$. Since $u$ is constant on
$S$, its distributional derivatives inside $S$ vanish. Hence
$$
D_i g=D_i f
\qquad\text{as measures on }S,
$$
for each coordinate $i$, and therefore
$$
\TV_I(g;S)=\TV_I(f;S).
$$

Because $I$ is nonempty and $\varphi$ is centered on $I$, choosing any
$i_0\in I$ and applying Fubini's theorem gives
$$
\int_S\varphi(\omega)\,d\omega=0.
$$
Thus $\mathbb E_u[\varphi]=0$, and hence
$$
\mathbb E_f[\varphi]
=
\mathbb E_f[\varphi]-\mathbb E_u[\varphi]
=
\int_S (f-u)\varphi\,d\omega
=
\int_S g\varphi\,d\omega.
$$

Define a vector field $\Psi:\overline S\to\mathbb R^d$ by
$$
\Psi_i(\omega)
:=
\frac1k
\int_{a_i}^{\omega_i}
\varphi(\omega_1,\ldots,\omega_{i-1},t,\omega_{i+1},\ldots,\omega_d)\,dt,
\qquad i\in I,
$$
and set $\Psi_i(\omega):=0$ for $i\notin I$. Since
$\varphi\in W^{1,\infty}(S)$, we use its Lipschitz representative on
$\overline S$. Then $\Psi\in\mathrm{Lip}(\overline S;\mathbb R^d)$. For
$i\in I$,
$$
\partial_i\Psi_i(\omega)=\frac1k\varphi(\omega)
\qquad\text{for a.e. }\omega\in S,
$$
and therefore
$$
\operatorname{div}\Psi
=
\sum_{i\in I}\partial_i\Psi_i
=
\varphi
\qquad\text{a.e. on }S.
$$

The key point is that the normal trace of $\Psi$ vanishes, and this is
exactly where the centering assumption is used. If $i\in I$, then $\Psi_i=0$ on the lower
face $\{\omega_i=a_i\}$. On the upper face $\{\omega_i=b_i\}$,
$$
\Psi_i(\omega_1,\ldots,b_i,\ldots,\omega_d)
=
\frac1k
\int_{a_i}^{b_i}
\varphi(\omega_1,\ldots,t,\ldots,\omega_d)\,dt,
$$
which is zero for a.e. $\omega_{-i}$ because $\varphi$ is centered in direction
$i$. If $j\notin I$, then $\Psi_j$ is identically zero. Since the outward
normal on the two faces orthogonal to $e_j$ is $\pm e_j$, it follows that
$$
\Psi\cdot n=0
\qquad
\mathcal H^{d-1}\text{-a.e. on }\partial S.
$$

It remains to estimate the size of $\Psi$. Fix $i\in I$ and a slice
$\omega_{-i}\in S_{-i}$ for which the centering identity holds. Set
$$
h_{\omega_{-i}}(t)
:=
\varphi(\omega_1,\ldots,\omega_{i-1},t,\omega_{i+1},\ldots,\omega_d).
$$
Then
$$
\int_{a_i}^{b_i}h_{\omega_{-i}}(t)\,dt=0.
$$
By Lemma~\ref{lem:half_oscillation},
$$
\sup_{s\in[a_i,b_i]}
\left|
\int_{a_i}^{s}h_{\omega_{-i}}(t)\,dt
\right|
\le
\frac12\int_{a_i}^{b_i}|h_{\omega_{-i}}(t)|\,dt.
$$
Thus
$$
\|\Psi_i\|_{L^\infty(S)}
\le
\frac{1}{2k}
\operatorname*{ess\,sup}_{\omega_{-i}\in S_{-i}}
\int_{a_i}^{b_i}
|\varphi(\omega_1,\ldots,\omega_{i-1},t,\omega_{i+1},\ldots,\omega_d)|\,dt.
$$
The factor $1/(2k)$ comes from averaging over the $k$ selected coordinate
directions and applying the half-oscillation estimate on each slice.
Taking the maximum over $i\in I$ gives
\begin{equation}
\label{eq:vector_field_sup_bound_tv}
\|\Psi\|_{L^\infty(S;\ell_\infty)}
\le
\frac{1}{2k}
\Lambda_{\mathrm{sl}}^I(\varphi).
\end{equation}

Since $g\in BV(S)$ and $\Psi\in\mathrm{Lip}(\overline S;\mathbb R^d)$, the
Gauss--Green identity yields
$$
\int_S g\varphi\,d\omega
=
\int_S g\,\operatorname{div}\Psi\,d\omega
=
\int_{\partial S}\operatorname{Tr}(g)(\Psi\cdot n)\,d\mathcal H^{d-1}
-
\int_S\Psi\cdot d(Dg).
$$
The boundary term vanishes. Since only coordinates in $I$ contribute,
$$
\left|\int_S g\varphi\,d\omega\right|
\le
\sum_{i\in I}\|\Psi_i\|_{L^\infty(S)}\,|D_i g|(S).
$$
Using
$$
\|\Psi_i\|_{L^\infty(S)}
\le
\|\Psi\|_{L^\infty(S;\ell_\infty)}
$$
and \eqref{eq:vector_field_sup_bound_tv}, we obtain
$$
|\mathbb E_f[\varphi]|
=
\left|\int_S g\varphi\,d\omega\right|
\le
\frac{1}{2k}
\Lambda_{\mathrm{sl}}^I(\varphi)
\TV_I(g;S)
=
\frac{1}{2k}
\Lambda_{\mathrm{sl}}^I(\varphi)
\TV_I(f;S).
$$
This proves \eqref{eq:k_direction_tv_bound}. Since
$\TV_I(f;S)\le \TV_\infty(f;S)$, \eqref{eq:k_direction_tv_bound_tvinfty}
follows.

Finally, if $\|\varphi\|_{L^\infty(S)}\le h$, then for every $i\in I$,
$$
\operatorname*{ess\,sup}_{\omega_{-i}\in S_{-i}}
\int_{a_i}^{b_i}
|\varphi(\omega_1,\ldots,\omega_{i-1},t,\omega_{i+1},\ldots,\omega_d)|\,dt
\le
(b_i-a_i)h.
$$
Therefore
$$
\Lambda_{\mathrm{sl}}^I(\varphi)
\le
\left(\max_{i\in I}(b_i-a_i)\right)h,
$$
and substitution into \eqref{eq:k_direction_tv_bound} gives
\eqref{eq:k_direction_tv_bound_sup}.
\end{proof}

\begin{corollary}[Bounded-residual anisotropic extension]
\label{cor:bounded_residual_tv_extension}
Let $I\subseteq[d]$ be nonempty and set $k:=|I|$. Let
$f\in BV(S)$ be a probability density, and let
$\varphi\in L^\infty(S)$ be centered on $I$. Then \eqref{eq:k_direction_tv_bound} and
\eqref{eq:k_direction_tv_bound_tvinfty} hold. If, in addition,
$\|\varphi\|_{L^\infty(S)}\le h$, then
\eqref{eq:k_direction_tv_bound_sup} also holds.
\end{corollary}

\begin{proof}
Fix $i\in I$. For a.e. $\omega_{-i}\in S_{-i}$, define
\[
F_i(\omega_{-i},t)
:=
\int_{a_i}^{t}\varphi(\omega_{-i},s)\,ds,
\qquad
t\in[a_i,b_i].
\]
Then $F_i(\omega_{-i},\cdot)$ is absolutely continuous,
\[
\partial_tF_i(\omega_{-i},t)
=
\varphi(\omega_{-i},t)
\qquad
\text{for a.e. }t,
\]
and coordinate-$i$ centering gives
\[
F_i(\omega_{-i},a_i)
=
F_i(\omega_{-i},b_i)
=
0.
\]

By the slicing theorem for $BV$ functions
\cite{AmbrosioFuscoPallara2000,giovanni_leoni_sobolev}, for a.e.
$\omega_{-i}\in S_{-i}$ the slice
\[
f_{\omega_{-i}}(t):=f(\omega_{-i},t)
\]
belongs to $BV(a_i,b_i)$, and
\[
\int_{S_{-i}}
|Df_{\omega_{-i}}|((a_i,b_i))\,d\omega_{-i}
=
|D_i f|(S).
\]
Applying the one-dimensional $BV$ integration-by-parts formula on each
slice and using the zero endpoint values of $F_i$ gives
\[
\int_S f(\omega)\varphi(\omega)\,d\omega
=
-
\int_{S_{-i}}
\int_{(a_i,b_i)}
F_i(\omega_{-i},t)\,
dDf_{\omega_{-i}}(t)\,
d\omega_{-i}.
\]
Averaging this identity over $i\in I$ and taking absolute values yields
\[
|\mathbb E_f[\varphi]|
\le
\frac1k
\sum_{i\in I}
\|F_i\|_{L^\infty(S)}\,|D_i f|(S).
\]
Lemma~\ref{lem:half_oscillation} gives
\[
\|F_i\|_{L^\infty(S)}
\le
\frac12
\operatorname*{ess\,sup}_{\omega_{-i}\in S_{-i}}
\int_{a_i}^{b_i}
|\varphi(\omega_{-i},t)|\,dt
\le
\frac12\Lambda_{\mathrm{sl}}^I(\varphi).
\]
Therefore
\[
|\mathbb E_f[\varphi]|
\le
\frac{1}{2k}
\Lambda_{\mathrm{sl}}^I(\varphi)
\sum_{i\in I}|D_i f|(S),
\]
which proves \eqref{eq:k_direction_tv_bound}. The conclusions
\eqref{eq:k_direction_tv_bound_tvinfty} and
\eqref{eq:k_direction_tv_bound_sup} follow exactly as in
Theorem~\ref{thm:k_direction_tv_bound}.
\end{proof}

\begin{corollary}[One-direction primitive bound]
\label{cor:one_direction_primitive_bound}
Let $f\in BV(S)$ be a probability density, let
$\varphi\in L^\infty(S)$, and fix $i\in[d]$. If $\varphi$ is centered in
direction $i$, then
\begin{equation}
\label{eq:one_direction_primitive_bound}
|\mathbb E_f[\varphi]|
\le
\Lambda_{\mathrm{pr}}^i(\varphi)\,|D_i f|(S).
\end{equation}
\end{corollary}

\begin{proof}
For a.e. $\omega_{-i}\in S_{-i}$, define
\[
F_i(\omega_{-i},t)
:=
\int_{a_i}^{t}\varphi(\omega_{-i},s)\,ds.
\]
Then $F_i$ is absolutely continuous in $t$, its derivative is
$\varphi$, and centering gives
\[
F_i(\omega_{-i},a_i)
=
F_i(\omega_{-i},b_i)
=
0.
\]
Applying one-dimensional $BV$ integration by parts on coordinate-$i$
slices and then using the slicing theorem gives
\[
|\mathbb E_f[\varphi]|
\le
\|F_i\|_{L^\infty(S)}\,|D_i f|(S).
\]
By definition,
\[
\|F_i\|_{L^\infty(S)}
=
\Lambda_{\mathrm{pr}}^i(\varphi),
\]
which proves the result.
\end{proof}
We pause to note that this is an intrinsically signed mechanism. If a
nonnegative residual averages to zero on almost every coordinate slice, then it
must vanish almost everywhere. The estimates below therefore do not control
nonnegative approximation errors by cancellation; they control signed residuals
whose overestimation and underestimation regions balance on slices.
\begin{corollary}[All-axes anisotropic $\TV_\infty$ bound]
\label{cor:all_axes_tv_bound}
Let $f\in BV(S)$ be a probability density, and let
$\varphi\in L^\infty(S)$ be centered in every coordinate direction. Then
\begin{equation}
\label{eq:all_axes_tv_bound}
|\mathbb E_f[\varphi]|
\le
\frac{1}{2d}\Lambda_{\mathrm{sl}}(\varphi)\,\TV_\infty(f;S).
\end{equation}
If $\|\varphi\|_{L^\infty(S)}\le h$, then
\begin{equation}
\label{eq:all_axes_tv_bound_sup}
|\mathbb E_f[\varphi]|
\le
\frac{h}{2d}
\left(\max_{i\in[d]}(b_i-a_i)\right)
\TV_\infty(f;S).
\end{equation}
\end{corollary}

\begin{proof}
Apply Corollary~\ref{cor:bounded_residual_tv_extension} with $I=[d]$.
Then $k=d$, $\TV_I(f;S)=\TV_\infty(f;S)$, and
$\Lambda_{\mathrm{sl}}^I(\varphi)=\Lambda_{\mathrm{sl}}(\varphi)$.
\end{proof}

\begin{remark}[Centered residuals must be signed]
\label{rem:positivity_centering_triviality}
If $\varphi\ge0$ a.e. on $S$ and $\varphi$ is centered in one coordinate
direction, then $\varphi=0$ a.e. on $S$. Thus the cancellation mechanism in
the centered bounds is inherently a signed-residual mechanism. It is not a
bound for nonnegative approximation errors; it is a bound for residuals whose
positive and negative parts cancel on coordinate slices.
\end{remark}


\subsection{Exact centered mixed-derivative and Vitali bounds } 
\label{subsec:exact_centered_mixed_derivative_vitali_bounds}

We next prove the mixed-derivative analogue of the anisotropic $\TV_\infty$
bound. The argument is again an integration-by-parts argument, but now the
primitive is taken successively in the selected coordinates. Throughout this
subsection, for a nonempty $I\subseteq[d]$ with ordered elements
$i_1<\cdots<i_k$, we write $\partial_I f:=\partial_{i_1}\cdots\partial_{i_k}f$.
In the all-axes case $I=[d]$, this is the top-order mixed derivative, and in
the smooth setting the Vitali identity gives
$$
V(f;S)=\int_S |\partial_{[d]} f(\omega)|\,d\omega .
$$
Thus the theorem below is a partial mixed-derivative bound, while its all-axes
case gives the Vitali-variation bound. 

\begin{theorem}[$k$-direction mixed-derivative residual bound]
\label{thm:k_direction_mixed_derivative_bound}
Let $\emptyset\neq I\subseteq[d]$, let $k=|I|$, let
$f\in C^k(\overline S)$, and let $\varphi\in L^\infty(S)$ be centered on $I$.
Define the upper-anchored $I$-orthant primitive
\begin{equation}
\label{eq:def_UI_primitive}
(U_I\varphi)(\omega)
:=
\int_{\omega_{i_1}}^{b_{i_1}}
\cdots
\int_{\omega_{i_k}}^{b_{i_k}}
\varphi(\omega_{-I},t_I)\,
dt_{i_k}\cdots dt_{i_1}.
\end{equation}
Then $U_I\varphi\in L^\infty(S)$ and
\begin{equation}
\label{eq:UI_primitive_sup_bound}
\|U_I\varphi\|_{L^\infty(S)}
\le
2^{-k}\Lambda_{\mathrm{mix}}^I(\varphi).
\end{equation}
Moreover,
\begin{equation}
\label{eq:k_direction_mixed_derivative_identity}
\int_S\varphi(\omega)f(\omega)\,d\omega
=
\int_S (U_I\varphi)(\omega)\,\partial_I f(\omega)\,d\omega .
\end{equation}
Consequently,
\begin{equation}
\label{eq:k_direction_mixed_derivative_bound}
\left|
\int_S\varphi(\omega)f(\omega)\,d\omega
\right|
\le
2^{-k}\Lambda_{\mathrm{mix}}^I(\varphi)
\int_S|\partial_I f(\omega)|\,d\omega .
\end{equation}
No sign or normalization assumption is imposed on $f$.
\end{theorem}

\begin{proof}
Since $\varphi\in L^\infty(S)$ and $S$ is bounded, $\varphi\in L^1(S)$.
Therefore the primitive $U_I\varphi$ in \eqref{eq:def_UI_primitive} is well
defined for a.e. $\omega\in S$.

We first prove the $L^\infty$ estimate
\eqref{eq:UI_primitive_sup_bound}. Fix $\omega\in S$. The variables
$\omega_{-I}$ are fixed. Define
$$
H_0(t_{i_1},\ldots,t_{i_k})
:=
\varphi(\omega_{-I},t_{i_1},\ldots,t_{i_k})
\qquad\text{on }S_I.
$$
For $r=1,\ldots,k$, define recursively
$$
H_r(t_{i_{r+1}},\ldots,t_{i_k})
:=
\int_{\omega_{i_r}}^{b_{i_r}}
H_{r-1}(s,t_{i_{r+1}},\ldots,t_{i_k})\,ds,
$$
with the convention that $H_k$ is a scalar. Then
$$
H_k=(U_I\varphi)(\omega).
$$

We claim that, for each $r=1,\ldots,k$ and for a.e. value of the remaining
variables,
$$
\int_{a_{i_r}}^{b_{i_r}}
H_{r-1}(s,t_{i_{r+1}},\ldots,t_{i_k})\,ds=0.
$$
For $r=1$, this is exactly the centering of $\varphi$ in direction $i_1$, with
$\omega_{-I}$ and $t_{i_2},\ldots,t_{i_k}$ fixed. For general $r$,
$H_{r-1}$ is obtained from $\varphi$ by integrating over the previous
variables $i_1,\ldots,i_{r-1}$ with fixed lower limits
$\omega_{i_1},\ldots,\omega_{i_{r-1}}$. Hence, by Fubini's theorem,
$$
\begin{aligned}
&\int_{a_{i_r}}^{b_{i_r}}
H_{r-1}(s,t_{i_{r+1}},\ldots,t_{i_k})\,ds
\\
&\quad =
\int_{\omega_{i_1}}^{b_{i_1}}
\cdots
\int_{\omega_{i_{r-1}}}^{b_{i_{r-1}}}
\left(
\int_{a_{i_r}}^{b_{i_r}}
\varphi(\omega_{-I},t_{i_1},\ldots,t_{i_{r-1}},s,
t_{i_{r+1}},\ldots,t_{i_k})\,ds
\right)
dt_{i_{r-1}}\cdots dt_{i_1}.
\end{aligned}
$$
The inner integral is zero for a.e. value of the remaining variables because
$\varphi$ is centered in direction $i_r$. This proves the claim.

Since the full integral of $H_{r-1}$ in coordinate $i_r$ is zero, its
upper-anchored primitive is the negative of its lower-anchored primitive:
$$
\int_{\omega_{i_r}}^{b_{i_r}}
H_{r-1}(s,t_{i_{r+1}},\ldots,t_{i_k})\,ds
=
-
\int_{a_{i_r}}^{\omega_{i_r}}
H_{r-1}(s,t_{i_{r+1}},\ldots,t_{i_k})\,ds.
$$
We may therefore apply Lemma~\ref{lem:half_oscillation} to
$H_{r-1}$ as a function of the single variable $t_{i_r}$, with the remaining
variables fixed. This gives
$$
|H_r(t_{i_{r+1}},\ldots,t_{i_k})|
\le
\frac12
\int_{a_{i_r}}^{b_{i_r}}
|H_{r-1}(s,t_{i_{r+1}},\ldots,t_{i_k})|\,ds
$$
for a.e. $(t_{i_{r+1}},\ldots,t_{i_k})$. Iterating this estimate for
$r=1,\ldots,k$ yields
$$
|H_k|
\le
2^{-k}
\int_{S_I}|H_0(t_I)|\,dt_I.
$$
Since $H_k=(U_I\varphi)(\omega)$ and
$$
\int_{S_I}|H_0(t_I)|\,dt_I
=
\int_{S_I}|\varphi(\omega_{-I},t_I)|\,dt_I,
$$
we obtain
$$
|(U_I\varphi)(\omega)|
\le
2^{-k}
\int_{S_I}|\varphi(\omega_{-I},t_I)|\,dt_I.
$$
Taking the essential supremum over $\omega_{-I}\in S_{-I}$ gives
$$
\|U_I\varphi\|_{L^\infty(S)}
\le
2^{-k}
\operatorname*{ess\,sup}_{\omega_{-I}\in S_{-I}}
\int_{S_I}|\varphi(\omega_{-I},t_I)|\,dt_I
=
2^{-k}\Lambda_{\mathrm{mix}}^I(\varphi).
$$
This proves \eqref{eq:UI_primitive_sup_bound}.

We now prove the integration-by-parts identity
\eqref{eq:k_direction_mixed_derivative_identity}. For $r=0,\ldots,k$, define
the partial upper-anchored primitives by
$$
U_I^{(0)}\varphi:=\varphi,
$$
and, for $r=1,\ldots,k$,
$$
(U_I^{(r)}\varphi)(\omega)
:=
\int_{\omega_{i_1}}^{b_{i_1}}
\cdots
\int_{\omega_{i_r}}^{b_{i_r}}
\varphi(\omega_{-I},t_{i_1},\ldots,t_{i_r},
\omega_{i_{r+1}},\ldots,\omega_{i_k})\,
dt_{i_r}\cdots dt_{i_1}.
$$
Thus $U_I^{(k)}\varphi=U_I\varphi$.

For each $r=1,\ldots,k$, the map
$\omega_{i_r}\mapsto (U_I^{(r)}\varphi)(\omega)$ is absolutely continuous for
a.e. value of the remaining variables, and differentiation with respect to the
lower integration limit gives
\begin{equation}
\label{eq:partial_primitive_derivative_mixed}
\partial_{i_r}(U_I^{(r)}\varphi)
=
-\,U_I^{(r-1)}\varphi
\qquad\text{a.e. on }S.
\end{equation}

We also record the boundary traces needed for integration by parts. On the
upper face $\{\omega_{i_r}=b_{i_r}\}$, the integral in the $i_r$ variable has
zero length, and therefore
$$
U_I^{(r)}\varphi=0
\qquad\text{on }\{\omega_{i_r}=b_{i_r}\}.
$$
On the lower face $\{\omega_{i_r}=a_{i_r}\}$, the integral over the $i_r$
variable becomes an integral over the full interval $[a_{i_r},b_{i_r}]$.
More explicitly,
$$
\begin{aligned}
&(U_I^{(r)}\varphi)(\omega_1,\ldots,a_{i_r},\ldots,\omega_d)
\\
&\quad =
\int_{\omega_{i_1}}^{b_{i_1}}
\cdots
\int_{\omega_{i_{r-1}}}^{b_{i_{r-1}}}
\left(
\int_{a_{i_r}}^{b_{i_r}}
\varphi(\omega_{-I},t_{i_1},\ldots,t_{i_{r-1}},s,
\omega_{i_{r+1}},\ldots,\omega_{i_k})\,ds
\right)
dt_{i_{r-1}}\cdots dt_{i_1}.
\end{aligned}
$$
The inner integral is zero for a.e. value of the remaining variables because
$\varphi$ is centered in direction $i_r$. Hence
\begin{equation}
\label{eq:partial_primitive_boundary_zero}
U_I^{(r)}\varphi=0
\qquad\text{for a.e. point on }\{\omega_{i_r}=a_{i_r}\}.
\end{equation}
Thus $U_I^{(r)}\varphi$ has zero trace on both faces orthogonal to direction
$i_r$.

The identity is proved by induction, integrating by parts in one selected
coordinate at a time. At each step, the boundary term vanishes: on the upper
face by construction of the upper-anchored primitive, and on the lower face by
the slice-centering assumption. So for every $r=0,\ldots,k$,
\begin{equation}
\label{eq:iterated_mixed_identity}
\int_S\varphi(\omega)f(\omega)\,d\omega
=
\int_S
(U_I^{(r)}\varphi)(\omega)\,
\partial_{i_1}\cdots\partial_{i_r}f(\omega)\,d\omega,
\end{equation}
with the convention that $\partial_{i_1}\cdots\partial_{i_0}f:=f$. For
$r=0$, \eqref{eq:iterated_mixed_identity} is immediate.

Assume that \eqref{eq:iterated_mixed_identity} holds for $r-1$, where
$1\le r\le k$. Then
$$
\int_S\varphi f
=
\int_S
(U_I^{(r-1)}\varphi)
\partial_{i_1}\cdots\partial_{i_{r-1}}f.
$$
Using \eqref{eq:partial_primitive_derivative_mixed}, this becomes
$$
\int_S\varphi f
=
-\int_S
\partial_{i_r}(U_I^{(r)}\varphi)
\partial_{i_1}\cdots\partial_{i_{r-1}}f.
$$
Since $f\in C^k(\overline S)$ and $U_I^{(r)}\varphi$ is absolutely continuous
in direction $i_r$ for a.e. value of the other variables, Fubini's theorem and
one-dimensional integration by parts in the variable $\omega_{i_r}$ give
$$
\begin{aligned}
\int_S\varphi f
&=
-\int_{S_{-i_r}}
\left[
(U_I^{(r)}\varphi)
\partial_{i_1}\cdots\partial_{i_{r-1}}f
\right]_{\omega_{i_r}=a_{i_r}}^{b_{i_r}}
\,d\omega_{-i_r}
\\
&\quad
+
\int_S
(U_I^{(r)}\varphi)
\partial_{i_r}
\bigl(
\partial_{i_1}\cdots\partial_{i_{r-1}}f
\bigr)
\,d\omega.
\end{aligned}
$$
The boundary term vanishes by the two face identities above. Since $f$ is
$C^k$, the mixed partial derivatives commute, and therefore
$$
\partial_{i_r}
\bigl(
\partial_{i_1}\cdots\partial_{i_{r-1}}f
\bigr)
=
\partial_{i_1}\cdots\partial_{i_r}f.
$$
Thus
$$
\int_S\varphi f
=
\int_S
(U_I^{(r)}\varphi)
\partial_{i_1}\cdots\partial_{i_r}f.
$$
This proves the induction step. Taking $r=k$ gives
$$
\int_S\varphi(\omega)f(\omega)\,d\omega
=
\int_S
(U_I\varphi)(\omega)\partial_I f(\omega)\,d\omega,
$$
which proves \eqref{eq:k_direction_mixed_derivative_identity}.

Finally, Hölder's inequality and \eqref{eq:UI_primitive_sup_bound} imply
$$
\begin{aligned}
\left|
\int_S\varphi(\omega)f(\omega)\,d\omega
\right|
&=
\left|
\int_S
(U_I\varphi)(\omega)\partial_I f(\omega)\,d\omega
\right|
\\
&\le
\|U_I\varphi\|_{L^\infty(S)}
\int_S|\partial_I f(\omega)|\,d\omega
\\
&\le
2^{-k}\Lambda_{\mathrm{mix}}^I(\varphi)
\int_S|\partial_I f(\omega)|\,d\omega.
\end{aligned}
$$
This proves \eqref{eq:k_direction_mixed_derivative_bound}.
\end{proof}

\begin{remark}[Analytic scope of the mixed-derivative identity]
\label{rem:analytic_scope_mixed_derivative}
Theorem~\ref{thm:k_direction_mixed_derivative_bound} is an analytic statement.
Its proof does not use nonnegativity or normalization of $f$. Thus the same
identity and estimate hold for any $g\in C^k(\overline S)$ after replacing
$f$ by $g$. This point is used in the mollification argument: a globally
mollified extension is smooth on $S$, but its restriction to $S$ need not be a
probability density.
\end{remark}

\begin{corollary}[All-axes mixed-derivative and Vitali bound]
\label{cor:all_axes_vitali_bound}
Let $f\in C^d(\overline S)$ and let $\varphi\in L^\infty(S)$ be centered on
$[d]$. Then $U_{[d]}\varphi\in L^\infty(S)$ and
$$
\|U_{[d]}\varphi\|_{L^\infty(S)}
\le
2^{-d}\|\varphi\|_{L^1(S)}.
$$
Moreover,
$$
\int_S\varphi(\omega)f(\omega)\,d\omega
=
\int_S
(U_{[d]}\varphi)(\omega)\,\partial_{[d]}f(\omega)\,d\omega,
$$
and therefore
$$
\left|
\int_S\varphi(\omega)f(\omega)\,d\omega
\right|
\le
2^{-d}\|\varphi\|_{L^1(S)}
\int_S|\partial_{[d]}f(\omega)|\,d\omega.
$$
In the smooth Vitali setting, where
$$
V(f;S)=\int_S|\partial_{[d]}f(\omega)|\,d\omega,
$$
this gives
\begin{equation}
\label{eq:all_axes_vitali_bound}
\left|
\int_S\varphi(\omega)f(\omega)\,d\omega
\right|
\le
2^{-d}\|\varphi\|_{L^1(S)}V(f;S).
\end{equation}
\end{corollary}

\begin{proof}
Apply Theorem~\ref{thm:k_direction_mixed_derivative_bound} with $I=[d]$.
Then $|I|=d$, $S_I=S$, and $S_{-I}$ is a singleton, so
$$
\Lambda_{\mathrm{mix}}^{[d]}(\varphi)
=
\int_S|\varphi(\omega)|\,d\omega
=
\|\varphi\|_{L^1(S)}.
$$
The primitive in the theorem is $U_{[d]}\varphi$, and the displayed estimates
follow directly from the primitive bound, the integration-by-parts identity,
and the mixed-derivative estimate in Theorem~\ref{thm:k_direction_mixed_derivative_bound}.
The final inequality follows from the smooth Vitali identity.
\end{proof}

\begin{remark}[Partial mixed derivatives and Vitali variation]
\label{rem:partial_mixed_and_vitali_bound}
Theorem~\ref{thm:k_direction_mixed_derivative_bound} is stated for an arbitrary
nonempty $I\subseteq[d]$ because the same integration-by-parts argument gives
partial mixed-derivative bounds. The classical Vitali variation appears only
in the all-axes case $I=[d]$, where $\partial_I f=\partial_{[d]}f$.
\end{remark}

\subsection{Defect-adjusted projector bounds } 
\label{subsec:defect_adjusted_projector_bounds}

The preceding estimates are exact-centered estimates. They quantify the
cancellation obtained when the residual has zero average on the selected
coordinate slices. For convex approximation, this cannot generally be imposed: convexity
constrains $\widetilde v$, not the difference $\widetilde v-v^{\INT}$. A
convex approximation induces the residual
$\varphi=\widetilde v-v^{\INT}$, and there is no reason for this residual to
be exactly centered.

The projector decomposition gives the corresponding general statement. For
every nonempty $I\subseteq[d]$,
$$
\varphi=\mathcal P_I\varphi+\mathcal R_I\varphi,
$$
where $\mathcal P_I\varphi$ is centered on $I$ and $\mathcal R_I\varphi$ is the
slice-mean defect. The exact-centered bounds apply to the first component; the
second component is paid for directly.

\begin{theorem}[Defect-adjusted projector bounds]
\label{thm:defect_adjusted_projector_bounds}
Fix a nonempty $I\subseteq[d]$ and write $k=|I|$.

\emph{(i) Anisotropic $\TV_\infty$ bound } 
If $f\in BV(S)$ is a probability density and
$\varphi\in L^\infty(S)$, then
\begin{equation}
\label{eq:defect_adjusted_tv_bound}
|\mathbb E_f[\varphi]|
\le
\frac{1}{2k}
\Lambda_{\mathrm{sl}}^I(\mathcal P_I\varphi)
\TV_I(f;S)
+
\|\mathcal R_I\varphi\|_{L^\infty(S)}.
\end{equation}

\emph{(ii) Smooth mixed-derivative bound } 
If $f\in C^k(\overline S)$ is a probability density and
$\varphi\in L^\infty(S)$, then
\begin{equation}
\label{eq:defect_adjusted_mixed_derivative_bound}
|\mathbb E_f[\varphi]|
\le
2^{-k}
\Lambda_{\mathrm{mix}}^I(\mathcal P_I\varphi)
\int_S|\partial_I f(\omega)|\,d\omega
+
\|\mathcal R_I\varphi\|_{L^\infty(S)}.
\end{equation}

If $\mathcal R_I\varphi=0$, equivalently if $\varphi$ is centered on $I$, the
corresponding exact-centered bound is recovered. For $I=[d]$, the second
estimate is the smooth defect-adjusted Vitali bound.
\end{theorem}

\begin{proof}
By Proposition~\ref{prop:coordinate_projector_algebra},
$\mathcal P_I\varphi$ is centered on $I$. Since $\mathcal P_I$ is bounded on
$L^\infty(S)$, one has
$\mathcal P_I\varphi\in L^\infty(S)$. Therefore
Corollary~\ref{cor:bounded_residual_tv_extension} gives
$$
|\mathbb E_f[\mathcal P_I\varphi]|
\le
\frac{1}{2k}
\Lambda_{\mathrm{sl}}^I(\mathcal P_I\varphi)
\TV_I(f;S).
$$
Since $f$ is a probability density,
$$
|\mathbb E_f[\mathcal R_I\varphi]|
\le
\|\mathcal R_I\varphi\|_{L^\infty(S)}.
$$
Using $\varphi=\mathcal P_I\varphi+\mathcal R_I\varphi$ and the triangle
inequality proves \eqref{eq:defect_adjusted_tv_bound}.

The mixed-derivative estimate follows by the same decomposition, now applying
Theorem~\ref{thm:k_direction_mixed_derivative_bound} to the centered component
$\mathcal P_I\varphi$. Since
$\mathcal P_I\varphi\in L^\infty(S)$ and is centered on $I$,
Theorem~\ref{thm:k_direction_mixed_derivative_bound} yields
$$
|\mathbb E_f[\mathcal P_I\varphi]|
\le
2^{-k}
\Lambda_{\mathrm{mix}}^I(\mathcal P_I\varphi)
\int_S|\partial_I f(\omega)|\,d\omega.
$$
The same probability-density bound on
$\mathbb E_f[\mathcal R_I\varphi]$ gives
\eqref{eq:defect_adjusted_mixed_derivative_bound}.

Finally, Proposition~\ref{prop:coordinate_projector_algebra} gives
$\mathcal R_I\varphi=0$ if and only if $\varphi$ is centered on $I$. In that
case $\mathcal P_I\varphi=\varphi$, and the defect term vanishes.
\end{proof}

For the stochastic recourse problem, the theorem is applied to the residual
$\varphi_x=\widetilde v(x,\cdot)-v^{\INT}(x,\cdot)$. This gives the
decision-facing form of the bounds.

\begin{corollary}[Defect-adjusted residual bounds]
\label{cor:defect_adjusted_recourse_residual}
Fix $x\in X$ and a nonempty $I\subseteq[d]$, and write $k=|I|$.

\emph{(i) Anisotropic $\TV_\infty$ bound } 
If $f\in BV(S)$ is a probability density and
$\varphi_x\in L^\infty(S)$, then
\begin{equation}
\label{eq:defect_adjusted_recourse_tv}
|\widetilde Q(x)-Q(x)|
\le
\frac{1}{2k}
\Lambda_{\mathrm{sl}}^I(\mathcal P_I\varphi_x)
\TV_I(f;S)
+
\|\mathcal R_I\varphi_x\|_{L^\infty(S)}.
\end{equation}

\emph{(ii) Smooth mixed-derivative bound } 
If $f\in C^k(\overline S)$ is a probability density and
$\varphi_x\in L^\infty(S)$, then
\begin{equation}
\label{eq:defect_adjusted_recourse_mixed_derivative}
|\widetilde Q(x)-Q(x)|
\le
2^{-k}
\Lambda_{\mathrm{mix}}^I(\mathcal P_I\varphi_x)
\int_S|\partial_I f(\omega)|\,d\omega
+
\|\mathcal R_I\varphi_x\|_{L^\infty(S)}.
\end{equation}
For $I=[d]$, the second estimate is the smooth defect-adjusted Vitali residual
bound.
\end{corollary}

\begin{proof}
By the residual identity,
$$
\widetilde Q(x)-Q(x)=\mathbb E_f[\varphi_x].
$$
The two estimates follow from
Theorem~\ref{thm:defect_adjusted_projector_bounds} with
$\varphi=\varphi_x$.
\end{proof}

The projector decomposition is therefore structural. It turns the exact
slice-centered cancellation mechanism into a bound for arbitrary bounded residuals:
$\mathcal P_I\varphi$ is controlled by variation of the density, while
$\mathcal R_I\varphi$ enters as an explicit defect. This is the form needed
for convex approximation. One may regularize or audit the defect, but the
theory does not require exact centering as a hard modeling constraint.

For a smooth density and a bounded residual, every nonempty direction set
$I$ gives two valid defect-adjusted certificates. It is therefore natural to
retain the best certificate rather than to prescribe a direction set or a
variation formulation in advance.

\begin{corollary}[Best residual-geometry certificate]
\label{cor:best_defect_adjusted_certificate}
Let $f\in C^d(\overline S)$ be a probability density and let
$\varphi\in L^\infty(S)$. For every nonempty $I\subseteq[d]$, define
$$
\mathfrak B_{\mathrm{TV},I}(\varphi;f)
:=
\frac{1}{2|I|}
\Lambda_{\mathrm{sl}}^I(\mathcal P_I\varphi)
\TV_I(f;S),
$$
$$
\mathfrak B_{\mathrm{mix},I}(\varphi;f)
:=
2^{-|I|}
\Lambda_{\mathrm{mix}}^I(\mathcal P_I\varphi)
\int_S|\partial_I f(\omega)|\,d\omega,
$$
and
$$
\Delta_I(\varphi)
:=
\|\mathcal R_I\varphi\|_{L^\infty(S)}.
$$
Then
\begin{equation}
\label{eq:best_defect_adjusted_certificate}
|\mathbb E_f[\varphi]|
\le
\min_{\emptyset\neq I\subseteq[d]}
\left\{
\Delta_I(\varphi)
+
\min\left
\{
\mathfrak B_{\mathrm{TV},I}(\varphi;f),
\mathfrak B_{\mathrm{mix},I}(\varphi;f)
\right\}
\right\}.
\end{equation}
\end{corollary}

\begin{proof}
Fix a nonempty $I\subseteq[d]$. The anisotropic and mixed-derivative parts of
Theorem~\ref{thm:defect_adjusted_projector_bounds} give
$$
|\mathbb E_f[\varphi]|
\le
\Delta_I(\varphi)+\mathfrak B_{\mathrm{TV},I}(\varphi;f)
$$
and
$$
|\mathbb E_f[\varphi]|
\le
\Delta_I(\varphi)+\mathfrak B_{\mathrm{mix},I}(\varphi;f).
$$
Taking the smaller of the two bounds and then the minimum over the finitely
many nonempty direction sets proves the result.
\end{proof}

The same principle closes the loop between residual geometry and first-stage
decision quality.

\begin{corollary}[Uniform residual and decision-quality certificate]
\label{cor:uniform_decision_quality_certificate}
Let $f\in C^d(\overline S)$ be a probability density, and assume that
$\varphi_x\in L^\infty(S)$ for every $x\in X$. For every nonempty
$I\subseteq[d]$, define
$$
L_{\mathrm{sl},I}
:=
\sup_{x\in X}
\Lambda_{\mathrm{sl}}^I(\mathcal P_I\varphi_x),
$$
$$
L_{\mathrm{mix},I}
:=
\sup_{x\in X}
\Lambda_{\mathrm{mix}}^I(\mathcal P_I\varphi_x),
$$
and
$$
\delta_I
:=
\sup_{x\in X}
\|\mathcal R_I\varphi_x\|_{L^\infty(S)}.
$$
These quantities are allowed to take the value $+\infty$. Assume that the
certificate defined below is finite, and set
\begin{equation}
\label{eq:def_uniform_best_certificate}
\varepsilon_*
:=
\min_{\emptyset\neq I\subseteq[d]}
\left\{
\delta_I
+
\min\left
\{
\frac{L_{\mathrm{sl},I}}{2|I|}\TV_I(f;S),
\;
2^{-|I|}L_{\mathrm{mix},I}
\int_S|\partial_I f(\omega)|\,d\omega
\right\}
\right\}.
\end{equation}
Then
\begin{equation}
\label{eq:uniform_best_residual_bound}
\sup_{x\in X}|\widetilde Q(x)-Q(x)|
\le
\varepsilon_*.
\end{equation}
If $x^\star\in\arg\min_XF$ and
$\widetilde x\in\arg\min_X\widetilde F$, then
\begin{equation}
\label{eq:uniform_best_decision_quality}
0
\le
F(\widetilde x)-F(x^\star)
\le
2\varepsilon_*,
\qquad
|F(x^\star)-\widetilde F(\widetilde x)|
\le
\varepsilon_*.
\end{equation}
\end{corollary}

\begin{proof}
Apply Corollary~\ref{cor:best_defect_adjusted_certificate} to $\varphi_x$ and
use the definitions of $L_{\mathrm{sl},I}$, $L_{\mathrm{mix},I}$, and
$\delta_I$. This proves \eqref{eq:uniform_best_residual_bound}. Since the
first-stage linear term is the same in $F$ and $\widetilde F$,
\eqref{eq:uniform_best_residual_bound} implies
$$
\sup_{x\in X}|F(x)-\widetilde F(x)|\le\varepsilon_*.
$$
The two inequalities in \eqref{eq:uniform_best_decision_quality} follow from
the standard comparison of the minima of two functions whose uniform distance
is at most $\varepsilon_*$.
\end{proof}

\subsection{Tensorization and optimality of the constants}
\label{subsec:main_tensorization_sharpness}

The preceding results do not require independence. Product structure is
nevertheless useful for interpreting the two certificates and for identifying
their optimal constants. The coordinate projectors preserve tensor structure,
while the two density-variation terms separate in fundamentally different
ways.

\begin{proposition}[Tensor-product residuals and product densities]
\label{prop:main_tensor_summary}
Let
$$
S=\prod_{j=1}^dJ_j,
\qquad
J_j=[a_j,b_j],
$$
and let $I\subseteq[d]$ be nonempty. Suppose
$$
\varphi(\omega)=\prod_{j=1}^dh_j(\omega_j),
\qquad
h_j\in L^\infty(J_j),
$$
and define
$$
\bar h_j
:=
\frac{1}{b_j-a_j}\int_{J_j}h_j(t)\,dt.
$$
Then
\begin{equation}
\label{eq:main_tensor_projector_formula}
\mathcal P_I\varphi
=
\left(\prod_{i\in I}(h_i-\bar h_i)\right)
\left(\prod_{j\notin I}h_j\right).
\end{equation}
In particular, if $\bar h_i=0$ for every $i\in I$, then
$\mathcal P_I\varphi=\varphi$.

Now let
$$
f(\omega)=\prod_{j=1}^df_j(\omega_j),
$$
where every $f_j\in C^d(J_j)$ is a nonnegative probability density. Then
\begin{equation}
\label{eq:main_product_TVI_formula}
\TV_I(f;S)
=
\sum_{i\in I}\int_{J_i}|f_i'(t)|\,dt,
\end{equation}
\begin{equation}
\label{eq:main_product_mixed_formula}
\int_S|\partial_If(\omega)|\,d\omega
=
\prod_{i\in I}\int_{J_i}|f_i'(t)|\,dt.
\end{equation}
For $I=[d]$, the smooth Vitali identity gives
\begin{equation}
\label{eq:main_product_vitali_formula}
V(f;S)
=
\prod_{i=1}^d\int_{J_i}|f_i'(t)|\,dt.
\end{equation}
\end{proposition}

\begin{proof}
For a tensor-product residual, $\Pi_i$ acts only on the $i$th factor and
replaces $h_i$ by its average $\bar h_i$. Since the operators commute,
\eqref{eq:main_tensor_projector_formula} follows. For the product density,
$$
\partial_i f(\omega)
=
f_i'(\omega_i)\prod_{j\neq i}f_j(\omega_j),
$$
and
$$
\partial_I f(\omega)
=
\left(\prod_{i\in I}f_i'(\omega_i)\right)
\left(\prod_{j\notin I}f_j(\omega_j)\right).
$$
The formulas follow from Fubini's theorem, nonnegativity, and normalization of
the marginal densities. The all-axes identity follows from
\eqref{eq:vitali_smooth_identity}. The corresponding measure-level formulas,
the tensorization of both residual envelopes, and the classical Vitali product
theorem for fixed pointwise representatives are proved in
Appendix~\ref{subsec:appendix_tensorization_projectors}--\ref{subsec:appendix_product_variation}.
\end{proof}
The tensor-product identity shows that the projector does not create new
interactions. It acts independently on each coordinate factor by replacing it
with its centered version.
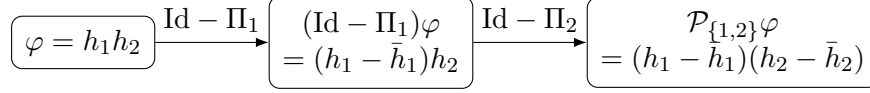
\begin{figure}[t]
\centering
\begin{tikzpicture}[
  node distance=1.5cm,
  box/.style={
    draw,
    rounded corners,
    align=center,
    inner sep=5pt
  },
  >=Latex
]

\node[box] (A)
{$\varphi=h_1h_2$};

\node[box, right=of A] (B)
{$(\Id-\Pi_1)\varphi$\\
$=(h_1-\bar h_1)h_2$};

\node[box, right=of B] (C)
{$\mathcal P_{\{1,2\}}\varphi$\\
$=(h_1-\bar h_1)(h_2-\bar h_2)$};

\draw[->] (A) -- node[above] {$\Id-\Pi_1$} (B);
\draw[->] (B) -- node[above] {$\Id-\Pi_2$} (C);

\end{tikzpicture}

\caption{
Factorwise action of the centered projector on a tensor-product residual.
Each selected factor is replaced by its centered version.
}
\label{fig:tensor_projector_geometry}
\end{figure}

\begin{corollary}[Additive and multiplicative product-density certificates]
\label{cor:main_additive_multiplicative_certificates}
Under the hypotheses of Proposition~\ref{prop:main_tensor_summary}, let
$\varphi$ be centered on $I$ and write $k=|I|$. If
$\varphi\in W^{1,\infty}(S)$, then
\begin{equation}
\label{eq:main_product_tv_certificate}
|\mathbb E_f[\varphi]|
\le
\frac{1}{2k}\Lambda_{\mathrm{sl}}^I(\varphi)
\sum_{i\in I}\int_{J_i}|f_i'(t)|\,dt.
\end{equation}
If $\varphi\in L^\infty(S)$, then
\begin{equation}
\label{eq:main_product_mixed_certificate}
|\mathbb E_f[\varphi]|
\le
2^{-k}\Lambda_{\mathrm{mix}}^I(\varphi)
\prod_{i\in I}\int_{J_i}|f_i'(t)|\,dt.
\end{equation}
\end{corollary}

\begin{proof}
Substitute \eqref{eq:main_product_TVI_formula} and
\eqref{eq:main_product_mixed_formula} into the two exact-centered bounds.
\end{proof}

\begin{remark}[No uniform ordering of the two certificates]
\label{rem:main_no_uniform_ordering_certificates}
The anisotropic certificate is additive in the selected marginal variations,
whereas the mixed-derivative certificate is multiplicative. The two bounds
also use different residual envelopes. Consequently, neither certificate
uniformly dominates the other. Appendix~\ref{subsec:appendix_branch_separation}
gives explicit smooth product densities and centered tensor-product residuals
for which the mixed-derivative certificate is strictly smaller, and a second
family for which the anisotropic certificate is strictly smaller.
\end{remark}

\begin{theorem}[Optimality of the variation and defect coefficients]
\label{thm:main_sharp_constants}
Let $I\subseteq[d]$ be nonempty and write $k=|I|$.

\begin{enumerate}
\item[\emph{(i)}] The coefficient $1/(2k)$ in
\eqref{eq:k_direction_tv_bound} is optimal. No smaller coefficient gives a
uniform bound over centered residuals in $W^{1,\infty}(S)$ and smooth
probability densities. The same conclusion holds for the corresponding
$\TV_\infty$ form because the sharpness construction has no variation outside
$I$.

\item[\emph{(ii)}] The coefficient $2^{-k}$ in
\eqref{eq:k_direction_mixed_derivative_bound} is optimal, even over smooth
probability densities and centered residuals in $W^{1,\infty}(S)$.

\item[\emph{(iii)}] For $I=[d]$, the coefficient $2^{-d}$ in the smooth
all-axes Vitali bound \eqref{eq:all_axes_vitali_bound} is optimal.

\item[\emph{(iv)}] The coefficient one multiplying
$\|\mathcal R_I\varphi\|_{L^\infty(S)}$ in the defect-adjusted bounds is
optimal.
\end{enumerate}
\end{theorem}

\begin{proof}
Parts \emph{(i)}--\emph{(iii)} follow from the single smooth tensor-product
construction in
Appendix~\ref{subsec:appendix_one_dimensional_near_extremizer}--\ref{subsec:appendix_mixed_vitali_sharpness}; see also
Figure~\ref{fig:appendix_unified_sharpness_construction}. For part \emph{(iv)},
take $\varphi\equiv c\neq0$. Then
$$
\mathcal P_I\varphi=0,
\qquad
\mathcal R_I\varphi=\varphi,
$$
and every probability density satisfies
$$
|\mathbb E_f[\varphi]|
=|c|
=\|\mathcal R_I\varphi\|_{L^\infty(S)}.
$$
\end{proof}

\begin{remark}[Scope of the sharpness statements]
\label{rem:main_scope_sharpness}
The sharpness statements concern the internal box quantities
$\TV_I(f;S)$ and $\int_S|\partial_If|$. They do not assert sharpness of the
nonsmooth extension-seminorm bound introduced below. They also do not require
one residual-density pair to saturate the centered and defect terms
simultaneously. The conclusion is coefficientwise: none of the displayed
coefficients can be reduced uniformly over its stated class.
\end{remark}

\subsection{Nonsmooth mixed derivatives and localization beyond box-supported laws}
\label{subsec:main_nonsmooth_extension}

The anisotropic formulation is already formulated at the $BV$ level. The
mixed-derivative formulation can also be extended beyond smooth densities, but the
quantity that is stable under mollification is the distributional measure
$D_I f$, not the representative-dependent corner formula for classical Vitali
variation. Because mollification is most naturally performed on the whole
space, the resulting quantity is defined through admissible extensions.

\begin{definition}[Extension mixed-variation seminorm]
\label{def:main_extension_mixed_variation}
Let $f\in L^1(S)$ and let $I\subseteq[d]$ be nonempty. Define
\begin{equation}
\label{eq:main_extension_mixed_variation}
V_I^{\mathrm{ext}}(f;S)
:=
\inf
\left\{
|D_I\bar f|(\mathbb R^d):
\begin{array}{l}
\bar f\in L^1(\mathbb R^d),\\
\bar f=f\text{ a.e. on }S,\\
D_I\bar f\in\mathcal M(\mathbb R^d)
\end{array}
\right\},
\end{equation}
with value $+\infty$ if the admissible class is empty.
\end{definition}

\begin{theorem}[Defect-adjusted nonsmooth mixed-derivative bound]
\label{thm:main_defect_adjusted_nonsmooth_mixed}
Fix a nonempty $I\subseteq[d]$. Let $f\in L^1(S)$ be a probability density and assume
$V_I^{\mathrm{ext}}(f;S)<\infty$. Then, for every
$\varphi\in L^\infty(S)$,
\begin{equation}
\label{eq:main_defect_adjusted_nonsmooth_mixed}
|\mathbb E_f[\varphi]|
\le
2^{-|I|}
\Lambda_{\mathrm{mix}}^I(\mathcal P_I\varphi)
V_I^{\mathrm{ext}}(f;S)
+
\|\mathcal R_I\varphi\|_{L^\infty(S)}.
\end{equation}
\end{theorem}

\begin{proof}
By Proposition~\ref{prop:coordinate_projector_algebra},
$\mathcal P_I\varphi$ is centered on $I$. The centered nonsmooth estimate is
proved in
Theorem~\ref{thm:appendix_distributional_extension_bound} and
Corollary~\ref{cor:appendix_intrinsic_nonsmooth_mixed_bound}. Applying that
estimate to $\mathcal P_I\varphi$ and bounding the expectation of
$\mathcal R_I\varphi$ by its $L^\infty$ norm gives
\eqref{eq:main_defect_adjusted_nonsmooth_mixed}.
\end{proof}

The seminorm in Definition~\ref{def:main_extension_mixed_variation} includes
possible boundary-extension costs. The following properties make this
distinction explicit.

\begin{proposition}[Internal variation and boundary-extension cost]
\label{prop:main_extension_variation_properties}
The map $V_I^{\mathrm{ext}}(\cdot;S)$ is an extended seminorm. Moreover:

\begin{enumerate}
\item[\emph{(i)}] If the internal derivative $D_If$ is a finite measure on
$\operatorname{int}S$, then
$$
|D_If|(\operatorname{int}S)
\le
V_I^{\mathrm{ext}}(f;S).
$$
In particular, for smooth $f$,
$$
\int_S|\partial_If(\omega)|\,d\omega
\le
V_I^{\mathrm{ext}}(f;S).
$$

\item[\emph{(ii)}] If $f\in C_c^k(\operatorname{int}S)$, then
$$
V_I^{\mathrm{ext}}(f;S)
=
\int_S|\partial_If(\omega)|\,d\omega.
$$

\item[\emph{(iii)}] Boundary contact can create a strict gap. For
$J=[a,b]$ and $f\equiv1$ on $J$,
$$
V_{\{1\}}^{\mathrm{ext}}(1;J)=2,
$$
although the internal variation is zero.
\end{enumerate}
\end{proposition}

\begin{proof}
The proof is given in
Proposition~\ref{prop:appendix_extension_mixed_variation_properties}. The
compact-support identity follows because the zero extension introduces no
boundary derivative measure, while the lower bound in part \emph{(i)} prevents
a smaller extension cost.
\end{proof}

The admissible class is nonempty for a broad class of product densities.

\begin{corollary}[A finite extension cost for product $BV$ densities]
\label{cor:main_product_bv_extension_cost}
Let
$$
f(\omega)=\prod_{j=1}^df_j(\omega_j),
$$
where every $f_j\in BV([a_j,b_j])$ is a nonnegative probability density. Let
$f_j(a_j+)$ and $f_j(b_j-)$ denote the one-sided traces. Then
\begin{equation}
\label{eq:main_product_bv_extension_cost}
V_I^{\mathrm{ext}}(f;S)
\le
\prod_{i\in I}
\left(
|Df_i|(J_i)+f_i(a_i+)+f_i(b_i-)
\right).
\end{equation}
If the selected marginal traces vanish at both endpoints, then
\begin{equation}
\label{eq:main_product_bv_extension_equality}
V_I^{\mathrm{ext}}(f;S)
=
\prod_{i\in I}|Df_i|(J_i).
\end{equation}
\end{corollary}

\begin{proof}
The zero extensions of the marginals give an admissible tensor-product
extension. The derivative and total-variation calculation is given in
Proposition~\ref{prop:appendix_product_bv_extension_cost}. If the selected
traces vanish, the resulting upper bound agrees with the internal lower bound
from Proposition~\ref{prop:main_extension_variation_properties}.
\end{proof}

The mollification proof does not smooth $f$ only inside $S$ and then
renormalize. One first chooses a whole-space extension $\bar f$, forms
$\bar f_\varepsilon=\bar f*\eta_\varepsilon$, applies the analytic smooth
mixed-derivative estimate to $\bar f_\varepsilon|_S$, and then passes to the
limit. Although global mass is preserved, the restriction
$\bar f_\varepsilon|_S$ need not be normalized because convolution may
transport mass across $\partial S$. Figure~\ref{fig:appendix_mollification_step_density}
illustrates this boundary spillover. Figure~\ref{fig:appendix_tensor_mollification_geometry}
illustrates the tensor-product identity
$$
\partial_I\bar f_\varepsilon
=
(D_I\bar f)*\eta_\varepsilon
$$
and the contraction of the global mixed-derivative variation budget.

The box results also give explicit bounds for laws with full support. Let $p$
be a probability density on $\mathbb R^d$, let
$$
Z_S:=\int_Sp(\omega)\,d\omega>0,
$$
and write $\mathcal P_I^S$ and $\mathcal R_I^S$ for the projectors applied to
the restriction of a residual to $S$.

\begin{corollary}[Localization of full-support laws]
\label{cor:main_full_support_localization}
Fix a nonempty $I\subseteq[d]$. Let $\varphi:\mathbb R^d\to\mathbb R$ be measurable and assume
$$
\int_{\mathbb R^d\setminus S}|\varphi(\omega)|p(\omega)\,d\omega<\infty.
$$
Then the following hold.

\begin{enumerate}
\item[\emph{(i)}] If $p|_S\in BV(S)$ and
$\varphi|_S\in L^\infty(S)$, then
\begin{equation}
\label{eq:main_full_support_tv_localization}
\begin{aligned}
|\mathbb E_p[\varphi]|
\le{}&
\frac{1}{2|I|}
\Lambda_{\mathrm{sl}}^I(\mathcal P_I^S\varphi)
\TV_I(p;S)
\\
&+
Z_S\|\mathcal R_I^S\varphi\|_{L^\infty(S)}
+
\int_{\mathbb R^d\setminus S}|\varphi|p.
\end{aligned}
\end{equation}

\item[\emph{(ii)}] If $p|_S\in C^{|I|}(\overline S)$ and
$\varphi|_S\in L^\infty(S)$, then
\begin{equation}
\label{eq:main_full_support_mixed_localization}
\begin{aligned}
|\mathbb E_p[\varphi]|
\le{}&
2^{-|I|}
\Lambda_{\mathrm{mix}}^I(\mathcal P_I^S\varphi)
\int_S|\partial_Ip(\omega)|\,d\omega
\\
&+
Z_S\|\mathcal R_I^S\varphi\|_{L^\infty(S)}
+
\int_{\mathbb R^d\setminus S}|\varphi|p.
\end{aligned}
\end{equation}

\item[\emph{(iii)}] If
$V_I^{\mathrm{ext}}(p|_S;S)<\infty$ and
$\varphi|_S\in L^\infty(S)$, then
\begin{equation}
\label{eq:main_full_support_nonsmooth_localization}
\begin{aligned}
|\mathbb E_p[\varphi]|
\le{}&
2^{-|I|}
\Lambda_{\mathrm{mix}}^I(\mathcal P_I^S\varphi)
V_I^{\mathrm{ext}}(p|_S;S)
\\
&+
Z_S\|\mathcal R_I^S\varphi\|_{L^\infty(S)}
+
\int_{\mathbb R^d\setminus S}|\varphi|p.
\end{aligned}
\end{equation}
If $D_Ip\in\mathcal M(\mathbb R^d)$, then
$$
V_I^{\mathrm{ext}}(p|_S;S)
\le
|D_Ip|(\mathbb R^d).
$$
\end{enumerate}
\end{corollary}

\begin{proof}
Decompose the expectation into its contribution on $S$ and its tail outside
$S$. On $S$, use the conditional density $p/Z_S$. The normalization factor
cancels in the variation term and remains only in front of the projector
defect. The complete calculation is given in
Theorem~\ref{thm:appendix_full_support_localization}.
\end{proof}

\begin{corollary}[Uniform tail control]
\label{cor:main_uniform_tail_control}
Let $\{\varphi_x:x\in X\}$ be a family of measurable residuals. If
$|\varphi_x(\omega)|\le H$ for every $x$ and almost every $\omega$, then
$$
\sup_{x\in X}
\int_{\mathbb R^d\setminus S}|\varphi_x(\omega)|p(\omega)\,d\omega
\le
H(1-Z_S).
$$
More generally, if $|\varphi_x|\le G$ for every $x$, where
$G\in L^1(p)$, then
$$
\sup_{x\in X}
\int_{\mathbb R^d\setminus S}|\varphi_x|p
\le
\int_{\mathbb R^d\setminus S}Gp,
$$
and the right-hand side converges to zero along boxes exhausting
$\mathbb R^d$.
\end{corollary}

\begin{proof}
The bounds follow directly from domination and integrability.
\end{proof}

\begin{remark}[Geometric scope beyond boxes]
\label{rem:main_geometric_scope_beyond_boxes}
Appendix~\ref{subsec:appendix_affine_images_boxes} shows that the box theory
transfers exactly to an invertible affine image $\Omega=T(S)$ once the
projectors, residual envelopes, and density derivatives are transported to the
affine frame. The first-order variation is then expressed through directional derivative
measures in the transported affine frame, and the mixed term through iterated
directional derivatives in the same frame. Appendix~\ref{subsec:appendix_lipschitz_domains} gives a weaker,
domain-dependent $BV$ substitute on a connected Lipschitz domain by means of a
right inverse of the divergence operator. The latter result does not preserve
the coefficients $1/(2|I|)$ or $2^{-|I|}$, and the present method does not
yield an intrinsic classical Vitali counterpart on an arbitrary nonlinear
domain.
\end{remark}

The exact-centered and defect-adjusted bounds identify the analytical content
of the residual certificate. The optimality theorem shows that the centered
variation coefficients cannot be reduced uniformly, while the extension
results identify the precise regularity and geometric scope of the argument.
We now turn to a different limitation: even though the defect is explicit, it
cannot generally be forced to zero by requiring the approximation to be
convex.

\section{Exact centering and convex-approximation obstruction}
\label{sec:exact_centering_obstruction}

The preceding section identifies an analytical limitation and a regularity
extension. The variation coefficients are sharp in their natural box classes,
and the mixed-derivative estimate remains stable beyond smooth densities after
introducing the extension seminorm. We now identify a distinct
representational limitation. The defect term is not merely a technical
convenience: exact centering is a compatibility condition between the slice
averages of the integer-recourse value and the convexity of the approximation.

There are exceptional cases. If $v^{\INT}$ is already convex on the box, then
$\widetilde v=v^{\INT}$ gives the zero residual. In dimension one, an
unrestricted constant approximation can always center the residual over the
interval. These cases should not be confused with the general situation. In
higher dimension, requiring
$\bar\Pi_i(\widetilde v-v^{\INT})=0$ forces the $i$-slice average of
$v^{\INT}$ to be representable by a convex function of the remaining
coordinates.

This is also where the present box-based theory differs from classical
shifted-LP cancellation. For scalar unit-lattice ceiling recourse, the
half-shifted LP residual is a centered unit-periodic sawtooth, and periodic
zero cell mean coincides with one-dimensional centering on full unit cells.
That coincidence is special. On bounded boxes and in dimension $d\ge2$,
periodic cell means, coordinate slice-centering, and convex representability
of the induced slice-average profile are different requirements. The periodic
comparison is recorded in Appendix~\ref{app:periodic_finite_box_shifted_lp}.

\begin{proposition}[Slice-average compatibility for exact one-direction centering]
\label{prop:one_direction_centered_realizability}
Fix $i\in[d]$ and let $v\in L^1(S)$. Let
$U=\prod_{j\in[d]}(\alpha_j,\beta_j)$ be an open box containing $S$, and let
$U_{-i}$ denote its projection onto the coordinates different from $i$. The
following statements are equivalent.

\begin{enumerate}
\item[\emph{(i)}] There exists a finite-valued convex function
$\widetilde v:U\to\mathbb R$ such that, after restriction to $S$,
$$
\bar\Pi_i(\widetilde v-v)=0
\qquad
\text{a.e. on }S_{-i}.
$$

\item[\emph{(ii)}] The slice-average profile $\bar\Pi_i v$ admits a
representative on $S_{-i}$ that is the restriction of a finite-valued convex
function $m_i:U_{-i}\to\mathbb R$.
\end{enumerate}

In particular, if $\bar\Pi_i v$ has no representative whose restriction to
$\operatorname{ri}(S_{-i})$ is continuous, then no finite-valued convex
approximation on a neighborhood of $S$ can induce a residual centered in
direction $i$.
\end{proposition}

\begin{proof}
Assume first that \emph{(i)} holds. Define, for $z\in U_{-i}$,
$$
m_i(z):=
\frac{1}{b_i-a_i}
\int_{a_i}^{b_i}
\widetilde v(z_1,\ldots,z_{i-1},t,z_i,\ldots,z_{d-1})\,dt .
$$
Since $\widetilde v$ is finite-valued and convex on the open box $U$, it is
continuous. Hence $m_i$ is finite-valued. Convexity of $m_i$ follows by
integrating the convexity inequality for $\widetilde v$ in the $i$th
coordinate. The centering condition gives
$\bar\Pi_i\widetilde v=\bar\Pi_i v$ a.e. on $S_{-i}$, while
$\bar\Pi_i\widetilde v=m_i$. Thus $\bar\Pi_i v$ has the required convex
representative.

Conversely, suppose \emph{(ii)} holds. Define
$\widetilde v(\omega):=m_i(\omega_{-i})$ on $U$. This function is finite-valued
and convex as the composition of a convex function with the coordinate
projection. It is independent of coordinate $i$, so
$\bar\Pi_i\widetilde v=m_i=\bar\Pi_i v$ a.e. on $S_{-i}$. Hence
$\bar\Pi_i(\widetilde v-v)=0$ a.e. on $S_{-i}$.

The final statement follows because every finite-valued convex function on an
open convex set is continuous.
\end{proof}

Proposition~\ref{prop:one_direction_centered_realizability} shows that exact
centering is not a free normalization. It requires the relevant slice-average
profile of $v^{\INT}$ to have convex geometry. This requirement can fail even
for very simple integer-recourse functions.

\begin{proposition}[TU ceiling obstruction]
\label{prop:2d_tu_no_exact_smz_main}
There exists a separable two-dimensional totally unimodular ceiling-recourse
function for which no finite-valued convex approximation on a neighborhood of
the box can induce a residual centered in one coordinate direction. Hence
exact all-axes slice-centering is not generally attainable within convex
approximation classes.
\end{proposition}

\begin{proof}
The construction is the identity-matrix recourse value
$$
v^{\INT}(b_1,b_2)=\lceil b_1\rceil+\lceil b_2\rceil
\qquad \text{on } [0,2]^2 .
$$
This is a separable totally unimodular ceiling-recourse function. Its
coordinate-$1$ slice average is equal a.e. to
$3/2+\lceil b_2\rceil$, which has a jump at $b_2=1$. By
Proposition~\ref{prop:one_direction_centered_realizability}, exact
coordinate-$1$ residual centering would require this slice-average profile to
admit a continuous convex representative on the remaining coordinate. This is
impossible. The full verification is given in Appendix~\ref{app:tu_ceiling_obstruction}.
\end{proof}
The obstruction is geometric rather than computational. Exact centering would
require the discontinuous slice-average profile of the integer recourse value
to coincide with the slice average of a finite-valued convex function.

The conclusion is not that exact centering never occurs. Rather, exact
centering is a structural compatibility condition. It requires the relevant
slice-average profiles of the exact recourse value to be representable by
convex functions. Since convexity constrains $\widetilde v$, not the residual
$\widetilde v-v^{\INT}$, such compatibility cannot be assumed for general
mixed-integer recourse.

This is why the defect-adjusted formulation is the natural general statement.
For any residual,
$$
\varphi=\mathcal P_I\varphi+\mathcal R_I\varphi .
$$
The component $\mathcal P_I\varphi$ is centered in the selected directions and
is controlled by the exact-centered estimates, while the slice-mean defect
$\mathcal R_I\varphi$ enters additively. The theory therefore does not require
exact centering as a hard modeling constraint; it quantifies what remains true
when exact centering is unavailable or undesirable.

The two limitations are complementary. The sharpness construction shows that
the centered variation term cannot be uniformly improved, even for smooth
densities. The ceiling-recourse obstruction shows that the noncentered defect
cannot generally be eliminated through convex approximation. Thus both terms
in the defect-adjusted certificate represent genuine structural phenomena
rather than artifacts of the proof.

\section{Computational framework and audit methodology}
\label{sec:framework_design_revised}

The bounds in Section~\ref{sec:problem_setup} do not say that exact
slice-centering should be imposed as a hard constraint. Section~\ref{sec:exact_centering_obstruction}
shows why: exact centering is a compatibility condition between slice averages
of $v^{\INT}$ and convexity of the approximation. The computational principle
is therefore weaker. We fit convex approximations using ordinary residual-error terms together with soft penalties on coordinate slice means and lifted projector defects. The
centered envelopes entering the variation terms are evaluated afterward on
the independent audit grid. Corollary~\ref{cor:best_defect_adjusted_certificate}
also shows that the direction set and the variation formulation may be selected
rather than fixed in advance. In the two-dimensional experiment, we audit the
one-direction residual through the directional primitive certificate and the
two-direction residual through the all-axes mixed-derivative/Vitali
certificate.

The section proceeds in the same order as the theory. We first develop the
discrete projector calculus on the training grid, mirroring the continuous
operators of Section~\ref{subsec:coordinate_projectors}. We then define the LP
relaxation and the finite-box LP-slope calibration comparator used as a structured
benchmark. The direct max-affine fits and their projector penalties follow,
and the section ends with the validation and audit diagnostics.

Throughout this section, points of the box $S$ are denoted by $b$. For a convex
approximation $\widetilde v:S\to\mathbb R$, write
$$
\varphi_{\widetilde v}(b):=\widetilde v(b)-v^{\INT}(b).
$$
For each nonempty $I\subseteq[d]$, the continuous decomposition is
$$
\varphi_{\widetilde v}
=
\mathcal P_I\varphi_{\widetilde v}
+
\mathcal R_I\varphi_{\widetilde v}.
$$
The fitted approximation is judged by ordinary approximation quality, but it is
also audited through the quantities appearing in the defect-adjusted bounds:
the centered residual component, its slice or mixed envelope, and the
projector defect. The discrete operators below are the grid analogues of the continuous
coordinate averages and projectors. Their purpose is not to approximate a
density, but to measure slice-level residual bias on the training or audit
grid.

\subsection{Discrete projector calculus on a tensor grid } 
\label{subsec:framework_discrete_projectors}

Let the training grid be the tensor grid
$$
B^h
=
\left\{
b^\alpha=
(\xi_{\alpha_1}^{(1)},\ldots,\xi_{\alpha_d}^{(d)}):
\alpha\in\mathcal N
\right\}
\subset S,
\qquad
\mathcal N:=\prod_{i\in[d]}\{1,\ldots,n_i\}.
$$
At each grid point we evaluate
$$
y_\alpha:=v^{\INT}(b^\alpha),
\qquad
\alpha\in\mathcal N.
$$
For a candidate approximation $\widetilde v$, define the grid residual array
$R_{\widetilde v}^h=(r_\alpha)_{\alpha\in\mathcal N}$ by
$$
r_\alpha:=\widetilde v(b^\alpha)-y_\alpha .
$$
When the approximation is clear, we write simply $R$.

For each coordinate $i\in[d]$, let $\varpi_r^i>0$, $r=1,\ldots,n_i$, be
normalized quadrature weights for coordinate-$i$ averaging:
$$
\sum_{r=1}^{n_i}\varpi_r^i=1.
$$
These are geometric averaging weights for the projector calculus. They are not
training weights, validation weights, audit weights, or density weights.

For $\alpha=(\alpha_i,\alpha_{-i})$, define the unlifted coordinate-$i$
average by
$$
(A_i^hR)_{\alpha_{-i}}
:=
\sum_{r=1}^{n_i}
\varpi_r^i
R_{(\alpha_1,\ldots,\alpha_{i-1},r,\alpha_{i+1},\ldots,\alpha_d)}.
$$
The lifted coordinate-average operator is
$$
(\Pi_i^hR)_\alpha:=(A_i^hR)_{\alpha_{-i}}.
$$
Thus $A_i^hR$ is the lower-dimensional slice-average profile, while
$\Pi_i^hR$ is the same profile lifted back to the full tensor grid by making it
constant in coordinate $i$.

For $I\subseteq[d]$, set
$$
\mathcal P_I^h
:=
\prod_{i\in I}(\Id-\Pi_i^h),
\qquad
\mathcal P_\varnothing^h:=\Id,
\qquad
\mathcal R_I^h:=\Id-\mathcal P_I^h.
$$
The array $\mathcal P_I^hR$ is centered in the selected discrete directions;
$\mathcal R_I^hR$ is the lifted slice-mean defect.

\begin{proposition}[Discrete projector algebra]
\label{prop:discrete_projector_algebra_framework}
The operators $\Pi_i^h$ are linear, idempotent, and commute. Consequently,
$\mathcal P_I^h$ and $\mathcal R_I^h$ are projections for every
$I\subseteq[d]$. If $i\in I$, then
$$
\Pi_i^h\mathcal P_I^hR=0
\qquad
\text{for every residual array }R.
$$
For nonempty $I$,
$$
\mathcal R_I^hR=0
\quad\Longleftrightarrow\quad
\Pi_i^hR=0\ \text{for every }i\in I.
$$
\end{proposition}

\begin{proof}
Linearity is immediate. Since $\Pi_i^hR$ is already constant in coordinate
$i$, applying $\Pi_i^h$ again does not change it, so $(\Pi_i^h)^2=\Pi_i^h$.
The operators commute because coordinate averages on a tensor grid can be
taken in any order. Hence the commuting product
$\mathcal P_I^h=\prod_{i\in I}(\Id-\Pi_i^h)$ is a projection, and so is
$\mathcal R_I^h=\Id-\mathcal P_I^h$.

If $i\in I$, then
$$
\Pi_i^h\mathcal P_I^h
=
\Pi_i^h(\Id-\Pi_i^h)
\prod_{j\in I\setminus\{i\}}(\Id-\Pi_j^h)
=0.
$$
This proves discrete centering. Finally,
$\mathcal R_I^hR=0$ implies $R=\mathcal P_I^hR$, hence
$\Pi_i^hR=0$ for every $i\in I$. Conversely, if $\Pi_i^hR=0$ for every
$i\in I$, then each factor $\Id-\Pi_i^h$ leaves $R$ unchanged, so
$\mathcal P_I^hR=R$ and $\mathcal R_I^hR=0$.
\end{proof}

In two dimensions, for $I=[2]$, the projector has the familiar
inclusion--exclusion form
$$
\mathcal P_{[2]}^hR
=
R-\Pi_1^hR-\Pi_2^hR+\bar r^h\mathbf 1,
$$
where
$$
\bar r^h
=
\sum_{\alpha\in\mathcal N}
\left(
\prod_{i\in[d]}\varpi_{\alpha_i}^i
\right)r_\alpha.
$$
The two-dimensional case makes the inclusion--exclusion structure explicit.
The operator $\mathcal P_{\{1,2\}}$ first removes the row and column means and
then restores the global mean to avoid double subtraction.















If the two grid indices are written as $(p,q)$, then
$$
(\mathcal R_{[2]}^hR)_{pq}
=
(A_1^hR)_q+(A_2^hR)_p-\bar r^h.
$$
Thus the two-direction defect is the row-mean plus column-mean residual bias,
corrected by the global mean.

\subsection{LP relaxation and finite-box LP-slope calibration comparator } 
\label{subsec:framework_shifted_lp}

The first benchmark is the LP relaxation
\begin{equation}
\label{eq:framework_lp_primal_definition}
v^{\LP}(b)
:=
\min_y
\{q^\top y:Wy\ge b,\ y\ge0\}.
\end{equation}
Assume that, on $S$, it admits the finite max-affine representation
\begin{equation}
\label{eq:framework_lp_max_affine_definition}
v^{\LP}(b)
=
\max_{k\in[K]}\{a_k+s_k^\top b\}.
\end{equation}
In the usual dual representation, the slopes $s_k$ come from dual extreme
points; the notation also allows affine offsets.

The finite-box LP-slope calibration comparator keeps this LP slope dictionary and fits
only vertical shifts:
$$
\mathcal G_{\LP}
=
\left\{
b\mapsto
\max_{k\in[K]}\{a_k+s_k^\top b+\beta_k\}:
\beta\in\mathbb R^K
\right\}.
$$
For every training point $b^\alpha$, choose an active LP piece
$$
k(\alpha)
\in
\arg\max_{\ell\in[K]}
\{a_\ell+s_\ell^\top b^\alpha\},
$$
using deterministic tie-breaking, and define
$$
\mathcal N_k:=\{\alpha\in\mathcal N:k(\alpha)=k\}.
$$
Let $\kappa_\alpha^{\mathrm{tr}}$ be normalized training weights. The weighted average
LP gap on the training grid is
$$
\bar g
:=
\sum_{\alpha\in\mathcal N}
\kappa_\alpha^{\mathrm{tr}}
\bigl(y_\alpha-v^{\LP}(b^\alpha)\bigr).
$$
The finite-box LP-slope calibration comparator answers a focused question: if the LP
affine geometry is kept fixed and only the intercepts are corrected, how much
of the finite-box residual can be removed? This separates slope accuracy from
vertical calibration. The parameter $\tau$ shrinks each piece-specific correction toward the global
average LP gap, which prevents unstable corrections for LP pieces active on
only a small part of the training grid.
\begin{definition}[finite-box LP-slope calibration comparator]
\label{def:framework_gamma_empirical}
Fix $\tau>0$. For each LP piece $k\in[K]$, let $\gamma_k$ be the minimizer of
\begin{equation}
\label{eq:framework_gamma_piece_problem}
\min_{\gamma\in\mathbb R}
\left\{
\sum_{\alpha\in\mathcal N_k}
\kappa_\alpha^{\mathrm{tr}}
\bigl(\gamma-(y_\alpha-a_k-s_k^\top b^\alpha)\bigr)^2
+
\tau(\gamma-\bar g)^2
\right\}.
\end{equation}
The finite-box LP-slope calibration comparator is
\begin{equation}
\label{eq:framework_gamma_empirical_formula_new}
v^\Gamma(b)
=
\max_{k\in[K]}
\{a_k+s_k^\top b+\gamma_k\}.
\end{equation}
\end{definition}
The construction is piecewise with respect to the LP-active partition: the sets $\mathcal N_k$ are determined by
the active partition of the unshifted LP relaxation. Because the intercept
corrections can change which affine piece is active, $v^\Gamma$ is not, in
general, the empirical least-squares minimizer over the entire class
$\mathcal G_{\LP}$. We use it as a transparent LP-slope calibration benchmark
that preserves the slope dictionary of the LP relaxation.
\begin{proposition}[Closed-form LP-slope calibration intercepts]
\label{prop:framework_gamma_optimizer}
For each $k\in[K]$, the minimizer in
\eqref{eq:framework_gamma_piece_problem} is
\begin{equation}
\label{eq:framework_eslp_intercepts_new}
\gamma_k
=
\frac{
\sum_{\alpha\in\mathcal N_k}
\kappa_\alpha^{\mathrm{tr}}
\bigl(y_\alpha-a_k-s_k^\top b^\alpha\bigr)
+
\tau\bar g
}{
\sum_{\alpha\in\mathcal N_k}\kappa_\alpha^{\mathrm{tr}}+\tau
}.
\end{equation}
If $\mathcal N_k=\emptyset$, then $\gamma_k=\bar g$.
\end{proposition}

\begin{proof}
For fixed $k$, the objective in \eqref{eq:framework_gamma_piece_problem} is a
strictly convex scalar quadratic because $\tau>0$. Its first-order condition is
$$
0
=
2\sum_{\alpha\in\mathcal N_k}
\kappa_\alpha^{\mathrm{tr}}
\bigl(\gamma-(y_\alpha-a_k-s_k^\top b^\alpha)\bigr)
+
2\tau(\gamma-\bar g).
$$
Solving this equation gives \eqref{eq:framework_eslp_intercepts_new}. If
$\mathcal N_k=\emptyset$, the corresponding sum is zero and the formula gives
$\gamma_k=\bar g$.
\end{proof}

The function $v^\Gamma$ retains the LP slope dictionary but uses intercepts
calibrated on the finite training box.
It preserves the affine geometry of the LP relaxation and fits the vertical
corrections on the same training box used for the direct convex fits. This is
different from the classical shifted-LP constants, which are derived from
full-cell periodic or asymptotically periodic residuals. The distinction is
discussed in Appendix~\ref{app:periodic_finite_box_shifted_lp}.

\subsection{Direct max-affine convex fitting } 
\label{subsec:framework_direct_fitting}

The direct convex fits use a max-affine model with one supporting plane
attached to each training point. For each $\alpha\in\mathcal N$, introduce a
fitted height $u_\alpha\in\mathbb R$ and a fitted subgradient
$g_\alpha\in\mathbb R^d$. Define
\begin{equation}
\label{eq:framework_max_affine_extension_new}
\widetilde v_{u,g}(b)
=
\max_{\alpha\in\mathcal N}
\{u_\alpha+g_\alpha^\top(b-b^\alpha)\}.
\end{equation}
The supporting-plane inequalities are
\begin{equation}
\label{eq:framework_convexity_cuts_new}
u_\beta
\ge
u_\alpha+g_\alpha^\top(b^\beta-b^\alpha),
\qquad
\alpha,\beta\in\mathcal N.
\end{equation}

\begin{proposition}[Supporting-plane construction]
\label{prop:supporting_plane_construction_framework}
If $(u,g)$ satisfies \eqref{eq:framework_convexity_cuts_new}, then
$\widetilde v_{u,g}$ is convex and interpolates the fitted heights:
$$
\widetilde v_{u,g}(b^\alpha)=u_\alpha,
\qquad
\alpha\in\mathcal N.
$$
Conversely, if $\widetilde v$ is convex on a neighborhood of $S$,
$u_\alpha=\widetilde v(b^\alpha)$, and
$g_\alpha\in\partial\widetilde v(b^\alpha)$, then
\eqref{eq:framework_convexity_cuts_new} holds.
\end{proposition}

\begin{proof}
The function $\widetilde v_{u,g}$ is a pointwise maximum of affine functions,
hence convex. At $b^\beta$, the affine function indexed by $\beta$ gives
$\widetilde v_{u,g}(b^\beta)\ge u_\beta$. The inequalities
\eqref{eq:framework_convexity_cuts_new} give
$$
u_\alpha+g_\alpha^\top(b^\beta-b^\alpha)\le u_\beta
\qquad
\text{for all }\alpha\in\mathcal N.
$$
Taking the maximum over $\alpha$ yields
$\widetilde v_{u,g}(b^\beta)\le u_\beta$, hence equality. The converse is the
subgradient inequality for convex functions.
\end{proof}

The fitted residual array is
$$
R(u)_\alpha:=u_\alpha-y_\alpha.
$$
Let $e_\alpha\ge0$ represent $|R(u)_\alpha|$, let $t\ge0$ represent the
training-grid residual sup-norm, and let $\eta_{\alpha,i}\ge0$ represent
$|g_{\alpha,i}|$. With $\theta_{\fit}\in[0,1]$ and
$\lambda_{\mathrm{grad}}\ge0$, define
\begin{equation}
\label{eq:framework_common_fit_objective}
J_0(u,g,e,\eta,t)
:=
\theta_{\fit}t
+
(1-\theta_{\fit})
\sum_{\alpha\in\mathcal N}\kappa_\alpha^{\mathrm{tr}}e_\alpha
+
\lambda_{\mathrm{grad}}
\sum_{\alpha\in\mathcal N}
\kappa_\alpha^{\mathrm{tr}}
\sum_{i\in[d]}\eta_{\alpha,i}.
\end{equation}
The parameter $\theta_{\fit}$ balances worst-case and average training error,
while $\lambda_{\mathrm{grad}}$ regularizes the fitted slopes.

The fit-only approximation $C_0$ minimizes $J_0$ subject to
\eqref{eq:framework_convexity_cuts_new} and
$$
e_\alpha\ge u_\alpha-y_\alpha,
\qquad
e_\alpha\ge y_\alpha-u_\alpha,
\qquad
t\ge e_\alpha,
\qquad
\alpha\in\mathcal N,
$$
together with
$$
\eta_{\alpha,i}\ge g_{\alpha,i},
\qquad
\eta_{\alpha,i}\ge -g_{\alpha,i},
\qquad
\alpha\in\mathcal N,\quad i\in[d].
$$
Thus $C_0$ is the direct convex fit without projector regularization.

\subsection{Projector-regularized max-affine fitting } 
\label{subsec:framework_projector_regularized_fitting}

The projector-regularized fits do not impose exact centering as a hard
constraint; Section~\ref{sec:exact_centering_obstruction} shows that exact
centering may be infeasible. Instead, the penalties reduce systematic
slice-level bias and the lifted defect quantities that appear in the
defect-adjusted bounds. For $i\in[d]$, define the signed coordinate-$i$ slice-mean
residual by
$$
m_{\alpha_{-i}}^{(i)}(u)
:=
(A_i^hR(u))_{\alpha_{-i}}.
$$
Large values of $m_{\alpha_{-i}}^{(i)}(u)$ indicate systematic residual bias
along the coordinate-$i$ grid line indexed by $\alpha_{-i}$. For every $i\in[d]$, let
$$
\mathcal N_{-i}
:=
\prod_{j\neq i}\{1,\ldots,n_j\}
$$
denote the index set of coordinate-$i$ slices.

\begin{definition}[Projector-regularized max-affine approximation]
\label{def:framework_CI_unified}
Let $\emptyset\neq I\subseteq[d]$. Fix penalties $\mu_i\ge0$ for $i\in I$ and
$\mu_I\ge0$. The approximation $C_I$ is any max-affine function
$C_I(b)=\widetilde v_{u^\star,g^\star}(b)$ associated with an optimal solution
of
\begin{equation}
\label{eq:framework_CI_unified_lp}
\begin{aligned}
\min\quad&
J_0(u,g,e,\eta,t)
+
\sum_{i\in I}
\mu_i
\sum_{\alpha_{-i}}
\left(
\prod_{j\neq i}\varpi_{\alpha_j}^j
\right)
d_{\alpha_{-i}}^{(i)}
+
\mu_I z_I
\\
\text{s.t } \quad&
u_\beta
\ge
u_\alpha
+
g_\alpha^\top
\left(
b^\beta-b^\alpha
\right),
\qquad
\alpha,\beta\in\mathcal N,
\\
&
e_\alpha
\ge
u_\alpha-y_\alpha,
\qquad
e_\alpha
\ge
y_\alpha-u_\alpha,
\qquad
t\ge e_\alpha,
\qquad
\alpha\in\mathcal N,
\\
&
\eta_{\alpha,r}
\ge
g_{\alpha,r},
\qquad
\eta_{\alpha,r}
\ge
-g_{\alpha,r},
\qquad
r\in[d],
\quad
\alpha\in\mathcal N,
\\
&
m_{\alpha_{-i}}^{(i)}(u)
=
\left(
A_i^hR(u)
\right)_{\alpha_{-i}},
\qquad
i\in I,
\quad
\alpha_{-i}\in\mathcal N_{-i},
\\
&
d_{\alpha_{-i}}^{(i)}
\ge
m_{\alpha_{-i}}^{(i)}(u),
\qquad
d_{\alpha_{-i}}^{(i)}
\ge
-m_{\alpha_{-i}}^{(i)}(u),
\qquad
i\in I,
\quad
\alpha_{-i}\in\mathcal N_{-i},
\\
&
z_I
\ge
\left(
\mathcal R_I^hR(u)
\right)_\alpha,
\qquad
z_I
\ge
-\left(
\mathcal R_I^hR(u)
\right)_\alpha,
\qquad
\alpha\in\mathcal N,
\\
&
e_\alpha\ge0,
\qquad
t\ge0,
\qquad
\eta_{\alpha,r}\ge0,
\qquad
\alpha\in\mathcal N,
\quad
r\in[d],
\\
&
d_{\alpha_{-i}}^{(i)}\ge0,
\qquad
i\in I,
\quad
\alpha_{-i}\in\mathcal N_{-i},
\\
&
z_I\ge0.
\end{aligned}
\end{equation}

\end{definition}

The formulation \eqref{eq:framework_CI_unified_lp} is a linear program. The
supporting-plane inequalities enforce convexity of the fitted max-affine
surface, and the residual, slope, slice-mean, and projector-defect quantities
are handled by linear epigraph constraints. Since each $\Pi_i^h$ is linear,
the operators $\mathcal P_I^h$ and $\mathcal R_I^h$ are linear as well.

The penalty terms are soft residual-geometry controls rather than hard
centering constraints. They discourage systematic slice-level bias and large lifted projector defects, which are precisely the residual quantities that enter the defect-adjusted bounds.

\subsection{Validation and audit diagnostics } 
\label{subsec:framework_validation_audit}

Model selection is performed on a validation grid, while the final diagnostics
are reported on a separate audit grid. Let
$R=(r_\alpha)_{\alpha\in\mathcal N_{\aud}}$ be the audit-grid residual of a
selected approximation, and let $\kappa_\alpha^{\aud}$ be normalized audit
weights. We report
$$
L^\infty(R):=\max_{\alpha\in\mathcal N_{\aud}}|r_\alpha|,
$$
$$
L^1(R):=
\sum_{\alpha\in\mathcal N_{\aud}}
\kappa_\alpha^{\aud}|r_\alpha|,
$$
and
$$
L^2(R):=
\left(
\sum_{\alpha\in\mathcal N_{\aud}}
\kappa_\alpha^{\aud}r_\alpha^2
\right)^{1/2}.
$$
These are ordinary pointwise, average absolute, and root-mean-square
approximation errors on the audit grid.

The decision-facing diagnostic is the density-weighted signed residual
mismatch
$$
\widehat{\mathcal E}^h(R)
:=
\left|
\sum_{\alpha\in\mathcal N_{\aud}}
\kappa_\alpha^{f,h}r_\alpha
\right|,
$$
where $\kappa_\alpha^{f,h}$ are normalized quadrature weights incorporating the
audit density. This discretizes $|\mathbb E_f[\widetilde v-v^{\INT}]|$. It is
not sufficient by itself, because positive and negative residuals can cancel
under the chosen density.

The projector diagnostics are
$$
\|\Pi_i^hR\|_{\infty,h},
\qquad i\in[d],
$$
and, for nonempty $I\subseteq[d]$,
$$
\|\mathcal R_I^hR\|_{\infty,h}.
$$
They distinguish oscillatory residuals from residuals with systematic
slice-level bias. A small density-weighted mismatch together with a large
projector defect is not a well-centered residual; it is a residual whose error
happens to cancel under the audit density. This is precisely the distinction
the theory is designed to make: the defect-adjusted bounds control the
centered component and the slice-mean defect separately.

On a tensor audit grid, these operators are computed by the same weighted
slice averages used in the training projector calculus. Thus, for fixed
$\alpha_{-i}$, $\Pi_i^hR$ replaces the residual values along the coordinate-$i$
grid line by their weighted average. The projector $\mathcal P_I^h$ then
subtracts the selected slice averages by inclusion--exclusion, and
$\mathcal R_I^hR$ records the lifted slice-mean bias that remains. Hence the
projector diagnostics are not additional optimization quantities; they are
post-processing computations on the audit residual array.

We also compute discrete analogues of the defect-adjusted bounds. For nonempty
$I\subseteq[d]$, define the slice-envelope proxy
$$
\Phi_{\mathrm{sl},I}^h(R)
:=
\max_{i\in I}
\max_{\alpha_{-i}}
\operatorname{Trapz}_{i,h}
\left[
|R(\alpha_{-i},\cdot)|
\right],
$$
and the mixed-slice proxy
$$
\Phi_{\mathrm{mix},I}^h(R)
:=
\max_{\alpha_{-I}}
\operatorname{Trapz}_{I,h}
\left[
|R(\alpha_{-I},\cdot)|
\right].
$$
The anisotropic defect proxy is
$$
\widehat{\mathcal B}_{\mathrm{TV},\Def}^{(I)}(R)
:=
\frac{1}{2|I|}
\TV_I(f;S)
\Phi_{\mathrm{sl},I}^h(\mathcal P_I^hR)
+
\|\mathcal R_I^hR\|_{\infty,h}.
$$
The smooth mixed-derivative, or all-axes Vitali, defect proxy is
$$
\widehat{\mathcal B}_{\mathrm{mix},\Def}^{(I)}(R)
:=
2^{-|I|}
\Phi_{\mathrm{mix},I}^h(\mathcal P_I^hR)
\int_S|\partial_I f(b)|\,db
+
\|\mathcal R_I^hR\|_{\infty,h}.
$$
For $I=[d]$, this is denoted
$\widehat{\mathcal B}_{\mathrm{Vit},\Def}^{([d])}$ when the smooth Vitali
identity is used.

These quantities mirror the continuous defect-adjusted estimates. They are
audit diagnostics, not certified continuous-domain upper bounds, unless
quadrature errors and off-grid ranges are enclosed separately. The numerical
study therefore evaluates each approximation at four levels: ordinary
approximation error, density-weighted signed mismatch, projector-defect
structure, and defect-adjusted residual proxies.

The supporting-plane system in the direct max-affine fits is enforced by cut
generation in the implementation. The finite termination and correctness
argument are deferred to Appendix~\ref{app:cut_generation}.

\section{Numerical illustration}
\label{sec:numerics}

This section gives a small numerical illustration of the residual-projector
framework. The purpose is not to propose a new benchmark problem, nor to claim
that one approximation family dominates another in general. The purpose is
more specific: to show how the residual quantities in
Section~\ref{sec:problem_setup} are computed, how projector-regularized convex
fits differ from an LP relaxation, and how the defect-adjusted bounds guide the
audit of a convex approximation.

We work on the box $S=[0,10]^2$ with recourse argument $b=(b_1,b_2)$. For
each $b\in S$, the exact mixed-integer recourse value is
\begin{equation}
\label{eq:numerics_exact_recourse}
v^{\INT}(b)
=
\min_{y,z,w}
\left\{
q^\top y+\rho z+\kappa w:
Ay+z\mathbf 1\ge b,\;
(A_1+A_2)y+w\ge b_1+b_2,\;
y\in\mathbb Z_+^3,\ z,w\ge0
\right\},
\end{equation}
where $A_1$ and $A_2$ are the rows of
$$
A=
\begin{pmatrix}
2&1&3\\
1&2&3
\end{pmatrix},
\qquad
q=(2,\ 2.45,\ 4.35)^\top,
\qquad
\rho=1.65,
\qquad
\kappa=0.20.
$$
The three integer variables represent discrete recourse actions with coverage
vectors $(2,1)$, $(1,2)$, and $(3,3)$. The variable $z$ prices the largest
remaining coordinate shortfall, while $w$ prices the remaining aggregate
shortfall. This gives a small finite recourse model whose value is finite on
$S$, but still nonsmooth and nonconvex because of the integer variables.

The exact value in \eqref{eq:numerics_exact_recourse} is evaluated by
enumerating $y_j\in\{0,\ldots,8\}$ for $j=1,2,3$. This cap is exact on
$[0,10]^2$ by a simple dominance argument. If $y_1\ge9$, then removing one unit
of the first action still leaves at least $(16,8)$ units of coverage from that
action alone. On the whole box, the first coordinate and the aggregate
constraint remain covered, while the second-coordinate shortfall can increase
by at most one unit. The objective decreases by $2$, and the penalty can
increase by at most $\rho=1.65$, so such a solution cannot be optimal. The same
argument applies to $y_2$. If $y_3\ge5$, removing one unit of the third action
still leaves at least $(12,12)$ units of coverage from that action alone,
which covers both coordinates and the aggregate requirement throughout $S$ and
strictly lowers the cost.

We compare five convex approximations, all constructed on the training grid and
then evaluated on the independent audit grid. The first is the LP relaxation
$v^{\LP}$, obtained by replacing $y\in\mathbb Z_+^3$ with
$y\in\mathbb R_+^3$. Its dual feasible region is
$$
\mathcal D
=
\left\{
\pi\in\mathbb R_+^3:
A_{1j}\pi_1+A_{2j}\pi_2+(A_{1j}+A_{2j})\pi_3\le q_j\ \forall j,\;
\pi_1+\pi_2\le\rho,\;
\pi_3\le\kappa
\right\}.
$$
Thus
\begin{equation}
\label{eq:numerics_lp_dual}
v^{\LP}(b)
=
\max_{\pi\in\operatorname{ext}(\mathcal D)}
\{(\pi_1+\pi_3)b_1+(\pi_2+\pi_3)b_2\}.
\end{equation}
The extreme points of $\mathcal D$ give the affine pieces of the LP relaxation.

The second approximation is the finite-box LP-slope calibration comparator $v^\Gamma$.
It keeps exactly the affine slopes in \eqref{eq:numerics_lp_dual}. For each
training point, we first assign the point to an active LP piece using the
unshifted LP relaxation. We then fit one intercept correction for each active
piece by the finite-box formula in Definition~\ref{def:framework_gamma_empirical}.
Thus $v^\Gamma$ tests how much of the residual can be removed by vertical
calibration of the LP pieces, without changing the LP slope dictionary.

The remaining three approximations are direct max-affine convex fits. The
model $C_0$ solves the fit-only supporting-plane formulation with approximation
error and slope regularization. The model $C_{\{1\}}$ adds a penalty on the
coordinate-$1$ slice-mean residual. The model $C_{\{1,2\}}$ penalizes both coordinate
slice means and the lifted two-direction projector defect. These are not
hard-centering models. They keep the approximation convex while reducing the
defect quantities that appear in the defect-adjusted bounds.

The training, validation, and audit grids have sizes $31\times31$,
$41\times41$, and $121\times121$. Hyperparameters are selected on the
validation grid; all reported values are computed on the audit grid. The audit
density is the product density $f(b_1,b_2)=f_1(b_1)f_1(b_2)$, where $f_1$ is a
normal density with mean $\mu=5$ and standard deviation $\sigma=3$, truncated
to $[0,10]$ and renormalized. Proposition~\ref{prop:main_tensor_summary} gives
the variation constants directly from the marginal variation:
$$
|D_1f|(S)=|D_2f|(S)=\TV(f_1;[0,10]),
\qquad
\TV_\infty(f;S)=2\,\TV(f_1;[0,10]),
$$
and
$$
V(f;S)
=
\int_S|\partial_1\partial_2 f(b)|\,db
=
\left(\TV(f_1;[0,10])\right)^2.
$$
For the truncated normal used here this gives
$$
\TV(f_1;[0,10])=0.220742,
\qquad
\TV_\infty(f;S)=0.441484,
\qquad
V(f;S)=0.048727.
$$
These variation constants enter only the defect-adjusted audit proxies. The
ordinary approximation errors and the density-weighted residual mismatch are
computed directly from the audit-grid residuals.

For a selected approximation $\widetilde v$, define the audit residual
$$
\varphi_{\widetilde v}^h(b^{pq})
=
\widetilde v(b^{pq})-v^{\INT}(b^{pq}).
$$
Write this residual array as $R=(R_{pq})$ on the $121\times121$ audit grid,
with geometric quadrature weights $\varpi_p^1$ and $\varpi_q^2$ in the two
coordinates. The discrete coordinate averages are computed explicitly as
$$
(\Pi_1^hR)_{pq}
=
\sum_r \varpi_r^1 R_{rq},
\qquad
(\Pi_2^hR)_{pq}
=
\sum_s \varpi_s^2 R_{ps}.
$$
Thus $\Pi_1^hR$ is the columnwise slice mean lifted constantly in the first
coordinate, while $\Pi_2^hR$ is the rowwise slice mean lifted constantly in the
second coordinate. With
$$
\bar R
:=
\sum_{r,s}\varpi_r^1\varpi_s^2R_{rs},
$$
the two-direction projectors are
$$
(\mathcal P_{\{1,2\}}^hR)_{pq}
=
R_{pq}-(\Pi_1^hR)_{pq}-(\Pi_2^hR)_{pq}+\bar R,
\text{ and }
(\mathcal R_{\{1,2\}}^hR)_{pq}
=
(\Pi_1^hR)_{pq}+(\Pi_2^hR)_{pq}-\bar R.
$$
These formulas are the discrete row--column centering identities used in the
reported diagnostics.

We report ordinary approximation errors $L^\infty$, $L^1$, and $L^2$, the
observed density-weighted residual mismatch
$$
|\widehat{\mathbb E}_f[\varphi_{\widetilde v}^h]|
=
\left|
\sum_{p,q}
\kappa_{pq}^{f,h}\varphi_{\widetilde v}^h(b^{pq})
\right|,
$$
and the projector quantities
$
\|\Pi_1^h\varphi^h\|_\infty,
\|\Pi_2^h\varphi^h\|_\infty,
\|\mathcal R_{\{1,2\}}^h\varphi^h\|_\infty.
$\\
The last three quantities distinguish oscillatory residual error from
systematic slice-level bias. This distinction is exactly what the theory uses:
the centered part is controlled by variation of the density, while the
noncentered defect is paid for directly.

We also compute two theorem-facing audit proxies. For one-direction centering,
$I=\{i\}$, the centered component is $(\Id-\Pi_i^h)R$ and the defect is
$\Pi_i^hR$. We therefore use the primitive-envelope version of the
one-direction bound:
$$
\widehat{\mathcal B}^{(1)}_{\mathrm{TV},\mathrm{Def}}(R)
=
\min_{i=1,2}
\left\{
|D_i f|(S)\,
\Phi_{i,\mathrm{pr}}^h((\Id-\Pi_i^h)R)
+
\|\Pi_i^hR\|_{\infty,h}
\right\}.
$$
Here $\Phi_{i,\mathrm{pr}}^h$ is the maximum absolute cumulative trapezoidal
primitive along coordinate-$i$ grid lines. For example,
$$
\Phi_{1,\mathrm{pr}}^h(R)
=
\max_q\max_m
\left|
\operatorname{Trapz}_{1,h}
\left[
R_{1q},\ldots,R_{mq}
\right]
\right|,
$$
with the analogous definition in direction $2$.

For the two-direction Vitali proxy, the centered component is
$\mathcal P_{\{1,2\}}^hR$ and the defect is
$\mathcal R_{\{1,2\}}^hR$. Since $d=2$, the factor in the all-axes
mixed-derivative bound is $2^{-2}=1/4$, giving
$$
\widehat{\mathcal B}^{(2)}_{\mathrm{Vit},\mathrm{Def}}(R)
=
\frac14\,
V(f;S)\,
\operatorname{Trapz}_{\{1,2\},h}
\left[
|\mathcal P_{\{1,2\}}^hR|
\right]
+
\|\mathcal R_{\{1,2\}}^hR\|_{\infty,h}.
$$
These are audit-grid diagnostics mirroring the continuous defect-adjusted
bounds. They are not claimed to be rigorous continuous-domain certificates
unless the quadrature errors and off-grid ranges are enclosed separately.

The selected hyperparameters and cut-generation diagnostics are shown in
Table~\ref{tab:numerics_hyperparameters}. The direct fits are solved by cut
generation for the supporting-plane inequalities, followed by a full all-pairs
convexity check. The small negative worst slacks are numerical tolerances from
the final separation check.

\begin{table}[t]
\centering
\caption{Selected hyperparameters and cut-generation diagnostics } 
\label{tab:numerics_hyperparameters}
\setlength{\tabcolsep}{4pt}
\small
\resizebox{\textwidth}{!}{%
\begin{tabular}{lccccc}
\toprule
Method & Selected parameters & Direction & Cuts & Iterations & Worst slack \\
\midrule
$C_0$ &
$\lambda_{\mathrm{grad}}=5\cdot10^{-4}$ &
-- & $9851$ & $11$ & $-3.720\cdot10^{-11}$ \\
$C_{\{1\}}$ &
$\lambda_{\mathrm{grad}}=5\cdot10^{-4},\ \mu=5\cdot10^{-2}$ &
$1$ & $10117$ & $10$ & $-7.387\cdot10^{-10}$ \\
$C_{\{1,2\}}$ &
$\lambda_{\mathrm{grad}}=5\cdot10^{-4},\
(\mu_1,\mu_2,\mu_{\{1,2\}})=(5\cdot10^{-3},5\cdot10^{-3},10^{-3})$ &
-- & $10154$ & $10$ & $-1.874\cdot10^{-11}$ \\
$v^\Gamma$ &
$\tau=10^{-4}$ &
-- & -- & -- & -- \\
\bottomrule
\end{tabular}
}
\end{table}
Table~\ref{tab:numerics_main_diagnostics} reports the audit-grid results. The
LP relaxation is weakest in every displayed metric. It has worst-case error
$1.1625$, average absolute error $0.4820$, and density-weighted residual
mismatch $0.5077$. Its two-direction defect is also large, equal to $0.8262$.
Thus the LP residual is not merely large pointwise; it contains a substantial
slice-mean component, which is precisely the component that cannot be removed
by centered cancellation.
\begin{table}[t]
\centering
\caption{Approximation metrics and residual-projector diagnostics on the audit grid } 
\label{tab:numerics_main_diagnostics}
\setlength{\tabcolsep}{3pt}
\scriptsize
\resizebox{\textwidth}{!}{%
\begin{tabular}{lrrrrrrrrr}
\toprule
Method
& $L^\infty$
& $L^1$
& $L^2$
& $|\widehat{\mathbb E}_f[\varphi^h]|$
& $\|\Pi_1^h\varphi^h\|_\infty$
& $\|\Pi_2^h\varphi^h\|_\infty$
& $\|\mathcal R_{\{1,2\}}^h\varphi^h\|_\infty$
& $\widehat{\mathcal B}^{(1)}_{\mathrm{TV},\mathrm{Def}}$
& $\widehat{\mathcal B}^{(2)}_{\mathrm{Vit},\mathrm{Def}}$ \\
\midrule
$v^{\LP}$   & $1.1625$ & $0.4820$ & $0.5437$ & $0.5077$ & $0.6444$ & $0.6639$ & $0.8262$ & $0.9628$ & $1.0461$ \\
$v^\Gamma$ & $\mathbf{0.6018}$ & $0.1944$ & $0.2331$ & $\mathbf{0.0068}$ & $\mathbf{0.1812}$ & $\mathbf{0.1713}$ & $\mathbf{0.3552}$ & $\mathbf{0.3982}$ & $\mathbf{0.5648}$ \\
$C_0$       & $0.6042$ & $0.1943$ & $0.2345$ & $0.0252$ & $0.2543$ & $0.1957$ & $0.4416$ & $0.4423$ & $0.6501$ \\
$C_{\{1\}}$       & $0.6042$ & $\mathbf{0.1935}$ & $\mathbf{0.2327}$ & $0.0173$ & $0.2383$ & $0.1861$ & $0.4388$ & $0.4319$ & $0.6468$ \\
$C_{\{1,2\}}$       & $0.6042$ & $0.1940$ & $0.2337$ & $0.0211$ & $0.2378$ & $0.1957$ & $0.4177$ & $0.4390$ & $0.6263$ \\
\bottomrule
\end{tabular}
}
\end{table}

LP-slope calibration comparator removes most of the systematic LP bias.
Relative to $v^{\LP}$, the approximation $v^\Gamma$ reduces $L^\infty$ from
$1.1625$ to $0.6018$, $L^1$ from $0.4820$ to $0.1944$, and the observed
density-weighted mismatch from $0.5077$ to $0.0068$. The two-direction defect
falls from $0.8262$ to $0.3552$. On this example, the LP slopes are already
well aligned with the integer-recourse surface, so finite-box intercept
correction is a strong structured benchmark.
The direct max-affine fits also improve strongly on the raw LP relaxation.
All three direct fits reduce the weighted $L^1$ error by about $60\%$
and reduce the density-weighted residual mismatch by more than 95\%. They also reduce the lifted slice means, the two-direction defect,
and the defect-adjusted proxies. Thus their improvement is visible not only in
ordinary approximation error, but also in the quantities that enter the
residual-projector bounds.
The direct fits do not dominate $v^\Gamma$ in this illustration. The finite-box LP-slope calibration comparator is best in worst-case error, density-weighted mismatch, both
one-direction slice-mean defects, the two-direction defect, and the two
displayed defect-adjusted proxies. This is not a contradiction of the
projector framework. It shows that, when the LP slope dictionary is already
well matched to the recourse surface, fitting finite-box intercepts can be more
effective than refitting a new max-affine surface.

Within the direct family, the projector penalties move the residual geometry in
the predicted directions. The one-direction fit $C_{\{1\}}$ gives the best direct
weighted average diagnostics, with $L^1=0.1935$, $L^2=0.2327$, and
density-weighted mismatch $0.0173$; it also gives the smallest one-direction
defect-adjusted proxy among the direct fits, $B^{(1)}_{\TV,\Def}=0.4319$.
The two-direction fit $C_{\{1,2\}}$ is stronger for the two-direction projector
geometry: its two-direction defect is $0.4177$, compared with $0.4416$ for
$C_0$ and $0.4388$ for $C_{\{1\}}$, and it gives the smallest direct
Vitali-type defect-adjusted proxy, $B^{(2)}_{\Vit,\Def}=0.6263$. Thus $C_{\{1\}}$
is the better direct fit for average and expectation-oriented diagnostics,
whereas $C_{\{1,2\}}$ is the better direct fit for the two-direction defect structure.

This illustration gives a limited but useful message. The raw LP relaxation is
weak because its residual is large and strongly noncentered. The LP-slope calibration comparator is strong here because the LP slope dictionary is already
well aligned with the exact value surface. The projector-regularized direct
fits improve substantially on the raw LP relaxation, and their penalties move
the residual geometry in the predicted directions: $C_{\{1\}}$ improves the
one-direction and expectation-oriented diagnostics, while $C_{\{1,2\}}$ improves the
two-direction defect structure.

\section{Conclusion}
\label{sec:conclusion}

In this paper we studied expectation error for convex approximations of
mixed-integer recourse through the signed residual
$\widetilde v-v^{\INT}$. The main point is that the error entering the
first-stage objective is not a pointwise approximation error, but its signed
expectation under the uncertainty density. Coordinate slice-centering removes
the boundary terms in the relevant integration-by-parts identities. This gives
anisotropic $\TV_\infty$ and mixed-derivative bounds for exactly centered
residuals, and coordinate projectors extend the theory to arbitrary bounded residuals
by separating the centered component from an explicit slice-mean defect. The
best-certificate formulation allows the selected directions and the variation
formulation to be chosen after the residual and density are specified, and its
uniform form gives a direct first-stage decision-quality guarantee.

The product formulas clarify the difference between the two variation
formulations. For product densities, the anisotropic term is additive in the
marginal variations, whereas the mixed-derivative term is multiplicative. A
single smooth tensor-product construction shows that the coefficients
$1/(2|I|)$ and $2^{-|I|}$ are sharp for every nonempty direction set. In the
all-axes case, this gives sharpness of the smooth Vitali coefficient $2^{-d}$.
The coefficient multiplying the projector defect is also optimal. These
results show that the two terms in the defect-adjusted certificate cannot be
removed by a uniform improvement of the analysis.

The mixed-derivative formulation also admits a precise nonsmooth extension. The
stable object is a finite distributional mixed-derivative measure of an
admissible whole-space extension, which leads to the extension seminorm
$V_I^{\mathrm{ext}}(f;S)$. The mollification argument preserves the global
mixed-variation budget but may move mass across the boundary of the box; this
is why the extension cost may exceed the internal smooth variation. The same
framework yields explicit tail-adjusted bounds for full-support laws and
transfers exactly to affine images of boxes in the transported coordinate
frame. On general Lipschitz domains, the present method retains only a
domain-dependent first-order $BV$ estimate rather than the explicit box
constants or a classical Vitali counterpart.

Finally, the convex-realizability obstruction identifies a separate modeling
limitation. The sharpness results show that the centered variation term cannot
be uniformly improved, while the ceiling-recourse example shows that the
defect cannot generally be forced to zero within convex approximation classes.
The numerical illustration demonstrates how these two structural components
can be audited for max-affine approximations. The experiment is not meant as a
large-scale computational benchmark; its role is to show that ordinary
approximation error, expected residual mismatch, the centered variation term,
and the projector defect are distinct and computable diagnostics for convex
approximations of integer-recourse value functions.

\appendix

\section*{Appendix overview}

The appendices collect the proofs, constructions, figures, and comparison
material supporting the main theorem path. Appendix~\ref{app:tensorization_mollification_sharpness_extensions}
contains the tensor-product identities, the two certificate-separation examples,
the smooth sharpness construction, the mollification figures and proofs, the
extension mixed-variation results, the full-support localization formulas, and
the extensions beyond boxes. Appendix~\ref{app:tu_ceiling_obstruction}
verifies the two-dimensional totally unimodular ceiling obstruction.
Appendix~\ref{app:cut_generation} records the finite-dimensional validity of
the max-affine fitting programs and the cut-generation procedure.
Appendix~\ref{app:periodic_finite_box_shifted_lp} separates the classical
periodic cancellation mechanism from the LP-slope calibration comparator
used in the numerical illustration.
 
\section{Tensorization, mollification, sharpness, and extensions beyond boxes}
\label{app:tensorization_mollification_sharpness_extensions}

This appendix collects the additional analytical material needed for the
extended version of the paper. The order is deliberate. We first isolate the
exact tensor-product identities for residuals and densities. We then use those
identities to construct a single one-dimensional near-extremizing pair whose
tensor products prove sharpness of both variation bounds. After that, we record
the mollification argument and the corresponding distributional extension of
the mixed-derivative formulation. The final parts explain precisely how the box
results localize full-support laws, how they transform under invertible affine
maps, and what remains available on a general Lipschitz domain.

Throughout this appendix,
$$
S=\prod_{j=1}^d J_j,
\qquad
J_j=[a_j,b_j],
\qquad
\ell_j:=b_j-a_j,
$$
and $I\subseteq[d]$ is nonempty, with
$$
I=\{i_1<\cdots<i_k\},
\qquad
k:=|I|.
$$
The coordinate-average operators $\Pi_i$, the centered projector
$\mathcal P_I$, the defect projector $\mathcal R_I$, and the envelopes
$\Lambda_{\mathrm{sl}}^I$ and $\Lambda_{\mathrm{mix}}^I$ are those defined in
Section~\ref{sec:problem_setup}. We use the distributional mixed derivative
and measure notation from Definition~\ref{def:distributional_mixed_derivative_main}.

\subsection{Tensorization of the coordinate projectors and residual envelopes}
\label{subsec:appendix_tensorization_projectors}

The coordinate projectors are particularly transparent on tensor-product
residuals. This is useful for interpretation, but it is also the structural
reason that one-dimensional near-extremizers can be lifted to arbitrary
selected direction sets.

\begin{proposition}[Tensorization of residual projectors and envelopes]
\label{prop:appendix_tensorized_projectors_envelopes}
Let $h_j\in L^\infty(J_j)$ for $j=1,\ldots,d$, and define
$$
\varphi(\omega):=\prod_{j=1}^d h_j(\omega_j),
\qquad
\omega\in S.
$$
For each $j$, let
$$
\bar h_j:=\frac1{\ell_j}\int_{J_j}h_j(t)\,dt.
$$
Then:

\begin{enumerate}
\item[\emph{(i)}] For every $i\in[d]$,
$$
\Pi_i\varphi
=
\bar h_i\prod_{j\neq i}h_j,
$$
and
\begin{equation}
\label{eq:appendix_tensorized_projector_formula}
\mathcal P_I\varphi
=
\left(\prod_{i\in I}(h_i-\bar h_i)\right)
\left(\prod_{j\notin I}h_j\right).
\end{equation}
In particular, if $\bar h_i=0$ for every $i\in I$, then
$$
\mathcal P_I\varphi=\varphi,
\qquad
\mathcal R_I\varphi=0.
$$

\item[\emph{(ii)}]
\begin{equation}
\label{eq:appendix_tensorized_mixed_envelope}
\Lambda_{\mathrm{mix}}^I(\varphi)
=
\left(\prod_{i\in I}\|h_i\|_{L^1(J_i)}\right)
\left(\prod_{j\notin I}\|h_j\|_{L^\infty(J_j)}\right).
\end{equation}

\item[\emph{(iii)}]
\begin{equation}
\label{eq:appendix_tensorized_slice_envelope}
\Lambda_{\mathrm{sl}}^I(\varphi)
=
\max_{i\in I}
\left\{
\|h_i\|_{L^1(J_i)}
\prod_{j\neq i}\|h_j\|_{L^\infty(J_j)}
\right\}.
\end{equation}

\item[\emph{(iv)}] If $f_j\in L^1(J_j)$ and
$f(\omega)=\prod_{j=1}^d f_j(\omega_j)$, then
\begin{equation}
\label{eq:appendix_tensorized_pairing}
\int_S\varphi(\omega)f(\omega)\,d\omega
=
\prod_{j=1}^d
\int_{J_j}h_j(t)f_j(t)\,dt.
\end{equation}
\end{enumerate}
\end{proposition}

\begin{proof}
Fix $i\in[d]$. Since every factor other than $h_i$ is constant along an
$i$-coordinate slice,
$$
\begin{aligned}
(\Pi_i\varphi)(\omega)
&=
\frac1{\ell_i}
\int_{J_i}
 h_i(t)\prod_{j\neq i}h_j(\omega_j)\,dt
\\
&=
\bar h_i\prod_{j\neq i}h_j(\omega_j).
\end{aligned}
$$
Consequently,
$$
(\Id-\Pi_i)\varphi
=
(h_i-\bar h_i)\prod_{j\neq i}h_j.
$$
The factors $\Id-\Pi_i$ commute and act on distinct tensor factors, which
proves \eqref{eq:appendix_tensorized_projector_formula}.

For fixed $\omega_{-I}$, Fubini gives
$$
\int_{S_I}|\varphi(\omega_{-I},t_I)|\,dt_I
=
\left(\prod_{j\notin I}|h_j(\omega_j)|\right)
\prod_{i\in I}\|h_i\|_{L^1(J_i)}.
$$
Taking the essential supremum over the independent coordinates outside $I$
proves \eqref{eq:appendix_tensorized_mixed_envelope}. Likewise, for fixed
$i\in I$,
$$
\int_{J_i}|\varphi(\omega_{-i},t)|\,dt
=
\left(\prod_{j\neq i}|h_j(\omega_j)|\right)
\|h_i\|_{L^1(J_i)}.
$$
Taking the essential supremum and then the maximum over $i\in I$ proves
\eqref{eq:appendix_tensorized_slice_envelope}. Finally,
\eqref{eq:appendix_tensorized_pairing} is Fubini's theorem.
\end{proof}

\begin{remark}[Projectors preserve tensor structure]
\label{rem:appendix_projectors_preserve_tensor_structure}
The identity \eqref{eq:appendix_tensorized_projector_formula} is stronger than
an envelope estimate. It shows that $\mathcal P_I$ does not create interactions
between tensor factors. It simply replaces each selected factor $h_i$ by its
centered version $h_i-\bar h_i$.
\end{remark}

\subsection{Product formulas for distributional and classical variation}
\label{subsec:appendix_product_variation}

The product formulas for density variation require one distinction. The
coordinate derivative measures and the mixed distributional derivative are
$L^1$-based objects and therefore depend only on the almost-everywhere classes
of the marginals. Classical Vitali variation is instead defined through point
values at rectangle vertices. Its product formula is therefore stated for
fixed pointwise representatives. These statements agree in the smooth regime,
but they should not be identified for arbitrary $BV$ representatives.

\begin{proposition}[Distributional derivatives of a product density]
\label{prop:appendix_product_distributional_derivatives}
Let $f_j\in BV([a_j,b_j])$ be nonnegative one-dimensional probability densities:
$$
f_j\ge0,
\qquad
\int_{J_j}f_j(t)\,dt=1,
\qquad
j=1,\ldots,d.
$$
Define
$$
f(\omega):=\prod_{j=1}^d f_j(\omega_j).
$$
Then $f\in BV(S)$ and:

\begin{enumerate}
\item[\emph{(i)}] For every $i\in[d]$,
\begin{equation}
\label{eq:appendix_product_first_derivative_measure}
D_i f
=
Df_i\otimes
\bigotimes_{j\neq i}\bigl(f_j(t_j)\,dt_j\bigr),
\end{equation}
with the tensor factors placed in coordinate order. Consequently,
\begin{equation}
\label{eq:appendix_product_coordinate_variation}
|D_i f|(S)=|Df_i|(J_i),
\qquad
\TV_I(f;S)=\sum_{i\in I}|Df_i|(J_i).
\end{equation}

\item[\emph{(ii)}]
\begin{equation}
\label{eq:appendix_product_mixed_derivative_measure}
D_I f
=
\left(\bigotimes_{i\in I}Df_i\right)
\otimes
\left(\bigotimes_{j\notin I}f_j(t_j)\,dt_j\right),
\end{equation}
and
\begin{equation}
\label{eq:appendix_product_mixed_derivative_mass}
|D_I f|(S)
=
\prod_{i\in I}|Df_i|(J_i).
\end{equation}
\end{enumerate}
\end{proposition}

\begin{proof}
Fix $i\in[d]$ and let $\psi\in C_c^1(\operatorname{int}S)$. Fubini's theorem
and one-dimensional $BV$ integration by parts give
$$
\begin{aligned}
\int_S f(\omega)\partial_i\psi(\omega)\,d\omega
&=
\int_{S_{-i}}
\left(\prod_{j\neq i}f_j(\omega_j)\right)
\left(\int_{J_i}f_i(t)\partial_i\psi(\omega_{-i},t)\,dt\right)d\omega_{-i}
\\
&=
-
\int_{S_{-i}}
\left(\prod_{j\neq i}f_j(\omega_j)\right)
\left(\int_{J_i}\psi(\omega_{-i},t)\,d(Df_i)(t)\right)d\omega_{-i}.
\end{aligned}
$$
This proves \eqref{eq:appendix_product_first_derivative_measure}. Repeating the
same argument in the distinct coordinates indexed by $I$ proves
\eqref{eq:appendix_product_mixed_derivative_measure}.

For finite signed measures $\mu_1,\ldots,\mu_r$, polar decomposition gives
$$
|\mu_1\otimes\cdots\otimes\mu_r|
=
|\mu_1|\otimes\cdots\otimes|\mu_r|.
$$
Applying this identity and using normalization of the remaining marginal
measures yields \eqref{eq:appendix_product_coordinate_variation} and
\eqref{eq:appendix_product_mixed_derivative_mass}.
\end{proof}

For a pointwise function $h:J=[a,b]\to\mathbb R$, its
one-dimensional Jordan variation is
$$
\operatorname{Var}(h;J)
:=
\sup_{a=t_0<\cdots<t_n=b}
\sum_{r=0}^{n-1}|h(t_{r+1})-h(t_r)|.
$$

\begin{theorem}[Classical Vitali variation of a product]
\label{thm:appendix_classical_vitali_product}
For each $j$, let $h_j:J_j\to\mathbb R$ be a fixed pointwise representative
of finite one-dimensional Jordan variation, and define
$$
h(\omega):=\prod_{j=1}^d h_j(\omega_j).
$$
Then
\begin{equation}
\label{eq:appendix_classical_vitali_product}
V(h;S)
=
\prod_{j=1}^d\operatorname{Var}(h_j;J_j).
\end{equation}
\end{theorem}

\begin{proof}
For a rectangle $R=\prod_{j=1}^d[u_j,v_j]\subset S$, the alternating vertex
sum factorizes:
\begin{equation}
\label{eq:appendix_product_quasivolume}
\sigma(h,R)
=
\prod_{j=1}^d\bigl(h_j(v_j)-h_j(u_j)\bigr).
\end{equation}
For one-dimensional partitions $P_j$, the sum over the induced tensor grid is
therefore the product of the corresponding one-dimensional variation sums.
Taking partitions approaching the marginal variations proves the lower bound
in \eqref{eq:appendix_classical_vitali_product}.

For the reverse inequality, refine an arbitrary rectangular partition to the
tensor grid induced by all of its coordinate cuts. Splitting a rectangle in
coordinate $m$ replaces an increment $\Delta_m$ by
$\Delta_m^-+\Delta_m^+$. By the triangle inequality,
$$
|\Delta_m|\prod_{j\neq m}|\Delta_j|
\le
\bigl(|\Delta_m^-|+|\Delta_m^+|\bigr)
\prod_{j\neq m}|\Delta_j|.
$$
Thus refinement cannot decrease the sum of absolute quasivolumes. The grid
calculation then bounds every rectangular partition by the product of the
one-dimensional variations. Taking the supremum proves the result.
\end{proof}

\begin{corollary}[Smooth product-density formulas]
\label{cor:appendix_smooth_product_density_formulas}
If each $f_j$ is a smooth probability density, then
\begin{equation}
\label{eq:appendix_smooth_product_TVI}
\TV_I(f;S)
=
\sum_{i\in I}\int_{J_i}|f_i'(t)|\,dt,
\end{equation}
\begin{equation}
\label{eq:appendix_smooth_product_mixed}
\int_S|\partial_I f(\omega)|\,d\omega
=
\prod_{i\in I}\int_{J_i}|f_i'(t)|\,dt,
\end{equation}
and, for $I=[d]$,
\begin{equation}
\label{eq:appendix_smooth_product_vitali}
V(f;S)
=
\prod_{i=1}^d\int_{J_i}|f_i'(t)|\,dt.
\end{equation}
\end{corollary}

\begin{proof}
For smooth marginals, $Df_i=f_i'(t)\,dt$, so the first two identities follow
from Proposition~\ref{prop:appendix_product_distributional_derivatives}. The
last identity follows from the smooth Vitali formula and
$$
\partial_1\cdots\partial_d f(\omega)=\prod_{i=1}^d f_i'(\omega_i).
$$
\end{proof}

\begin{corollary}[Additive and multiplicative product-density bounds]
\label{cor:appendix_additive_multiplicative_bounds}
Let $f=\prod_jf_j$ be a smooth product density and let $\varphi$ be centered
on $I$. Then
$$
|\mathbb E_f[\varphi]|
\le
\frac1{2k}\Lambda_{\mathrm{sl}}^I(\varphi)
\sum_{i\in I}\int_{J_i}|f_i'(t)|\,dt
$$
for $\varphi\in L^\infty(S)$, whereas
$$
|\mathbb E_f[\varphi]|
\le
2^{-k}\Lambda_{\mathrm{mix}}^I(\varphi)
\prod_{i\in I}\int_{J_i}|f_i'(t)|\,dt
$$
for $\varphi\in L^\infty(S)$.
\end{corollary}

\begin{proof}
Substitute \eqref{eq:appendix_smooth_product_TVI} and
\eqref{eq:appendix_smooth_product_mixed} into the exact-centered bounds from
the main text.
\end{proof}

\begin{remark}[Representative dependence]
\label{rem:appendix_representative_distinction}
Theorem~\ref{thm:appendix_classical_vitali_product} concerns fixed pointwise
representatives. Proposition~\ref{prop:appendix_product_distributional_derivatives}
concerns $L^1$ equivalence classes. For smooth marginals the distinction
disappears. For arbitrary $BV$ classes, changing a function at one point can
change a corner quasivolume without changing any distributional derivative
measure.
\end{remark}

\subsection{Two regimes separating the variation certificates}
\label{subsec:appendix_branch_separation}

The additive and multiplicative product formulas show that the two variation
certificates cannot be ordered uniformly. The following smooth examples make
both possible orderings explicit. Throughout this subsection, let $k\ge2$,
let $S=[0,1]^k$, and take $I=[k]$.

\begin{example}[Small marginal variation]
\label{ex:appendix_mixed_certificate_better}
Let
$$
\varphi(x)
:=
\prod_{i=1}^k\sin(2\pi x_i),
$$
and, for $0<\varepsilon<\pi/4$, let
$$
f_\varepsilon(x)
:=
\prod_{i=1}^k
\left(1+\varepsilon\sin(2\pi x_i)\right).
$$
Then $f_\varepsilon$ is a smooth product density and $\varphi$ is centered in
every direction. Moreover,
$$
\Lambda_{\mathrm{sl}}^{[k]}(\varphi)=\frac2\pi,
\qquad
\Lambda_{\mathrm{mix}}^{[k]}(\varphi)
=\left(\frac2\pi\right)^k,
$$
and every marginal has variation
$$
\int_0^1
\left|\frac{d}{dt}
\left(1+\varepsilon\sin(2\pi t)\right)
\right|dt
=4\varepsilon.
$$
Consequently, the anisotropic certificate equals
$$
\frac{1}{2k}\frac2\pi(4k\varepsilon)
=\frac{4\varepsilon}{\pi},
$$
whereas the mixed-derivative certificate equals
$$
2^{-k}\left(\frac2\pi\right)^k(4\varepsilon)^k
=\left(\frac{4\varepsilon}{\pi}\right)^k.
$$
Since $4\varepsilon/\pi<1$ and $k\ge2$, the mixed-derivative certificate is
strictly smaller. The actual expectation is
$$
\mathbb E_{f_\varepsilon}[\varphi]
=\left(\frac\varepsilon2\right)^k.
$$
\end{example}

\begin{example}[Large oscillatory marginal variation]
\label{ex:appendix_tv_certificate_better}
Fix an integer $m\ge1$ and $0<a<1$, and define
$$
\varphi_m(x)
:=
\prod_{i=1}^k\sin(2\pi m x_i),
$$
$$
f_{a,m}(x)
:=
\prod_{i=1}^k
\left(1+a\sin(2\pi m x_i)\right).
$$
Again, $f_{a,m}$ is a smooth product density and $\varphi_m$ is centered in
every direction. The residual envelopes are unchanged:
$$
\Lambda_{\mathrm{sl}}^{[k]}(\varphi_m)=\frac2\pi,
\qquad
\Lambda_{\mathrm{mix}}^{[k]}(\varphi_m)
=\left(\frac2\pi\right)^k.
$$
Every marginal now has variation $4am$. Hence the anisotropic certificate is
$$
\frac{4am}{\pi},
$$
while the mixed-derivative certificate is
$$
\left(\frac{4am}{\pi}\right)^k.
$$
Whenever $4am/\pi>1$, the anisotropic certificate is strictly smaller for
$k\ge2$. The actual expectation is
$$
\mathbb E_{f_{a,m}}[\varphi_m]
=\left(\frac a2\right)^k.
$$
\end{example}

These examples separate the certificates, not the underlying regularity
classes. Both densities and residuals are smooth. The difference is caused by
the additive-versus-multiplicative dependence of the two bounds.

\subsection{A one-dimensional smooth near-extremizing pair}
\label{subsec:appendix_one_dimensional_near_extremizer}

Let $J=[a,b]$, choose
$$
0<w<b-a,
\qquad
m:=b-\frac w2,
\qquad
0<\delta<\frac w4,
$$
and define
$$
h_{w,\delta}(r)
:=
\begin{cases}
r/\delta, & 0\le r\le\delta,\\[1mm]
1, & \delta<r<w/2-\delta,\\[1mm]
(w/2-r)/\delta, & w/2-\delta\le r\le w/2.
\end{cases}
$$
Set
\begin{equation}
\label{eq:appendix_definition_psi_w_delta}
\psi_{w,\delta}(t)
:=
\begin{cases}
-h_{w,\delta}(m-t), & m-w/2\le t\le m,\\[1mm]
\hphantom{-}h_{w,\delta}(t-m), & m\le t\le m+w/2,\\[1mm]
0, & |t-m|>w/2.
\end{cases}
\end{equation}
Thus $\psi_{w,\delta}$ is supported on $[b-w,b]$ and is odd about $m$.

For the smooth density, define
$$
\vartheta(r)
:=
\begin{cases}
0, & r\le0,\\[1mm]
\exp(-1/r), & r>0,
\end{cases}
$$
$$
H(s)
:=
\frac{\vartheta(s+1)}
{\vartheta(s+1)+\vartheta(1-s)}.
$$
Then $H\in C^\infty(\mathbb R)$,
$$
H=0\text{ on }(-\infty,-1],
\qquad
H=1\text{ on }[1,\infty),
$$
$$
H'\ge0,
\qquad
H(-s)=1-H(s).
$$
For $0<\varepsilon<w/2$, define
\begin{equation}
\label{eq:appendix_definition_g_w_epsilon}
g_{w,\varepsilon}(t)
:=
\frac2wH\left(\frac{t-m}{\varepsilon}\right).
\end{equation}

\begin{lemma}[One-dimensional smooth near-extremizer]
\label{lem:appendix_one_dimensional_near_extremizer}
The functions above satisfy:

\begin{enumerate}
\item[\emph{(i)}]
$\psi_{w,\delta}\in W^{1,\infty}(J)$,
$$
\int_J\psi_{w,\delta}(t)\,dt=0,
\qquad
\|\psi_{w,\delta}\|_{L^\infty(J)}=1,
$$
\begin{equation}
\label{eq:appendix_psi_L1}
\|\psi_{w,\delta}\|_{L^1(J)}=w-2\delta,
\end{equation}
$$
\int_m^b\psi_{w,\delta}(t)\,dt
=\frac w2-\delta.
$$

\item[\emph{(ii)}]
$g_{w,\varepsilon}$ is the restriction of a $C^\infty$ function, is
nonnegative, and
$$
\int_Jg_{w,\varepsilon}(t)\,dt=1,
\qquad
\int_J|g_{w,\varepsilon}'(t)|\,dt=\frac2w.
$$

\item[\emph{(iii)}] If
$$
g_{w,0}(t):=\frac2w\mathbf 1_{(m,b]}(t),
$$
then
$$
\|g_{w,\varepsilon}-g_{w,0}\|_{L^1(J)}
\le
\frac{4\varepsilon}{w},
$$
and
\begin{equation}
\label{eq:appendix_one_dimensional_pairing_limit}
\int_J\psi_{w,\delta}(t)g_{w,\varepsilon}(t)\,dt
\longrightarrow
1-\frac{2\delta}{w}
\qquad
\text{as }\varepsilon\downarrow0.
\end{equation}
\end{enumerate}
\end{lemma}

\begin{proof}
Oddness gives the zero mean. Moreover,
$$
\int_0^{w/2}h_{w,\delta}(r)\,dr
=
\frac{\delta}{2}
+
\left(\frac w2-2\delta\right)
+
\frac{\delta}{2}
=
\frac w2-\delta,
$$
which proves the $L^1$ and positive-half identities.

The symmetry $H(s)+H(-s)=1$ gives
$$
\int_{-1}^1H(s)\,ds=1.
$$
Since $H((t-m)/\varepsilon)$ is zero before $m-\varepsilon$, one after
$m+\varepsilon$, and $b-m=w/2$,
$$
\int_JH\left(\frac{t-m}{\varepsilon}\right)dt
=
\varepsilon\int_{-1}^1H(s)\,ds
+
\left(\frac w2-\varepsilon\right)
=
\frac w2.
$$
Thus $g_{w,\varepsilon}$ has mass one. It is nondecreasing, so
$$
\int_J|g_{w,\varepsilon}'|
=
g_{w,\varepsilon}(b)-g_{w,\varepsilon}(a)
=
\frac2w.
$$
The smooth density and the half-window density differ only on
$[m-\varepsilon,m+\varepsilon]$, where both are bounded by $2/w$. Hence
$$
\|g_{w,\varepsilon}-g_{w,0}\|_1
\le
2\varepsilon\frac2w.
$$
Finally,
$$
\int_J\psi_{w,\delta}g_{w,0}
=
\frac2w\int_m^b\psi_{w,\delta}
=
1-\frac{2\delta}{w},
$$
and the pairing limit follows from the $L^1$ estimate.
\end{proof}

\begin{remark}[Relative variation]
\label{rem:appendix_smooth_extremizer_internal_variation}
The density rises from zero to $2/w$ across one smooth transition and remains
constant up to $b$. Its variation on $J$ is therefore $2/w$. No boundary jump
at $b$ is counted because the variation is relative to the interval.
\end{remark}

\subsection{Sharpness of the
\texorpdfstring{$k$}{k}-direction anisotropic
\texorpdfstring{$\TV_\infty$}{TV-infinity} constant}
\label{subsec:appendix_k_direction_tv_sharpness}

We now show that the coefficient $1/(2k)$ in
Theorem~\ref{thm:k_direction_tv_bound} cannot be reduced. The construction
tensorizes the one-dimensional near-extremizing pair from
Lemma~\ref{lem:appendix_one_dimensional_near_extremizer}.

Choose
$$
0<w<\min_{i\in I}\ell_i,
\qquad
0<\delta<\frac{w}{4},
\qquad
0<\varepsilon<\frac{w}{2}.
$$
For each $i\in I$, translate the one-dimensional functions
$\psi_{w,\delta}$ and $g_{w,\varepsilon}$ to the interval $J_i$, and denote
the translated functions by
$\psi_{w,\delta}^{(i)}$ and $g_{w,\varepsilon}^{(i)}$, respectively.
For every $j\notin I$, let
$$
\nu_j(t):=\ell_j^{-1},
\qquad t\in J_j,
$$
be the uniform probability density on $J_j$. Define
\begin{equation}
\label{eq:appendix_tensorized_tv_sharpness_pair}
f_{w,\varepsilon}^I(\omega)
:=
\left(
\prod_{i\in I}
g_{w,\varepsilon}^{(i)}(\omega_i)
\right)
\left(
\prod_{j\notin I}
\nu_j(\omega_j)
\right),
\end{equation}
and
\begin{equation}
\label{eq:appendix_tensorized_tv_sharpness_residual}
\varphi_{w,\delta}^I(\omega)
:=
\prod_{i\in I}
\psi_{w,\delta}^{(i)}(\omega_i).
\end{equation}
Each selected residual factor has zero mean. Hence
$\varphi_{w,\delta}^I$ is centered on $I$. Moreover,
$f_{w,\varepsilon}^I$ is a smooth probability density on $S$.

\begin{theorem}[Sharpness of the $k$-direction anisotropic bound]
\label{thm:appendix_k_direction_tv_sharpness}
Let $\emptyset\neq I\subseteq[d]$ and let $k:=|I|$. Then
\begin{equation}
\label{eq:appendix_tv_sharpness_supremum}
\sup
\left\{
\frac{|\mathbb E_f[\varphi]|}
{\Lambda_{\mathrm{sl}}^I(\varphi)\,\TV_I(f;S)}
:
\begin{array}{l}
f\in C^\infty(\overline S),\
f\ge0,\
\displaystyle\int_S f=1,
\\[1mm]
\varphi\in W^{1,\infty}(S),\
\varphi\text{ centered on }I,
\\[1mm]
\Lambda_{\mathrm{sl}}^I(\varphi)\,\TV_I(f;S)>0
\end{array}
\right\}
=
\frac{1}{2k}.
\end{equation}
Here $C^\infty(\overline S)$ denotes smooth extendability to an open
neighborhood of $S$.
\end{theorem}

\begin{proof}
Theorem~\ref{thm:k_direction_tv_bound} gives
$$
\frac{|\mathbb E_f[\varphi]|}
{\Lambda_{\mathrm{sl}}^I(\varphi)\,\TV_I(f;S)}
\le
\frac{1}{2k}
$$
for every admissible pair. It therefore remains to construct a sequence of
admissible pairs for which the ratio approaches $1/(2k)$.

Consider the pair
$(\varphi_{w,\delta}^I,f_{w,\varepsilon}^I)$ defined in
\eqref{eq:appendix_tensorized_tv_sharpness_pair} and
\eqref{eq:appendix_tensorized_tv_sharpness_residual}.
By Proposition~\ref{prop:appendix_tensorized_projectors_envelopes} and
Lemma~\ref{lem:appendix_one_dimensional_near_extremizer},
$$
\Lambda_{\mathrm{sl}}^I
\bigl(\varphi_{w,\delta}^I\bigr)
=
w-2\delta.
$$
Indeed, every selected residual factor has $L^\infty$ norm one and
$L^1$ norm $w-2\delta$.

For every $i\in I$, the selected marginal density satisfies
$$
\int_{J_i}
\left|
\frac{d}{dt}
g_{w,\varepsilon}^{(i)}(t)
\right|dt
=
\frac{2}{w}.
$$
The remaining marginal densities are constant. The product-density formula
therefore gives
$$
\TV_I(f_{w,\varepsilon}^I;S)
=
\sum_{i\in I}
\int_{J_i}
\left|
\frac{d}{dt}
g_{w,\varepsilon}^{(i)}(t)
\right|dt
=
\frac{2k}{w}.
$$

Let
$$
q_{w,\delta,\varepsilon}
:=
\int_{J_i}
\psi_{w,\delta}^{(i)}(t)
g_{w,\varepsilon}^{(i)}(t)\,dt.
$$
The translated one-dimensional pairs are identical up to translation, so this
quantity is independent of $i\in I$. Tensorization gives
$$
\mathbb E_{f_{w,\varepsilon}^I}
\bigl[\varphi_{w,\delta}^I\bigr]
=
q_{w,\delta,\varepsilon}^{\,k}.
$$
Consequently,
\begin{equation}
\label{eq:appendix_tv_sharpness_ratio}
\frac{
\left|
\mathbb E_{f_{w,\varepsilon}^I}
\bigl[\varphi_{w,\delta}^I\bigr]
\right|
}
{
\Lambda_{\mathrm{sl}}^I
\bigl(\varphi_{w,\delta}^I\bigr)
\TV_I(f_{w,\varepsilon}^I;S)
}
=
\frac{
|q_{w,\delta,\varepsilon}|^k
}
{
(w-2\delta)(2k/w)
}.
\end{equation}

Lemma~\ref{lem:appendix_one_dimensional_near_extremizer} gives
$$
q_{w,\delta,\varepsilon}
\longrightarrow
1-\frac{2\delta}{w}
\qquad
\text{as }\varepsilon\downarrow0.
$$
Taking the limit in \eqref{eq:appendix_tv_sharpness_ratio} yields
\begin{align*}
\lim_{\varepsilon\downarrow0}
\frac{
\left|
\mathbb E_{f_{w,\varepsilon}^I}
\bigl[\varphi_{w,\delta}^I\bigr]
\right|
}
{
\Lambda_{\mathrm{sl}}^I
\bigl(\varphi_{w,\delta}^I\bigr)
\TV_I(f_{w,\varepsilon}^I;S)
}
&=
\frac{
\left(1-\frac{2\delta}{w}\right)^k
}
{
2k\left(1-\frac{2\delta}{w}\right)
}
\\
&=
\frac{1}{2k}
\left(
1-\frac{2\delta}{w}
\right)^{k-1}.
\end{align*}
Finally, letting $\delta/w\downarrow0$ gives
$$
\frac{1}{2k}
\left(
1-\frac{2\delta}{w}
\right)^{k-1}
\longrightarrow
\frac{1}{2k}.
$$
The upper bound is therefore attained in the limit. Hence the coefficient
$1/(2k)$ is optimal, even when the density is smooth and the centered residual
is Lipschitz.
\end{proof}

The geometry of the construction is shown in
Figure~\ref{fig:appendix_unified_sharpness_construction}. In one dimension,
the residual is odd about the midpoint of a terminal window, while the smooth
density approaches the normalized indicator of its upper half. Tensorization
produces an alternating sign pattern for the residual and concentrates the
product density near the corner on which all residual factors are positive.
The quantitative content of the construction is contained entirely in the
preceding proof.

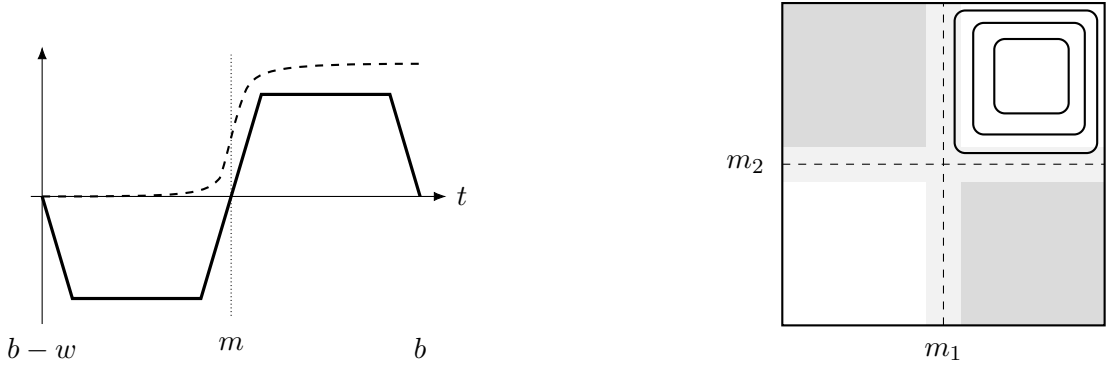
\begin{figure}[t]
\centering

\begin{subfigure}[t]{0.43\textwidth}
\centering
\begin{tikzpicture}[x=5.0cm,y=1.35cm,>=Latex]

\draw[->] (-0.03,0) -- (1.07,0) node[right] {$t$};
\draw[->] (0,-1.25) -- (0,1.47);

\draw[very thick]
(0,0)
-- (0.08,-1)
-- (0.42,-1)
-- (0.50,0)
-- (0.58,1)
-- (0.92,1)
-- (1,0);

\draw[thick,dashed]
(0,0)
.. controls (0.38,0) and (0.43,0.01) .. (0.47,0.18)
.. controls (0.49,0.38) and (0.51,0.80) .. (0.54,1.05)
.. controls (0.57,1.25) and (0.62,1.30) .. (1,1.30);

\draw[densely dotted]
(0.50,-1.17) -- (0.50,1.39);

\node[anchor=north] at (0,-1.28) {$b-w$};
\node[anchor=north] at (0.50,-1.28) {$m$};
\node[anchor=north] at (1,-1.28) {$b$};

\end{tikzpicture}
\caption{One-dimensional building blocks. The solid curve is the residual and
the dashed curve is the density.}
\label{fig:appendix_unified_sharpness_1d}
\end{subfigure}
\hfill
\begin{subfigure}[t]{0.49\textwidth}
\centering
\begin{tikzpicture}[scale=0.82,>=Latex]

\def\W{5.2}
\def\m{2.6}
\def\e{0.28}

\fill[gray!28] (0,\m) rectangle (\m,\W);
\fill[gray!28] (\m,0) rectangle (\W,\m);

\fill[gray!10] (\m-\e,0) rectangle (\m+\e,\W);
\fill[gray!10] (0,\m-\e) rectangle (\W,\m+\e);

\draw[thick] (0,0) rectangle (\W,\W);

\draw[dashed] (\m,0) -- (\m,\W);
\draw[dashed] (0,\m) -- (\W,\m);

\draw[thick,rounded corners=4pt]
(\m+0.18,\m+0.18) rectangle (\W-0.12,\W-0.12);
\draw[thick,rounded corners=4pt]
(\m+0.48,\m+0.48) rectangle (\W-0.32,\W-0.32);
\draw[thick,rounded corners=4pt]
(\m+0.82,\m+0.82) rectangle (\W-0.58,\W-0.58);

\node[anchor=north] at (\m,-0.12) {$m_1$};
\node[anchor=east] at (-0.12,\m) {$m_2$};

\end{tikzpicture}
\caption{Two-dimensional tensor geometry. The alternating shades represent
the two residual signs, and the nested curves indicate density concentration.}
\label{fig:appendix_unified_sharpness_2d}
\end{subfigure}

\caption{
Geometry of the tensor-product near-extremizing construction. The
one-dimensional pair is shown on the left, and its two-dimensional
tensorization is shown on the right. Quantitative identities and limiting
ratios are given in the proof of
Theorem~\ref{thm:appendix_k_direction_tv_sharpness}.
}
\label{fig:appendix_unified_sharpness_construction}
\end{figure}

The same family will be used in
Theorem~\ref{thm:appendix_k_direction_mixed_sharpness}. The difference between
the two sharpness arguments lies only in the density-variation scaling:
first-order coordinate variations add over the selected directions, whereas
the selected mixed derivative tensorizes multiplicatively.

\subsection{Sharpness of the mixed-derivative and Vitali constants}
\label{subsec:appendix_mixed_vitali_sharpness}

\begin{theorem}[Sharpness of the $k$-direction mixed-derivative constant]
\label{thm:appendix_k_direction_mixed_sharpness}
Let $I\subseteq[d]$ be nonempty and $k=|I|$. Then
\begin{equation}
\label{eq:appendix_mixed_sharpness_supremum}
\sup
\left\{
\frac{
\left|\int_S\varphi(\omega)f(\omega)\,d\omega\right|
}
{
\Lambda_{\mathrm{mix}}^I(\varphi)
\int_S|\partial_I f(\omega)|\,d\omega
}:
\begin{array}{l}
f\in C^\infty(\overline S),\ f\ge0,\ \int_Sf=1,\\
\varphi\in W^{1,\infty}(S),\ \varphi\text{ centered on }I,\\
\Lambda_{\mathrm{mix}}^I(\varphi)\int_S|\partial_I f|>0
\end{array}
\right\}
=
2^{-k}.
\end{equation}
\end{theorem}

\begin{proof}
The exact-centered mixed-derivative theorem gives the upper bound. For the
same tensorized pair,
$$
\Lambda_{\mathrm{mix}}^I(\varphi_{w,\delta}^I)
=(w-2\delta)^k,
$$
$$
\int_S|\partial_I f_{w,\varepsilon}^I(\omega)|\,d\omega
=\left(\frac2w\right)^k,
$$
and
$$
\int_S\varphi_{w,\delta}^If_{w,\varepsilon}^I
=q_{w,\delta,\varepsilon}^k.
$$
Therefore
$$
\frac{
\left|\int_S\varphi_{w,\delta}^If_{w,\varepsilon}^I\right|
}
{
\Lambda_{\mathrm{mix}}^I(\varphi_{w,\delta}^I)
\int_S|\partial_I f_{w,\varepsilon}^I|
}
=
\frac{|q_{w,\delta,\varepsilon}|^k}
{(w-2\delta)^k(2/w)^k}.
$$
Letting $\varepsilon\downarrow0$ gives
$$
\frac{(1-2\delta/w)^k}
{(w-2\delta)^k(2/w)^k}
=
2^{-k}.
$$
Thus the upper constant is approached by smooth product densities and
Lipschitz tensor-product residuals.
\end{proof}

\begin{corollary}[Sharpness of the all-axes Vitali constant]
\label{cor:appendix_vitali_sharpness}
For $I=[d]$,
\begin{equation}
\label{eq:appendix_vitali_sharpness_supremum}
\sup
\left\{
\frac{
\left|\int_S\varphi(\omega)f(\omega)\,d\omega\right|
}
{\|\varphi\|_{L^1(S)}V(f;S)}:
\begin{array}{l}
f\in C^\infty(\overline S),\ f\ge0,\ \int_Sf=1,\\
\varphi\in W^{1,\infty}(S),\ \varphi\text{ centered on }[d],\\
\|\varphi\|_{L^1(S)}V(f;S)>0
\end{array}
\right\}
=
2^{-d}.
\end{equation}
Hence the coefficient $2^{-d}$ in the smooth all-axes Vitali bound is sharp.
\end{corollary}

\begin{proof}
For $I=[d]$,
$$
\Lambda_{\mathrm{mix}}^{[d]}(\varphi)=\|\varphi\|_{L^1(S)},
$$
and, for smooth $f$,
$$
V(f;S)=\int_S|\partial_1\cdots\partial_df(\omega)|\,d\omega.
$$
The result is therefore the all-axes specialization of
Theorem~\ref{thm:appendix_k_direction_mixed_sharpness}.
\end{proof}

\begin{remark}[Sharpness of the defect coefficient]
\label{rem:appendix_defect_coefficient_sharp}
The coefficient one in front of the projector defect is also optimal. If
$\varphi\equiv c\neq0$, then, for every nonempty $I$,
$$
\mathcal P_I\varphi=0,
\qquad
\mathcal R_I\varphi=\varphi,
$$
and every probability density satisfies
$$
|\mathbb E_f[\varphi]|
=|c|
=\|\mathcal R_I\varphi\|_{L^\infty(S)}.
$$
\end{remark}

\begin{remark}[Scope of the sharpness construction]
\label{rem:appendix_scope_sharpness_construction}
The construction proves sharpness for the internal smooth box quantities
$\TV_I(f;S)$ and $\int_S|\partial_If|$. It does not assert sharpness of the
extension-seminorm estimate involving $V_I^{\mathrm{ext}}(f;S)$. The defect
example is separate and establishes optimality of its coefficient without
claiming simultaneous equality in both parts of a defect-adjusted bound.
\end{remark}

\subsection{Mollification and the extension mixed-variation seminorm}
\label{subsec:appendix_mollification_extension}

The anisotropic $\TV_\infty$ formulation is already formulated at the $BV$ level.
The role of mollification is specific to the mixed-derivative formulation. The
quantity stable under convolution is the distributional mixed derivative
measure $D_I f$, not the representative-dependent corner formula for classical
Vitali variation. We use the standard approximate-identity and $BV$
convolution properties; see, for example, \cite{evans,AmbrosioFuscoPallara2000}. Figures~\ref{fig:appendix_mollification_step_density} and
\ref{fig:appendix_tensor_mollification_geometry} illustrate the two points that
are load-bearing in the proof. First, convolution replaces a singular
mixed-derivative measure by a smooth density without increasing its total
variation mass. Second, smoothing is carried out for an extension on the whole
space. Consequently, global mass is preserved, but the restriction of the
mollified extension to $S$ need not remain normalized on $S$.

Let $\eta\in C_c^\infty(\mathbb R^d)$ satisfy
$$
\eta\ge0,
\qquad
\int_{\mathbb R^d}\eta(x)\,dx=1,
$$
and set
$$
\eta_\varepsilon(x):=\varepsilon^{-d}\eta(x/\varepsilon),
\qquad
h_\varepsilon:=h*\eta_\varepsilon.
$$

\begin{theorem}[Mollification of a finite mixed-derivative measure]
\label{thm:appendix_mollification_mixed_measure}
Let $h\in L^1(\mathbb R^d)$ and assume that $D_Ih$ is a finite signed Radon
measure. Then:

\begin{enumerate}
\item[\emph{(i)}]
$h_\varepsilon\in C^\infty(\mathbb R^d)$ and
$$
h_\varepsilon\to h
\qquad
\text{in }L^1(\mathbb R^d).
$$

\item[\emph{(ii)}]
\begin{equation}
\label{eq:appendix_mollification_commutes_mixed}
\partial_Ih_\varepsilon=(D_Ih)*\eta_\varepsilon,
\end{equation}
and
\begin{equation}
\label{eq:appendix_mollification_mixed_contraction}
\int_{\mathbb R^d}|\partial_Ih_\varepsilon(x)|\,dx
\le
|D_Ih|(\mathbb R^d).
\end{equation}

\item[\emph{(iii)}] If $h\ge0$, then $h_\varepsilon\ge0$ and
$$
\int_{\mathbb R^d}h_\varepsilon
=
\int_{\mathbb R^d}h.
$$
\end{enumerate}
\end{theorem}

\begin{proof}
Smoothness and $L^1$ convergence are standard approximate-identity facts.
Convolution commutes with distributional differentiation, which proves
\eqref{eq:appendix_mollification_commutes_mixed}. Pointwise,
$$
\partial_Ih_\varepsilon(x)
=
\int_{\mathbb R^d}\eta_\varepsilon(x-y)\,d(D_Ih)(y),
$$
so
$$
|\partial_Ih_\varepsilon(x)|
\le
\int_{\mathbb R^d}\eta_\varepsilon(x-y)\,d|D_Ih|(y).
$$
Fubini--Tonelli and $\int\eta_\varepsilon=1$ give the contraction. Positivity
and preservation of mass follow in the same way.
\end{proof}

\begin{example}[A step density and its global mollification]
\label{ex:appendix_step_density_mollification}
Let $S=[0,1]$ and consider the probability density
$$
f(t):=2\mathbf 1_{[1/2,1]}(t),
\qquad t\in S.
$$
Take its zero extension
$$
\bar f(t):=2\mathbf 1_{[1/2,1]}(t),
\qquad t\in\mathbb R.
$$
Then
$$
D\bar f=2\delta_{1/2}-2\delta_1,
\qquad
|D\bar f|(\mathbb R)=4,
$$
whereas the relative variation inside $S$ is only
$$
|Df|(S)=2.
$$
For $\bar f_\varepsilon=\bar f*\eta_\varepsilon$, one has
$$
\bar f_\varepsilon\in C^\infty(\mathbb R),
\qquad
\bar f_\varepsilon\ge0,
\qquad
\int_{\mathbb R}\bar f_\varepsilon=1,
$$
$$
\bar f_\varepsilon\to\bar f
\quad\text{in }L^1(\mathbb R),
\qquad
\int_{\mathbb R}|\bar f_\varepsilon'(t)|\,dt\le4.
$$
If the mollifier is positive on a neighborhood of the origin, then for every
sufficiently small $\varepsilon>0$ some mass is transported across the boundary
point $t=1$. Hence
$$
\int_0^1\bar f_\varepsilon(t)\,dt<1,
$$
even though global mass is preserved. This is the simplest example showing why
the proof of Theorem~\ref{thm:appendix_distributional_extension_bound} uses
the analytic scope of
Theorem~\ref{thm:k_direction_mixed_derivative_bound}; see
Remark~\ref{rem:analytic_scope_mixed_derivative}. That estimate does not
require the mollified restriction to be a probability density on $S$.
\end{example}

\begin{figure}[H]
\centering
\begin{subfigure}[t]{0.49\linewidth}
\centering
\begin{tikzpicture}
\begin{axis}[
  width=0.94\linewidth,
  height=5.0cm,
  axis lines=left,
  xmin=-0.18,
  xmax=1.20,
  ymin=-0.10,
  ymax=2.45,
  xtick={0,0.5,1},
  xticklabels={$0$,$1/2$,$1$},
  ytick={0,1,2},
  xlabel={$t$},
  ylabel={density},
  clip=false,
  legend style={draw=none,at={(0.03,0.97)},anchor=north west,font=\small}
]
\addplot[very thick,const plot] coordinates {
  (-0.18,0) (0.5,0) (0.5,2) (1,2) (1,0) (1.20,0)
};
\addlegendentry{$\bar f=2\mathbf 1_{[1/2,1]}$}
\addplot[thick,dashed,smooth] coordinates {
  (-0.18,0) (0.36,0) (0.43,0.10) (0.47,0.48) (0.50,1.00)
  (0.53,1.52) (0.57,1.90) (0.64,2.00) (0.86,2.00)
  (0.93,1.90) (0.97,1.52) (1.00,1.00) (1.03,0.48)
  (1.07,0.10) (1.14,0)
};
\addlegendentry{schematic $\bar f_\varepsilon$}
\addplot[densely dotted] coordinates {(0,-0.10) (0,2.25)};
\addplot[densely dotted] coordinates {(1,-0.10) (1,2.25)};
\node[anchor=south,align=center] at (axis cs:1.08,0.25)
{mass outside\\$S=[0,1]$};
\draw[->] (axis cs:1.08,0.22) -- (axis cs:1.035,0.48);
\end{axis}
\end{tikzpicture}
\caption{Global smoothing and boundary spillover.}
\label{fig:appendix_mollification_step_density_profile}
\end{subfigure}
\hfill
\begin{subfigure}[t]{0.49\linewidth}
\centering
\begin{tikzpicture}
\begin{axis}[
  width=0.94\linewidth,
  height=5.0cm,
  axis lines=middle,
  xmin=0.30,
  xmax=1.14,
  ymin=-2.65,
  ymax=2.65,
  xtick={0.5,1},
  xticklabels={$1/2$,$1$},
  ytick={-2,0,2},
  xlabel={$t$},
  clip=false
]
\draw[->,very thick] (axis cs:0.5,0) -- (axis cs:0.5,2.15);
\draw[->,very thick] (axis cs:1,0) -- (axis cs:1,-2.15);
\node[anchor=south] at (axis cs:0.5,2.15) {$2\delta_{1/2}$};
\node[anchor=north] at (axis cs:1,-2.15) {$-2\delta_1$};
\addplot[thick,dashed,smooth] coordinates {
  (0.39,0) (0.43,0.10) (0.46,0.70) (0.50,1.75)
  (0.54,0.70) (0.57,0.10) (0.61,0)
};
\addplot[thick,dashed,smooth] coordinates {
  (0.89,0) (0.93,-0.10) (0.96,-0.70) (1.00,-1.75)
  (1.04,-0.70) (1.07,-0.10) (1.11,0)
};
\node[align=center] at (axis cs:0.75,1.48)
{$(D\bar f)*\eta_\varepsilon$};
\end{axis}
\end{tikzpicture}
\caption{A singular derivative becomes a smooth signed density.}
\label{fig:appendix_mollification_step_density_derivative}
\end{subfigure}
\caption{One-dimensional illustration of Theorem~\ref{thm:appendix_mollification_mixed_measure}. The solid step in panel~\textup{(a)} is the whole-space extension $\bar f=2\mathbf 1_{[1/2,1]}$. Convolution rounds both jumps and preserves the total mass on $\mathbb R$, but part of the mollified mass lies outside $S=[0,1]$. Panel~\textup{(b)} shows the corresponding distributional identity $D\bar f=2\delta_{1/2}-2\delta_1$ and its convolution with the mollifier. The theorem uses only the contraction $\|\bar f_\varepsilon'\|_{L^1(\mathbb R)}\le |D\bar f|(\mathbb R)=4$; it does not require equality, disjoint positive and negative bumps, or normalization of $\bar f_\varepsilon|_S$.}
\label{fig:appendix_mollification_step_density}
\end{figure}
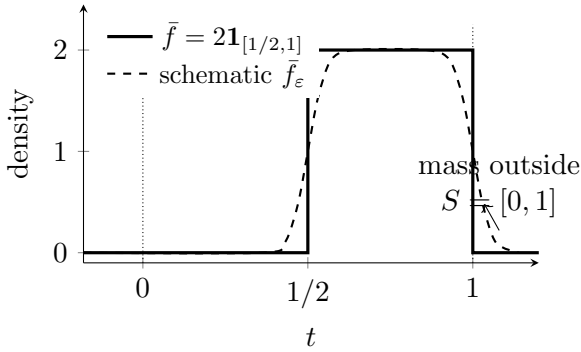
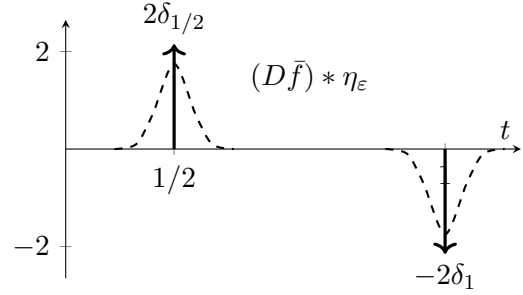

Figure~\ref{fig:appendix_mollification_step_density} separates the three
operations that are sometimes conflated. The extension $f\mapsto\bar f$ may
create boundary derivative mass. The convolution $\bar f\mapsto\bar
f_\varepsilon$ smooths this measure and does not increase its total variation.
The restriction $\bar f_\varepsilon\mapsto\bar f_\varepsilon|_S$ is then used
only inside the analytic box estimate. One notices that the contraction is a
global statement on $\mathbb R$, whereas the expectation is taken on $S$.
This is why the admissible-extension cost in
$V_I^{\mathrm{ext}}(f;S)$ can be strictly larger than the internal mixed
variation of $f$ on $S$.

\begin{proposition}[Tensor-product mollification]
\label{prop:appendix_tensor_product_mollification}
Let
$$
\eta_\varepsilon^\otimes(x)
:=
\prod_{j=1}^d\eta_\varepsilon^{(j)}(x_j)
$$
be a tensor product of one-dimensional mollifiers. If
$h(x)=\prod_{j=1}^dh_j(x_j)$ with $h_j\in L^1(\mathbb R)$, then
\begin{equation}
\label{eq:appendix_tensor_mollifier_factorization}
h*\eta_\varepsilon^\otimes
=
\prod_{j=1}^d(h_j*\eta_\varepsilon^{(j)}).
\end{equation}
\end{proposition}

\begin{proof}
Apply Fubini's theorem to the product integrand in the convolution.
\end{proof}

\begin{example}[A two-dimensional tensor-product mollification]
\label{ex:appendix_tensor_product_mollification}
Let
$$
g(t):=2\mathbf 1_{[1/2,1]}(t),
\qquad
\bar f(x_1,x_2):=g(x_1)g(x_2).
$$
Then $\bar f$ is a probability density on $\mathbb R^2$, supported on
$[1/2,1]^2$, and
$$
D_{\{1,2\}}\bar f=Dg\otimes Dg.
$$
Since
$$
Dg=2\delta_{1/2}-2\delta_1,
$$
one obtains
$$
D_{\{1,2\}}\bar f
=
4\bigl(
\delta_{(1/2,1/2)}
-
\delta_{(1/2,1)}
-
\delta_{(1,1/2)}
+
\delta_{(1,1)}
\bigr),
$$
and therefore
$$
|D_{\{1,2\}}\bar f|(\mathbb R^2)=16.
$$
For a tensor-product mollifier,
$$
\bar f_\varepsilon
=
\bar f*\eta_\varepsilon^\otimes
=
(g*\eta_\varepsilon^{(1)})(x_1)
(g*\eta_\varepsilon^{(2)})(x_2),
$$
while
$$
\partial_1\partial_2\bar f_\varepsilon
=
(D_{\{1,2\}}\bar f)*\eta_\varepsilon^\otimes,
\qquad
\int_{\mathbb R^2}
|\partial_1\partial_2\bar f_\varepsilon|
\le16.
$$
\end{example}

\begin{figure}[H]
\centering

\begin{subfigure}[t]{0.30\textwidth}
\centering
\begin{tikzpicture}[scale=2.55,>=Latex]

\draw[thick] (0,0) rectangle (1,1);
\draw[->] (-0.04,0) -- (1.09,0) node[right] {$x_1$};
\draw[->] (0,-0.04) -- (0,1.09) node[above] {$x_2$};

\fill[gray!42] (0.5,0.5) rectangle (1,1);
\draw[thick] (0.5,0.5) rectangle (1,1);

\draw[dashed] (0.5,0) -- (0.5,1);
\draw[dashed] (0,0.5) -- (1,0.5);

\node[anchor=north,font=\scriptsize] at (0,0) {$0$};
\node[anchor=north,font=\scriptsize] at (0.5,0) {$1/2$};
\node[anchor=north,font=\scriptsize] at (1,0) {$1$};

\node[anchor=east,font=\scriptsize] at (0,0) {$0$};
\node[anchor=east,font=\scriptsize] at (0,0.5) {$1/2$};
\node[anchor=east,font=\scriptsize] at (0,1) {$1$};

\node[font=\small] at (0.75,0.75) {$4$};

\end{tikzpicture}
\caption{$\bar f=g\otimes g$.}
\label{fig:appendix_tensor_mollification_original}
\end{subfigure}
\hfill
\begin{subfigure}[t]{0.30\textwidth}
\centering
\begin{tikzpicture}[scale=2.55,>=Latex]

\draw[thick] (0,0) rectangle (1,1);
\draw[->] (-0.04,0) -- (1.09,0) node[right] {$x_1$};
\draw[->] (0,-0.04) -- (0,1.09) node[above] {$x_2$};

\draw[dashed] (0.5,0) -- (0.5,1);
\draw[dashed] (0,0.5) -- (1,0.5);

\fill[gray!10,rounded corners=7pt]
  (0.38,0.38) rectangle (1.04,1.04);
\fill[gray!18,rounded corners=7pt]
  (0.44,0.44) rectangle (1.02,1.02);
\fill[gray!27,rounded corners=6pt]
  (0.50,0.50) rectangle (1.00,1.00);
\fill[gray!38,rounded corners=5pt]
  (0.58,0.58) rectangle (0.98,0.98);
\fill[gray!50,rounded corners=4pt]
  (0.67,0.67) rectangle (0.96,0.96);

\draw[gray!70,rounded corners=7pt]
  (0.38,0.38) rectangle (1.04,1.04);
\draw[gray!70,rounded corners=6pt]
  (0.50,0.50) rectangle (1.00,1.00);
\draw[gray!70,rounded corners=4pt]
  (0.67,0.67) rectangle (0.96,0.96);

\node[anchor=north,font=\scriptsize] at (0,0) {$0$};
\node[anchor=north,font=\scriptsize] at (0.5,0) {$1/2$};
\node[anchor=north,font=\scriptsize] at (1,0) {$1$};

\node[anchor=east,font=\scriptsize] at (0,0) {$0$};
\node[anchor=east,font=\scriptsize] at (0,0.5) {$1/2$};
\node[anchor=east,font=\scriptsize] at (0,1) {$1$};

\end{tikzpicture}
\caption{$\bar f_\varepsilon
=\bar f*\eta_\varepsilon^\otimes$.}
\label{fig:appendix_tensor_mollification_smoothed}
\end{subfigure}
\hfill
\begin{subfigure}[t]{0.30\textwidth}
\centering
\begin{tikzpicture}[scale=2.55,>=Latex]

\draw[thick] (0,0) rectangle (1,1);
\draw[->] (-0.04,0) -- (1.09,0) node[right] {$x_1$};
\draw[->] (0,-0.04) -- (0,1.09) node[above] {$x_2$};

\draw[dashed] (0.5,0) -- (0.5,1);
\draw[dashed] (0,0.5) -- (1,0.5);

\draw[thick,fill=white] (0.5,0.5) circle (0.072);
\node[font=\small] at (0.5,0.5) {$+$};

\draw[thick,fill=white] (1,1) circle (0.072);
\node[font=\small] at (1,1) {$+$};

\draw[thick,fill=white] (0.5,1) circle (0.072);
\node[font=\small] at (0.5,1) {$-$};

\draw[thick,fill=white] (1,0.5) circle (0.072);
\node[font=\small] at (1,0.5) {$-$};

\node[anchor=north,font=\scriptsize] at (0,0) {$0$};
\node[anchor=north,font=\scriptsize] at (0.5,0) {$1/2$};
\node[anchor=north,font=\scriptsize] at (1,0) {$1$};

\node[anchor=east,font=\scriptsize] at (0,0) {$0$};
\node[anchor=east,font=\scriptsize] at (0,0.5) {$1/2$};
\node[anchor=east,font=\scriptsize] at (0,1) {$1$};

\end{tikzpicture}
\caption{$D_{\{1,2\}}\bar f$.}
\label{fig:appendix_tensor_mollification_mixed_derivative}
\end{subfigure}

\caption{
Tensor-product mollification geometry for
$\bar f=g\otimes g$, where
$g=2\mathbf 1_{[1/2,1]}$.
Panel~\textup{(a)} shows the nonsmooth product density, which equals $4$ on
$[1/2,1]^2$ and vanishes elsewhere. Panel~\textup{(b)} schematically shows its
mollification by a tensor-product kernel. Panel~\textup{(c)} shows the signed
atomic structure
$\displaystyle
D_{\{1,2\}}\bar f
=
4\bigl(
\delta_{(1/2,1/2)}
-\delta_{(1/2,1)}
-\delta_{(1,1/2)}
+\delta_{(1,1)}
\bigr).
$
After convolution, these four atoms become smooth signed bumps and
$\displaystyle
\partial_1\partial_2\bar f_\varepsilon
=
(D_{\{1,2\}}\bar f)*\eta_\varepsilon^\otimes,
\qquad
\|\partial_1\partial_2\bar f_\varepsilon\|_{L^1(\mathbb R^2)}
\le
|D_{\{1,2\}}\bar f|(\mathbb R^2)
=
16.
$
}
\label{fig:appendix_tensor_mollification_geometry}
\end{figure}
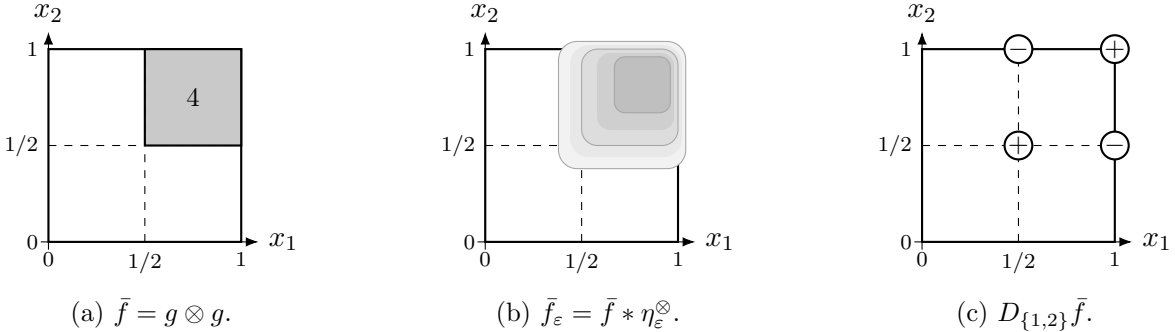

The figure separates the density, its smooth approximation, and the signed
mixed-derivative measure. The third panel represents the atoms of
$D_{\{1,2\}}\bar f$ before convolution, not the spatial profile of
$\partial_1\partial_2\bar f_\varepsilon$. The latter consists of four smooth
signed bumps centered at the displayed locations and may exhibit overlap and
cancelation.

Figure~\ref{fig:appendix_tensor_mollification_geometry} explains why the
product formulas and the mollification argument fit together. The nonsmooth
mixed derivative is a tensor product of signed one-dimensional jump measures.
A tensor-product mollifier smooths each factor separately, and therefore the
checkerboard sign pattern survives at the level of the smoothed mixed
derivative. The total variation estimate is nevertheless global and
sign-insensitive: cancellations may occur pointwise after the four blobs begin
to overlap, but convolution cannot increase the $L^1$ mass beyond the total
variation of the original measure. This is precisely the property required to
pass from the smooth mixed-derivative estimate to the distributional extension.

\begin{theorem}[Distributional extension of the mixed-derivative bound]
\label{thm:appendix_distributional_extension_bound}
Let $f\in L^1(S)$ be a probability density and let
$\varphi\in L^\infty(S)$ be centered on $I$. Assume that there exists
$\bar f\in L^1(\mathbb R^d)$ such that
$$
\bar f=f\quad\text{a.e. on }S,
$$
and $D_I\bar f$ is a finite signed Radon measure. Then
\begin{equation}
\label{eq:appendix_distributional_extension_bound}
\left|\int_S\varphi(\omega)f(\omega)\,d\omega\right|
\le
2^{-k}\Lambda_{\mathrm{mix}}^I(\varphi)
|D_I\bar f|(\mathbb R^d).
\end{equation}
\end{theorem}

\begin{proof}
Let $\bar f_\varepsilon=\bar f*\eta_\varepsilon$. By
Theorem~\ref{thm:appendix_mollification_mixed_measure},
$$
\int_{\mathbb R^d}|\partial_I\bar f_\varepsilon|
\le
|D_I\bar f|(\mathbb R^d).
$$
Applying Theorem~\ref{thm:k_direction_mixed_derivative_bound} on $S$
with $f$ replaced by $\bar f_\varepsilon|_S$ gives
$$
\left|\int_S\varphi\bar f_\varepsilon\right|
\le
2^{-k}\Lambda_{\mathrm{mix}}^I(\varphi)
|D_I\bar f|(\mathbb R^d).
$$
The restriction $\bar f_\varepsilon|_S$ need not be a probability density;
this is immaterial because the smooth estimate is analytic. Finally,
$\bar f_\varepsilon\to\bar f$ in $L^1$, and therefore
$$
\int_S\varphi\bar f_\varepsilon
\longrightarrow
\int_S\varphi f.
$$
\end{proof}

We use the extension mixed-variation seminorm $V_I^{\mathrm{ext}}(f;S)$ from Definition~\ref{def:main_extension_mixed_variation}.

\begin{proposition}[Basic properties and boundary-extension cost]
\label{prop:appendix_extension_mixed_variation_properties}
The map $V_I^{\mathrm{ext}}(\cdot;S)$ is an extended seminorm. Moreover:

\begin{enumerate}
\item[\emph{(i)}] If the internal derivative $D_If$ is a finite measure on
$\operatorname{int}S$, then
\begin{equation}
\label{eq:appendix_internal_mass_lower_bound}
|D_If|(\operatorname{int}S)
\le
V_I^{\mathrm{ext}}(f;S).
\end{equation}
For smooth $f$,
$$
\int_S|\partial_If(\omega)|\,d\omega
\le
V_I^{\mathrm{ext}}(f;S).
$$

\item[\emph{(ii)}] If $f\in C_c^k(\operatorname{int}S)$, then
\begin{equation}
\label{eq:appendix_compact_support_extension_equality}
V_I^{\mathrm{ext}}(f;S)
=
\int_S|\partial_If(\omega)|\,d\omega.
\end{equation}

\item[\emph{(iii)}] For $J=[a,b]$ and $f\equiv1$ on $J$,
\begin{equation}
\label{eq:appendix_constant_extension_variation}
V_{\{1\}}^{\mathrm{ext}}(1;J)=2,
\end{equation}
although the internal variation is zero.
\end{enumerate}
\end{proposition}

\begin{proof}
Homogeneity follows by scaling extensions. Subadditivity follows by adding
extensions and using the triangle inequality for total variation. If $\bar f$
extends $f$, then the restriction of $D_I\bar f$ to
$\operatorname{int}S$ is $D_If$, which proves
\eqref{eq:appendix_internal_mass_lower_bound}.

If $f\in C_c^k(\operatorname{int}S)$, its zero extension belongs to
$C_c^k(\mathbb R^d)$ and satisfies
$$
D_I\bar f=\partial_If(x)\,dx.
$$
This gives the upper bound in
\eqref{eq:appendix_compact_support_extension_equality}, while part \emph{(i)}
gives the reverse inequality.

For the constant function, the zero extension has derivative
$\delta_a-\delta_b$, so the extension cost is at most two. Conversely, every
$L^1(\mathbb R)$ extension with finite derivative belongs to
$BV(\mathbb R)\cap L^1(\mathbb R)$ and has limits zero at both infinities. It
must spend at least one unit of variation to rise from zero to one and at least
one unit to return to zero. Thus the cost is at least two.
\end{proof}

\begin{proposition}[A finite extension cost for product $BV$ densities]
\label{prop:appendix_product_bv_extension_cost}
Let
$$
f(\omega)=\prod_{j=1}^df_j(\omega_j),
$$
where every $f_j\in BV([a_j,b_j])$ is a nonnegative probability density. Let
$f_j(a_j+)$ and $f_j(b_j-)$ denote the one-sided traces. Then
\begin{equation}
\label{eq:appendix_product_bv_extension_cost}
V_I^{\mathrm{ext}}(f;S)
\le
\prod_{i\in I}
\left(
|Df_i|(J_i)+f_i(a_i+)+f_i(b_i-)
\right).
\end{equation}
If
$$
f_i(a_i+)=f_i(b_i-)=0
\qquad
\text{for every }i\in I,
$$
then
\begin{equation}
\label{eq:appendix_product_bv_extension_equality}
V_I^{\mathrm{ext}}(f;S)
=
\prod_{i\in I}|Df_i|(J_i).
\end{equation}
\end{proposition}

\begin{proof}
Let $\bar f_j$ be the zero extension of $f_j$ to $\mathbb R$. The standard
one-dimensional $BV$ extension formula \cite{AmbrosioFuscoPallara2000} gives
$$
D\bar f_j
=
\widetilde{Df_j}
+
f_j(a_j+)\delta_{a_j}
-
f_j(b_j-)\delta_{b_j},
$$
where $\widetilde{Df_j}$ denotes the extension by zero of the relative
derivative measure on $J_j$. Since the three measures are supported on
disjoint sets,
$$
|D\bar f_j|(\mathbb R)
=
|Df_j|(J_j)+f_j(a_j+)+f_j(b_j-).
$$
Define
$$
\bar f(x):=\prod_{j=1}^d\bar f_j(x_j).
$$
Then $\bar f\in L^1(\mathbb R^d)$, $\bar f=f$ a.e. on $S$, and the
whole-space product formula gives
$$
D_I\bar f
=
\left(\bigotimes_{i\in I}D\bar f_i\right)
\otimes
\left(\bigotimes_{j\notin I}\bar f_j(x_j)\,dx_j\right).
$$
Because every marginal has $L^1$ mass one,
$$
|D_I\bar f|(\mathbb R^d)
=
\prod_{i\in I}|D\bar f_i|(\mathbb R),
$$
which proves \eqref{eq:appendix_product_bv_extension_cost}. If the selected
traces vanish, this upper bound is
$\prod_{i\in I}|Df_i|(J_i)$. Proposition~\ref{prop:appendix_product_distributional_derivatives}
and the internal lower bound
\eqref{eq:appendix_internal_mass_lower_bound} give the reverse inequality.
\end{proof}

\begin{corollary}[Intrinsic nonsmooth mixed-derivative bound]
\label{cor:appendix_intrinsic_nonsmooth_mixed_bound}
If $f$ is a probability density, $\varphi\in L^\infty(S)$ is centered on $I$,
and $V_I^{\mathrm{ext}}(f;S)<\infty$, then
\begin{equation}
\label{eq:appendix_intrinsic_nonsmooth_mixed_bound}
|\mathbb E_f[\varphi]|
\le
2^{-k}\Lambda_{\mathrm{mix}}^I(\varphi)
V_I^{\mathrm{ext}}(f;S).
\end{equation}
\end{corollary}

\begin{proof}
Apply Theorem~\ref{thm:appendix_distributional_extension_bound} to an
arbitrary admissible extension and take the infimum.
\end{proof}

\begin{remark}[Interpretation]
\label{rem:appendix_extension_seminorm_interpretation}
The quantity $V_I^{\mathrm{ext}}(f;S)$ is not classical nonsmooth Vitali
variation. It is an extension mixed-variation seminorm and includes the cost
of connecting the prescribed values on $S$ to an integrable whole-space
extension. If the density vanishes in a neighborhood of $\partial S$, the internal
quantity and the extension seminorm agree. If the density reaches the boundary,
the extension cost may be strictly larger.
\end{remark}

\subsection{Uniform nonsmooth residual form}
\label{subsec:appendix_uniform_nonsmooth_residual}

\begin{corollary}[Uniform nonsmooth residual form]
\label{cor:appendix_uniform_residual_nonsmooth}
Let $\{\varphi_x:x\in X\}\subset L^\infty(S)$ satisfy
$$
\sup_{x\in X}\Lambda_{\mathrm{mix}}^I(\mathcal P_I\varphi_x)
\le L_I,
\qquad
\sup_{x\in X}\|\mathcal R_I\varphi_x\|_{L^\infty(S)}
\le \delta_I.
$$
If $f$ is a probability density and
$V_I^{\mathrm{ext}}(f;S)<\infty$, then
\begin{equation}
\label{eq:appendix_uniform_residual_nonsmooth}
\sup_{x\in X}|\mathbb E_f[\varphi_x]|
\le
2^{-k}L_IV_I^{\mathrm{ext}}(f;S)+\delta_I.
\end{equation}
\end{corollary}

\begin{proof}
Apply Theorem~\ref{thm:main_defect_adjusted_nonsmooth_mixed} pointwise in
$x$ and take the supremum.
\end{proof}

\subsection{Localization of full-support laws to a box}
\label{subsec:appendix_full_support_localization}

Let $p$ be a probability density on $\mathbb R^d$, let
$$
Z_S:=\int_Sp(\omega)\,d\omega>0,
$$
and define the conditional density
$$
f_S(\omega):=\frac{p(\omega)}{Z_S},
\qquad
\omega\in S.
$$
Write $\mathcal P_I^S$ and $\mathcal R_I^S$ for the box projectors applied to
the restriction of a residual to $S$.

\begin{theorem}[Full-support localization]
\label{thm:appendix_full_support_localization}
Let $\varphi:\mathbb R^d\to\mathbb R$ be measurable and assume
$$
\int_{\mathbb R^d\setminus S}|\varphi(\omega)|p(\omega)\,d\omega<\infty.
$$
Then:

\begin{enumerate}
\item[\emph{(i)}] If $p|_S\in BV(S)$ and
$\varphi|_S\in L^\infty(S)$, then
\begin{equation}
\label{eq:appendix_full_support_tv_localization}
\begin{aligned}
|\mathbb E_p[\varphi]|
\le{}&
\frac1{2k}
\Lambda_{\mathrm{sl}}^I(\mathcal P_I^S\varphi)
\TV_I(p;S)
\\
&+
Z_S\|\mathcal R_I^S\varphi\|_{L^\infty(S)}
+
\int_{\mathbb R^d\setminus S}|\varphi|p.
\end{aligned}
\end{equation}

\item[\emph{(ii)}] If $p|_S\in C^k(\overline S)$ and
$\varphi|_S\in L^\infty(S)$, then
\begin{equation}
\label{eq:appendix_full_support_mixed_localization}
\begin{aligned}
|\mathbb E_p[\varphi]|
\le{}&
2^{-k}
\Lambda_{\mathrm{mix}}^I(\mathcal P_I^S\varphi)
\int_S|\partial_Ip(\omega)|\,d\omega
\\
&+
Z_S\|\mathcal R_I^S\varphi\|_{L^\infty(S)}
+
\int_{\mathbb R^d\setminus S}|\varphi|p.
\end{aligned}
\end{equation}

\item[\emph{(iii)}] If $V_I^{\mathrm{ext}}(p|_S;S)<\infty$ and
$\varphi|_S\in L^\infty(S)$, then
\begin{equation}
\label{eq:appendix_full_support_nonsmooth_localization}
\begin{aligned}
|\mathbb E_p[\varphi]|
\le{}&
2^{-k}
\Lambda_{\mathrm{mix}}^I(\mathcal P_I^S\varphi)
V_I^{\mathrm{ext}}(p|_S;S)
\\
&+
Z_S\|\mathcal R_I^S\varphi\|_{L^\infty(S)}
+
\int_{\mathbb R^d\setminus S}|\varphi|p.
\end{aligned}
\end{equation}
If $D_Ip$ is a finite measure on $\mathbb R^d$, then
$$
V_I^{\mathrm{ext}}(p|_S;S)
\le
|D_Ip|(\mathbb R^d).
$$
\end{enumerate}
\end{theorem}

\begin{proof}
Decompose
$$
\mathbb E_p[\varphi]
=
\int_S\varphi p
+
\int_{\mathbb R^d\setminus S}\varphi p.
$$
The tail is bounded by its absolute integral, while
$$
\int_S\varphi p
=
Z_S\mathbb E_{f_S}[\varphi].
$$
For the anisotropic formulation,
$$
\TV_I(f_S;S)=\frac1{Z_S}\TV_I(p;S).
$$
For the smooth mixed formulation,
$$
\partial_If_S=\frac1{Z_S}\partial_Ip.
$$
For the nonsmooth formulation,
$$
V_I^{\mathrm{ext}}(f_S;S)
=
\frac1{Z_S}V_I^{\mathrm{ext}}(p|_S;S).
$$
Applying the corresponding defect-adjusted box bound and multiplying by $Z_S$
proves the three statements. If $D_Ip$ is finite globally, then $p$ itself is
an admissible extension of $p|_S$.
\end{proof}

\begin{corollary}[Uniform tail control]
\label{cor:appendix_uniform_tail_control}
If a residual family satisfies $|\varphi_x(\omega)|\le H$ for every $x$ and
almost every $\omega$, then
$$
\sup_x\int_{\mathbb R^d\setminus S}|\varphi_x|p
\le
H(1-Z_S).
$$
More generally, if $|\varphi_x|\le G$ for every $x$ with $G\in L^1(p)$, then
$$
\sup_x\int_{\mathbb R^d\setminus S}|\varphi_x|p
\le
\int_{\mathbb R^d\setminus S}Gp,
$$
and the right-hand side converges to zero along boxes exhausting
$\mathbb R^d$.
\end{corollary}

\begin{proof}
Both statements follow directly from domination and integrability.
\end{proof}

\begin{remark}[Localization trade-off]
\label{rem:appendix_localization_tradeoff}
Increasing $S$ decreases the tail, but changes the slice projectors, residual
envelopes, and variation of the restricted density. Thus the box should be
chosen by balancing tail mass against the residual geometry and the
variation certificate inside the box.
\end{remark}

\subsection{Affine images of boxes}
\label{subsec:appendix_affine_images_boxes}

Let
$$
T(y):=My+c,
$$
where $M$ is invertible, and set $\Omega:=T(S)$. Let
$$
v_i:=Me_i,
\qquad
\widehat\varphi:=\varphi\circ T,
\qquad
\widehat f:=|\det M|f\circ T.
$$
Then
$$
\int_\Omega\varphi(x)f(x)\,dx
=
\int_S\widehat\varphi(y)\widehat f(y)\,dy.
$$
Transport the projectors by
$$
(\Pi_i^T\varphi)\circ T
:=
\Pi_i(\varphi\circ T),
$$
$$
\mathcal P_I^T:=\prod_{i\in I}(\Id-\Pi_i^T),
\qquad
\mathcal R_I^T:=\Id-\mathcal P_I^T,
$$
and define
$$
\Lambda_{\mathrm{sl},T}^I(\varphi)
:=
\Lambda_{\mathrm{sl}}^I(\varphi\circ T),
\qquad
\Lambda_{\mathrm{mix},T}^I(\varphi)
:=
\Lambda_{\mathrm{mix}}^I(\varphi\circ T).
$$
For $f\in BV(\Omega)$ and $v\in\mathbb R^d$, write
$$
D_vf:=v\cdot Df.
$$
If $\mu$ is a finite signed measure on $\Omega$, its pushforward under
$T^{-1}$ is defined by
$$
((T^{-1})_\#\mu)(B):=\mu(T(B))
$$
for Borel sets $B\subset S$.

\begin{proposition}[Derivative transformation under an affine map]
\label{prop:appendix_affine_derivative_transformation}
If $f\in BV(\Omega)$, then
\begin{equation}
\label{eq:appendix_affine_first_derivative_pushforward}
D_i\widehat f=(T^{-1})_\#(D_{v_i}f),
\end{equation}
so
\begin{equation}
\label{eq:appendix_affine_first_derivative_mass}
|D_i\widehat f|(S)=|D_{v_i}f|(\Omega).
\end{equation}
If the iterated directional derivative
$D_{v_{i_1}}\cdots D_{v_{i_k}}f$ is a finite measure, then
\begin{equation}
\label{eq:appendix_affine_mixed_derivative_pushforward}
D_I\widehat f
=
(T^{-1})_\#
\bigl(D_{v_{i_1}}\cdots D_{v_{i_k}}f\bigr).
\end{equation}
For $f\in C^k(\overline\Omega)$,
\begin{equation}
\label{eq:appendix_affine_mixed_derivative_pointwise}
\partial_I\widehat f(y)
=
|\det M|
(v_{i_1}\cdot\nabla)\cdots(v_{i_k}\cdot\nabla)f(Ty),
\end{equation}
and
\begin{equation}
\label{eq:appendix_affine_mixed_derivative_mass}
\int_S|\partial_I\widehat f(y)|\,dy
=
\int_\Omega
\left|
(v_{i_1}\cdot\nabla)\cdots(v_{i_k}\cdot\nabla)f(x)
\right|dx.
\end{equation}
\end{proposition}

\begin{proof}
For $\psi\in C_c^1(\operatorname{int}S)$, set
$\zeta(x)=\psi(T^{-1}x)$. Since $v_i=Me_i$,
$$
v_i\cdot\nabla\zeta(x)
=
\partial_i\psi(T^{-1}x).
$$
Change of variables and the definition of $D_{v_i}f$ give
$$
\int_S\widehat f\,\partial_i\psi
=
\int_\Omega f\,v_i\cdot\nabla\zeta
=
-\int_\Omega\zeta\,d(D_{v_i}f),
$$
which proves \eqref{eq:appendix_affine_first_derivative_pushforward}.
Pushforward under a measurable bijection preserves the total mass of the total
variation measure, giving
\eqref{eq:appendix_affine_first_derivative_mass}. Iteration proves the mixed
distributional identity. The smooth identities follow from the chain rule and
change of variables.
\end{proof}

\begin{theorem}[Residual bounds on a parallelotope]
\label{thm:appendix_parallelotope_bounds}
Let $f$ be a probability density on $\Omega=T(S)$.

\begin{enumerate}
\item[\emph{(i)}] If $f\in BV(\Omega)$ and
$\varphi\circ T\in L^\infty(S)$, then
\begin{equation}
\label{eq:appendix_parallelotope_tv_defect_bound}
\begin{aligned}
\left|\int_\Omega\varphi f\right|
\le{}&
\frac1{2k}
\Lambda_{\mathrm{sl},T}^I(\mathcal P_I^T\varphi)
\sum_{i\in I}|D_{v_i}f|(\Omega)
\\
&+
\|\mathcal R_I^T\varphi\|_{L^\infty(\Omega)}.
\end{aligned}
\end{equation}

\item[\emph{(ii)}] If $f\in C^k(\overline\Omega)$ and
$\varphi\circ T\in L^\infty(S)$, then
\begin{equation}
\label{eq:appendix_parallelotope_mixed_defect_bound}
\begin{aligned}
\left|\int_\Omega\varphi f\right|
\le{}&
2^{-k}
\Lambda_{\mathrm{mix},T}^I(\mathcal P_I^T\varphi)
\\
&\times
\int_\Omega
\left|
(v_{i_1}\cdot\nabla)\cdots(v_{i_k}\cdot\nabla)f(x)
\right|dx
\\
&+
\|\mathcal R_I^T\varphi\|_{L^\infty(\Omega)}.
\end{aligned}
\end{equation}
\end{enumerate}
\end{theorem}

\begin{proof}
Apply the corresponding defect-adjusted box theorem to
$(\widehat\varphi,\widehat f)$ and use
Proposition~\ref{prop:appendix_affine_derivative_transformation}. The
transported projectors and envelopes were defined precisely so that all
residual-side quantities agree under pullback.
\end{proof}

\begin{remark}[Frame dependence]
\label{rem:appendix_affine_vitali_interpretation}
For $I=[d]$, the mixed quantity on $\Omega$ is the top-order directional
derivative in the affine frame $v_1,\ldots,v_d$. It is the pullback of
classical Vitali variation on $S$, but it is not classical axis-parallel
Vitali variation on $\Omega$ unless $M$ preserves the coordinate axes up to
permutation and scaling.
\end{remark}

\subsection{Connected Lipschitz domains and what is lost beyond boxes}
\label{subsec:appendix_lipschitz_domains}

Let $\Omega\subset\mathbb R^d$ be a bounded connected Lipschitz domain and let
$r>d$. Set
$$
L_0^r(\Omega)
:=
\left\{g\in L^r(\Omega):\int_\Omega g=0\right\}.
$$
By the Bogovskii right-inverse theorem, there exists a bounded linear map
\cite{AcostaDuranMuschietti2006}
$$
\mathfrak B_{\Omega,r}:L_0^r(\Omega)
\longrightarrow W_0^{1,r}(\Omega;\mathbb R^d)
$$
such that
$$
\operatorname{div}\mathfrak B_{\Omega,r}g=g.
$$
Combining its $W^{1,r}$ bound with Sobolev embedding gives
\begin{equation}
\label{eq:appendix_bogovskii_sup_bound}
\|\mathfrak B_{\Omega,r}g\|_{L^\infty(\Omega)}
\le
C_{\Omega,r}\|g\|_{L^r(\Omega)}.
\end{equation}
For $f\in BV(\Omega)$, define the coordinatewise anisotropic variation by
$$
\TV_\infty(f;\Omega)
:=
\sum_{i=1}^d|D_if|(\Omega).
$$

\begin{theorem}[A domain-dependent $BV$ residual bound]
\label{thm:appendix_lipschitz_domain_bv_bound}
Let $f\in BV(\Omega)$ be a probability density and let
$\varphi\in L^r(\Omega)$. Define
$$
m_\Omega(\varphi)
:=
\frac1{|\Omega|}\int_\Omega\varphi(x)\,dx.
$$
Then
\begin{equation}
\label{eq:appendix_lipschitz_domain_bv_bound}
|\mathbb E_f[\varphi]|
\le
C_{\Omega,r}
\|\varphi-m_\Omega(\varphi)\|_{L^r(\Omega)}
|Df|(\Omega)
+
|m_\Omega(\varphi)|.
\end{equation}
Consequently,
\begin{equation}
\label{eq:appendix_lipschitz_domain_anisotropic_bound}
|\mathbb E_f[\varphi]|
\le
C_{\Omega,r}
\|\varphi-m_\Omega(\varphi)\|_{L^r(\Omega)}
\TV_\infty(f;\Omega)
+
|m_\Omega(\varphi)|.
\end{equation}
\end{theorem}

\begin{proof}
Set $g=\varphi-m_\Omega(\varphi)$ and
$\Psi=\mathfrak B_{\Omega,r}g$. Then
$$
\operatorname{div}\Psi=g,
\qquad
\|\Psi\|_\infty\le C_{\Omega,r}\|g\|_r.
$$
Choose $\Psi_n\in C_c^\infty(\Omega;\mathbb R^d)$ converging to $\Psi$ in
$W^{1,r}$. Since $r>d$, Sobolev embedding implies uniform convergence. Also,
$BV(\Omega)\subset L^{d/(d-1)}(\Omega)\subset L^{r'}(\Omega)$ for $d\ge2$;
in dimension one, $BV(\Omega)\subset L^\infty(\Omega)$. Hence
$$
\int_\Omega f\,\operatorname{div}\Psi_n
\longrightarrow
\int_\Omega f\,\operatorname{div}\Psi.
$$
For each $n$,
$$
\int_\Omega f\,\operatorname{div}\Psi_n
=
-\int_\Omega\Psi_n\cdot d(Df).
$$
Uniform convergence gives
$$
\int_\Omega fg
=
-\int_\Omega\Psi\cdot d(Df),
$$
so
$$
\left|\int_\Omega fg\right|
\le
C_{\Omega,r}\|g\|_r|Df|(\Omega).
$$
Since $\int_\Omega f=1$,
$$
\mathbb E_f[\varphi]
=
\int_\Omega fg+m_\Omega(\varphi).
$$
This proves \eqref{eq:appendix_lipschitz_domain_bv_bound}; the anisotropic
version follows from $|Df|\le \TV_\infty(f;\Omega)$.
\end{proof}

\begin{remark}[Disconnected domains]
\label{rem:appendix_disconnected_domains}
On a disconnected domain, the mean must be removed separately on each
connected component, because a zero-trace divergence field has zero divergence
integral on every component. The natural defect produced by this construction
is therefore a collection of componentwise means.
\end{remark}

\begin{remark}[What is lost beyond boxes]
\label{rem:appendix_what_is_lost_beyond_boxes}
The general-domain result is not the box theorem with the geometry suppressed.
The explicit constant $1/(2k)$ is replaced by $C_{\Omega,r}$. The coordinate
projector, slice envelope, and slice-average convex-realizability obstruction
have no canonical general-domain analogues. The natural defect produced by the
divergence construction is the domain mean, or the componentwise means on a
disconnected domain.

The present argument also does not yield a direct intrinsic general-domain
counterpart of the pure Vitali estimate with coefficient $2^{-k}$. Classical
Vitali variation is tied to axis-parallel rectangles and alternating vertex
sums. Under a nonlinear change of variables, derivatives of the Jacobian and
of the variable coordinate frame produce lower-order terms. The clean
top-order identity is therefore special to boxes and their invertible affine
images.
\end{remark}\begin{remark}[Scope of the extension]
\label{rem:appendix_scope_nonclaims}
The precise conclusions are as follows. Tensor-product residuals and densities
admit exact projector and variation formulas. The constants $1/(2k)$ and
$2^{-k}$ are sharp in the stated classes. The nonsmooth mixed-derivative formulation
uses $V_I^{\mathrm{ext}}$, not an unqualified nonsmooth classical Vitali
variation. Full-support laws require an explicit tail term. Parallelotopes are
handled exactly in the pulled-back affine frame. General Lipschitz domains
retain a domain-dependent first-order $BV$ estimate, while the present method
does not preserve the explicit box constants or produce a pure classical
Vitali counterpart.
\end{remark}

\section{Finite-dimensional fitting and cut-generation details}
\label{app:cut_generation}

This appendix records the finite-dimensional facts used in
Section~\ref{sec:framework_design_revised}. The point is modest: the fitted
models $C_0$ and $C_I$ are ordinary finite linear programs, and the
cut-generation loop is only a way of enforcing the same supporting-plane
system. It does not define a different approximation model.

We use the notation of Section~\ref{sec:framework_design_revised}. The training
grid is $B^h=\{b^\alpha:\alpha\in\mathcal N\}\subset S$, the recourse labels are
$y_\alpha=v^{\INT}(b^\alpha)$, and a fitted pair $(u,g)$ defines the max-affine
extension
$
\widetilde v_{u,g}(b)=
\max_{\alpha\in\mathcal N}\{u_\alpha+g_\alpha^\top(b-b^\alpha)\}.
$
The fitted heights and slopes are required to satisfy the supporting-plane
inequalities
$
u_\beta\ge u_\alpha+g_\alpha^\top(b^\beta-b^\alpha)
$
for all $\alpha,\beta\in\mathcal N$.

\begin{proposition}[Finite-dimensional validity of the fitted programs]
\label{prop:finite_dimensional_validity_fitted_lps}
Fix the fitting parameters $\theta_{\fit}\in[0,1]$ and
$\lambda_{\mathrm{grad}}\ge0$. For a projector-regularized model $C_I$, also
fix a nonempty set $I\subseteq[d]$ and penalties $\mu_i\ge0$ for $i\in I$ and
$\mu_I\ge0$. Then $C_0$ and $C_I$ are feasible finite-dimensional linear
programs. If the objective has finite optimal value, an optimal solution
exists. Every feasible solution $(u,g)$ defines a convex max-affine function
$\widetilde v_{u,g}$ on $S$ satisfying
$\widetilde v_{u,g}(b^\alpha)=u_\alpha$ for every $\alpha\in\mathcal N$.
\end{proposition}

\begin{proof}
The variables in $C_0$ and $C_I$ are finite-dimensional. The
supporting-plane inequalities are linear in $(u,g)$. The absolute-value terms
in the residual fit are represented by the usual epigraph inequalities
$e_\alpha\ge u_\alpha-y_\alpha$ and $e_\alpha\ge y_\alpha-u_\alpha$, and the
worst-case residual term is represented by $t\ge e_\alpha$. The slope penalty
is linearized in the same way, by variables satisfying
$\eta_{\alpha j}\ge g_{\alpha j}$ and
$\eta_{\alpha j}\ge -g_{\alpha j}$.

For $C_I$, the residual array $R(u)_\alpha=u_\alpha-y_\alpha$ is affine in
$u$. The discrete averaging operators, the discrete projectors
$\mathcal P_I^h$ and $\mathcal R_I^h$, and the corresponding slice-mean
quantities are linear maps of $R(u)$. Their absolute values are again handled
by epigraph variables. Hence all constraints and all objective terms in
$C_0$ and $C_I$ are linear.

Feasibility is immediate: choose any constant $c$, set $u_\alpha=c$ and
$g_\alpha=0$ for all $\alpha\in\mathcal N$, and choose the epigraph variables
large enough. The feasible region is therefore a nonempty closed polyhedron.
A linear objective with finite infimum over a nonempty polyhedron attains its
infimum, so an optimal solution exists whenever the optimal value is finite.

It remains only to justify the interpolation statement. If $(u,g)$ satisfies
the supporting-plane inequalities, then for every grid point $b^\beta$,
$
u_\alpha+g_\alpha^\top(b^\beta-b^\alpha)\le u_\beta
$
for all $\alpha\in\mathcal N$. Taking the maximum over $\alpha$ gives
$\widetilde v_{u,g}(b^\beta)\le u_\beta$. The reverse inequality follows by
choosing $\alpha=\beta$. Thus $\widetilde v_{u,g}(b^\beta)=u_\beta$. Since
$\widetilde v_{u,g}$ is the pointwise maximum of affine functions, it is convex
on $S$.
\end{proof}

The full supporting-plane system contains one inequality for every ordered pair
of training points. If $B^h=\{b^1,\ldots,b^N\}$, this system is
$u_j\ge u_i+g_i^\top(b^j-b^i)$ for $i,j=1,\ldots,N$. The diagonal inequalities
are redundant but harmless. In the implementation, a restricted LP is solved
first, violated supporting-plane inequalities are identified by checking the
slacks $s_{ij}=u_j-u_i-g_i^\top(b^j-b^i)$, and the violated inequalities are
added.

\begin{proposition}[Correctness of cut generation]
\label{prop:cut_generation_correctness}
Suppose the cut-generation loop is applied to either $C_0$ or $C_I$. If the
loop terminates with no violated supporting-plane inequality, then the returned
solution is feasible for the full supporting-plane LP. With exact separation,
it is optimal for the full LP. If separation is performed with tolerance $\tau>0$, termination guarantees
that every supporting-plane inequality is satisfied up to $\tau$. No
objective-gap estimate follows from this feasibility tolerance alone.
\end{proposition}

\begin{proof}
There are finitely many supporting-plane inequalities. At each iteration, the
restricted LP contains a subset of them and has the same objective as the full
LP. If the separation check finds a negative slack $s_{ij}$, the corresponding
missing inequality is added. If no violation remains, the current restricted
solution satisfies every supporting-plane inequality and is therefore feasible
for the full LP.

The restricted LP is a relaxation of the full LP, so its optimal value is no
larger than the full optimal value. Once the restricted optimum is feasible for
the full LP, its objective value is also attainable in the full LP. Hence the
two optimal values coincide, and the returned solution is optimal for the full
LP. If separation uses a tolerance $\tau>0$, termination guarantees only
$\tau$-feasibility of the supporting-plane system. Without an additional
error-bound argument, this feasibility tolerance does not imply a bound on
the objective gap relative to the full LP.
\end{proof}

This proves that the reported fitted approximations are solutions of the
finite-dimensional max-affine fitting models stated in the main text. The
cut-generation loop is only a computational device for reaching those models
without loading all pairwise supporting-plane inequalities at once.

\section{Details for the TU ceiling obstruction}
\label{app:tu_ceiling_obstruction}

This appendix gives the full verification of the obstruction stated in
Proposition~\ref{prop:2d_tu_no_exact_smz_main}. The example is deliberately
minimal: the recourse matrix is the identity, hence totally unimodular, and
the value function is a separable ceiling function. The obstruction does not
come from complicated recourse geometry. It comes from the incompatibility
between exact slice-centering and convex representability of the induced
slice-average profile.
\begin{proposition}[A separable two-dimensional TU ceiling obstruction]
\label{prop:2d_tu_no_exact_smz_app}
Let $S=[0,2]^2$ and define
$$
v^{\INT}(b)
:=
\min_{y\in\mathbb Z_{\ge0}^2}
\{\mathbf 1^\top y:y\ge b\},
\qquad
b=(b_1,b_2)\in S.
$$
Then
$$
v^{\INT}(b_1,b_2)=\lceil b_1\rceil+\lceil b_2\rceil
\qquad
\text{a.e. on }S.
$$
There is no finite-valued convex function $\widetilde v$ on any open
neighborhood of $S$ such that the residual
$\varphi:=\widetilde v-v^{\INT}$ satisfies
$$
\bar\Pi_1\varphi=0
\qquad
\text{a.e. on }[0,2].
$$
Consequently, exact all-axes slice-centering is impossible for this instance.
\end{proposition}

\begin{proof}
The second-stage constraint matrix is the identity and is therefore totally
unimodular. The problem is separable:
$$
v^{\INT}(b_1,b_2)
=
\min_{y_1\in\mathbb Z_{\ge0}}\{y_1:y_1\ge b_1\}
+
\min_{y_2\in\mathbb Z_{\ge0}}\{y_2:y_2\ge b_2\}.
$$
Hence $v^{\INT}(b_1,b_2)=\lceil b_1\rceil+\lceil b_2\rceil$ a.e. on $S$.

The obstruction comes from the slice-average profile. Averaging
$v^{\INT}(t,b_2)=\lceil t\rceil+\lceil b_2\rceil$ over $t\in[0,2]$ produces a
step function of $b_2$ with a jump at $b_2=1$, while a slice average of a
finite-valued convex function must be continuous.
Suppose, to the contrary, that such a finite-valued convex $\widetilde v$
exists on an open neighborhood of $S$. Define
$$
F(b_2):=(\bar\Pi_1\widetilde v)(b_2)
=
\frac12\int_0^2\widetilde v(t,b_2)\,dt
$$
and
$$
G(b_2):=(\bar\Pi_1v^{\INT})(b_2)
=
\frac12\int_0^2v^{\INT}(t,b_2)\,dt.
$$
The centering condition gives $F=G$ a.e. on $[0,2]$. Since
$\widetilde v$ is finite-valued and convex on an open neighborhood of $S$, it
is continuous there. Therefore $F$ is continuous on $[0,2]$.

For a.e. $b_2\in[0,2]$,
$$
G(b_2)
=
\frac12\int_0^2\lceil t\rceil\,dt+\lceil b_2\rceil
=
\frac32+\lceil b_2\rceil .
$$
Thus $G(b_2)=5/2$ for a.e. $b_2\in(0,1)$ and
$G(b_2)=7/2$ for a.e. $b_2\in(1,2)$. Since $F$ is continuous and agrees with
$G$ a.e., it follows that
$$
F(b_2)=\frac52
\qquad
\text{for all } b_2\in(0,1),
$$
and
$$
F(b_2)=\frac72
\qquad
\text{for all } b_2\in(1,2).
$$
This contradicts continuity of $F$ at $b_2=1$. Hence no such
$\widetilde v$ exists.

All-axes centering implies one-direction centering in each coordinate. Since
direction-$1$ centering is already impossible, all-axes centering is impossible
a fortiori.
\end{proof}
\begin{remark}[Why the example is enough]
\label{rem:tu_obstruction_interpretation}
The example is not meant to show that every mixed-integer recourse function
prevents exact centering. It shows that exact centering is not a generic
normalization available inside convex approximation classes. Even for an
identity-matrix, separable, totally unimodular ceiling value, exact
coordinate-wise centering would require a discontinuous slice-average profile
to be represented by a finite-valued convex function. This is impossible.
Thus the defect term in the main paper is structurally necessary rather than
only technically convenient.
\end{remark}

\section{Periodic cancellation and finite-box LP-slope calibration}
\label{app:periodic_finite_box_shifted_lp}

This appendix separates two objects that are easily conflated. The first is
the classical shifted-LP mechanism: a periodic residual with zero cell mean is
integrated against a density and controlled through a one-dimensional $BV$
norm. The second is the comparator $v^\Gamma$ used in the numerical study. It
has the same fixed-slope form, but its intercepts are calibrated on the finite
training box. Thus $v^\Gamma$ is a finite-box fixed-slope comparator, not an
implementation of the classical full-cell shifted-LP constants.

The distinction matters because the theory in the main text is box-based. It
uses coordinate slice-centering and the decomposition
$$
\varphi=\mathcal P_I\varphi+\mathcal R_I\varphi.
$$
Classical shifted-LP arguments instead use periodic cancellation. These two
mechanisms coincide in the scalar unit-lattice ceiling example, but they are
not the same object on arbitrary bounded boxes or in higher dimensions.

\subsection{Periodic cancellation on the line } 
\label{subsec:periodic_comparison_appendix}

We first record the elementary one-dimensional periodic estimate behind the
comparison.

\begin{definition}[Periodic primitive radius]
\label{def:periodic_primitive_radius_appendix}
Let $\psi\in L^\infty_{\mathrm{loc}}(\mathbb R)$ be $1$-periodic and satisfy
$$
\int_0^1\psi(t)\,dt=0.
$$
Define
$$
P(u):=\int_0^u\psi(t)\,dt,
\qquad u\in[0,1],
$$
and extend $P$ periodically to $\mathbb R$. The recentered periodic primitive
radius is
$$
\rho_{\mathrm{per}}(\psi)
:=
\inf_{c\in\mathbb R}\|P-c\|_{L^\infty(\mathbb R)}.
$$
Equivalently,
$$
\rho_{\mathrm{per}}(\psi)
=
\frac12\operatorname{osc}_{[0,1]}P.
$$
\end{definition}

\begin{theorem}[Recentered periodic $BV$ bound]
\label{thm:periodic_TV_bound_appendix}
Let $\psi\in L^\infty_{\mathrm{loc}}(\mathbb R)$ be $1$-periodic and satisfy
$$
\int_0^1\psi(t)\,dt=0.
$$
Then, for every probability density $g\in BV(\mathbb R)\cap L^1(\mathbb R)$,
\begin{equation}
\label{eq:periodic_TV_bound_appendix}
\left|
\int_{\mathbb R}\psi(t)g(t)\,dt
\right|
\le
\rho_{\mathrm{per}}(\psi)\,\TV_\infty(g;\mathbb R).
\end{equation}
In one dimension, $\TV_\infty(g;\mathbb R)=|Dg|(\mathbb R)$.
\end{theorem}

\begin{proof}
Let $P(u)=\int_0^u\psi(t)\,dt$ on $[0,1]$ and extend it periodically. Since
$\psi$ has zero mean over one period, this extension is well defined,
$P\in W^{1,\infty}_{\mathrm{loc}}(\mathbb R)$, and $P'=\psi$ a.e. on
$\mathbb R$.

A one-dimensional $BV$ function has one-sided limits at $\pm\infty$. Since
$g\in BV(\mathbb R)\cap L^1(\mathbb R)$, these limits are both zero, and hence
$Dg(\mathbb R)=0$.

We first justify integration by parts on $\mathbb R$. Choose
$\chi\in C_c^\infty(\mathbb R)$ with $\chi=1$ on $[-1,1]$ and $\chi=0$ outside
$[-2,2]$, and set $\chi_R(t):=\chi(t/R)$. Since $\chi_Rg$ has compact support
and bounded variation,
$$
\int_{\mathbb R}P'(t)\chi_R(t)g(t)\,dt
=
-\int_{\mathbb R}P(t)\,dD(\chi_Rg)(t).
$$
Using $D(\chi_Rg)=\chi_R\,Dg+g\chi_R'(t)\,dt$, we get
$$
\int_{\mathbb R}\psi(t)\chi_R(t)g(t)\,dt
=
-\int_{\mathbb R}P(t)\chi_R(t)\,dDg(t)
-
\int_{\mathbb R}P(t)g(t)\chi_R'(t)\,dt.
$$
The left-hand side converges to $\int_{\mathbb R}\psi g$ by dominated
convergence, and the first term on the right converges to
$-\int_{\mathbb R}P\,dDg$. For the cutoff term,
$$
\left|
\int_{\mathbb R}P(t)g(t)\chi_R'(t)\,dt
\right|
\le
\frac{\|P\|_{L^\infty}\|\chi'\|_{L^\infty}}{R}
\int_{\{R\le |t|\le 2R\}} |g(t)|\,dt
\to0.
$$
Thus
\begin{equation}
\label{eq:periodic_ibp_identity}
\int_{\mathbb R}\psi(t)g(t)\,dt
=
-\int_{\mathbb R}P(t)\,dDg(t).
\end{equation}

Since $Dg(\mathbb R)=0$, constants may be subtracted from the primitive:
$$
\int_{\mathbb R}P(t)\,dDg(t)
=
\int_{\mathbb R}(P(t)-c)\,dDg(t),
\qquad c\in\mathbb R.
$$
Therefore
$$
\left|
\int_{\mathbb R}\psi(t)g(t)\,dt
\right|
\le
\|P-c\|_{L^\infty(\mathbb R)}\,|Dg|(\mathbb R).
$$
Taking the infimum over $c$ gives the result.
\end{proof}

\begin{corollary}[Scalar shifted-LP ceiling residual]
\label{cor:scalar_shifted_lp_ceiling_residual}
Let $q>0$ and consider the scalar unit-lattice ceiling value
$$
v^{\INT}(s)=q\lceil s\rceil,
\qquad s\in\mathbb R.
$$
Its LP relaxation is $v^{\LP}(s)=qs$, and the half-shifted LP approximation is
$$
v^{\SH}(s)=q\left(s+\frac12\right).
$$
Define $\varphi^{\SH}:=v^{\SH}-v^{\INT}$. Then, for every integer $\ell$,
$$
\varphi^{\SH}(s)
=
q\left(s-\ell-\frac12\right)
\qquad
\text{for a.e. }s\in[\ell,\ell+1].
$$
Thus $\varphi^{\SH}$ is the centered unit-periodic sawtooth. Its one-period
primitive is
$$
P(u)
=
\int_0^u q\left(t-\frac12\right)\,dt
=
q\left(\frac{u^2}{2}-\frac u2\right),
\qquad u\in[0,1],
$$
so
$$
\operatorname{osc}_{[0,1]}P=\frac q8,
\qquad
\rho_{\mathrm{per}}(\varphi^{\SH})=\frac q{16}.
$$
The factor $q/16$ is the recentered primitive radius; using the unrecentered
primitive $P$ would give the weaker constant $q/8$.
Consequently, for every probability density
$g\in BV(\mathbb R)\cap L^1(\mathbb R)$,
\begin{equation}
\label{eq:scalar_shifted_lp_periodic_bound}
\left|
\int_{\mathbb R}\varphi^{\SH}(s)g(s)\,ds
\right|
\le
\frac q{16}\,\TV_\infty(g;\mathbb R).
\end{equation}
\end{corollary}

\begin{proof}
For $s\in(\ell,\ell+1]$, one has $\lceil s\rceil=\ell+1$. Hence
$$
\varphi^{\SH}(s)
=
q\left(s+\frac12\right)-q(\ell+1)
=
q\left(s-\ell-\frac12\right).
$$
Endpoint values are immaterial. The primitive computation is displayed above,
and the bound follows from Theorem~\ref{thm:periodic_TV_bound_appendix}.
\end{proof}

\begin{remark}[Why this is only a comparison]
\label{rem:periodic_only_comparison}
The scalar ceiling example is special. Its shifted-LP residual is both
periodic with zero cell mean and centered on every unit interval. The main
paper does not assume such periodic structure. On bounded boxes in dimension
$d\ge2$, periodic cell means and coordinate slice means are different objects.
The residual theory in the main text is therefore box-based, not periodic.
\end{remark}

\subsection{Finite-box LP-slope calibration}
\label{subsec:finite_box_shifted_lp_calibration}

We now record the finite-box fixed-slope comparator used in
Section~\ref{sec:framework_design_revised}. Suppose that the LP relaxation has
the finite max-affine representation
$$
v^{\LP}(b)
=
\max_{k\in[K]}\{a_k+s_k^\top b\},
\qquad
[K]:=\{1,\ldots,K\}.
$$
The LP-slope calibration comparator keeps the LP slopes and fits only the vertical
corrections:
$$
b\mapsto
\max_{k\in[K]}\{a_k+s_k^\top b+\gamma_k\}.
$$
For each training point $b^\alpha$, choose an LP-active piece
$$
k(\alpha)
\in
\arg\max_{\ell\in[K]}
\{a_\ell+s_\ell^\top b^\alpha\},
$$
using deterministic tie-breaking, and define
$
\mathcal N_k:=\{\alpha\in\mathcal N:k(\alpha)=k\}.
$
Let $\kappa_\alpha^{\mathrm{tr}}$ be normalized training weights and set
$$
\bar g
:=
\sum_{\alpha\in\mathcal N}
\kappa_\alpha^{\mathrm{tr}}
\bigl(y_\alpha-v^{\LP}(b^\alpha)\bigr).
$$
For $\tau>0$, the finite-box intercept is
\begin{equation}
\label{eq:finite_box_gamma_formula_appendix}
\gamma_k
=
\frac{
\sum_{\alpha\in \mathcal N_k}
\kappa_\alpha^{\mathrm{tr}}
\bigl(y_\alpha-a_k-s_k^\top b^\alpha\bigr)
+
\tau\bar g
}{
\sum_{\alpha\in \mathcal N_k}\kappa_\alpha^{\mathrm{tr}}+\tau
}.
\end{equation}
If $\mathcal N_k=\emptyset$, then $\gamma_k=\bar g$. The corresponding
comparator is
$$
v^\Gamma(b)
=
\max_{k\in[K]}
\{a_k+s_k^\top b+\gamma_k\}.
$$

The constants in \eqref{eq:finite_box_gamma_formula_appendix} depend on the
finite box, the training weights, and the LP-active assignment. They are not
classical full-cell constants, which are derived from periodic residuals
averaged over full lattice or basis cells. The next example shows that the
finite-box and full-cell constants can differ even in one dimension.

\begin{proposition}[Finite-box and full-cell shifted constants can differ]
\label{prop:finite_box_vs_full_cell_shift}
Consider
$$
v(s)
=
\min
\left\{
y_1+2y_2+2y_3:
y_1+y_2-y_3=s,\ 
y_1\in\mathbb Z_+,\ 
y_2,y_3\in\mathbb R_+
\right\}.
$$
Then $v^{\LP}(s)=\max\{s,-2s\}$. On $[0,1]$, one has
$v^{\LP}(s)=s$ and
$$
v(s)
=
\begin{cases}
2s, & 0\le s\le 3/4,\\[1mm]
3-2s, & 3/4\le s\le 1.
\end{cases}
$$
Hence the LP gap $\psi:=v-v^{\LP}$ satisfies
$$
\psi(s)
=
\begin{cases}
s, & 0\le s\le 3/4,\\[1mm]
3-3s, & 3/4\le s\le 1.
\end{cases}
$$
The full-cell shifted constant is
$$
\Gamma^W
=
\int_0^1\psi(s)\,ds
=
\frac38.
$$
By contrast, if the same fixed-slope shift is calibrated uniformly only on
$D=[0,1/2]$, then the mean-square optimal finite-domain shift is
$$
\gamma_D
=
\frac1{|D|}\int_D\psi(s)\,ds
=
\frac14.
$$
Thus $\gamma_D\neq\Gamma^W$.
\end{proposition}

\begin{proof}
For fixed $s$ and fixed integer $y_1$, the continuous variables satisfy
$$
y_2-y_3=s-y_1,
\qquad
y_2,y_3\ge0.
$$
The minimum of $2y_2+2y_3$ under this constraint is $2|s-y_1|$. Hence
$$
v(s)
=
\min_{y_1\in\mathbb Z_+}
\{y_1+2|s-y_1|\}.
$$

If $y_1$ is relaxed to be continuous and nonnegative, then for $s\ge0$ the
minimizer is $y_1=s$ and the value is $s$. For $s<0$, the constraint
$y_1\ge0$ forces $y_1=0$, giving value $-2s$. Thus
$
v^{\LP}(s)=\max\{s,-2s\}.
$\\
On $[0,1]$, only $y_1=0$ and $y_1=1$ can be optimal. These choices give
$
y_1=0:\quad 2s$, and $
y_1=1:\quad 1+2(1-s)=3-2s
$.
Therefore $v(s)=\min\{2s,3-2s\}$, with breakpoint $s=3/4$. Since
$v^{\LP}(s)=s$ on $[0,1]$, the displayed expression for $\psi$ follows.
The full-cell average is
$$
\Gamma^W
=
\int_0^{3/4}s\,ds
+
\int_{3/4}^{1}(3-3s)\,ds
=
\frac{9}{32}+\frac{3}{32}
=
\frac38.
$$
On $D=[0,1/2]$, the gap is $\psi(s)=s$. For a fixed-slope approximation
$s\mapsto s+\gamma$, the residual is $\gamma-\psi(s)$, so the mean-square
optimal shift is the average of $\psi$ over $D$:
$$
\gamma_D
=
\frac1{|D|}\int_D\psi(s)\,ds
=
2\int_0^{1/2}s\,ds
=
\frac14.
$$
This proves the claim.
\end{proof}
\begin{remark}[Interpretation of $v^\Gamma$]
\label{rem:finite_box_shifted_lp_interpretation}
The comparator $v^\Gamma$ keeps the LP affine geometry and asks how much of the
finite-box residual can be removed by vertical corrections. This is the
relevant fixed-slope comparison for the direct max-affine fits, which are
trained and audited on the same finite domain. It should not be read as a
classical full-cell shifted-LP approximation.
\end{remark}

\begin{remark}[When the constants coincide]
\label{rem:finite_box_full_cell_coincidence}
Finite-box and full-cell constants may coincide in special cases, for example
under uniform weighting over an integer number of complete cells inside a
single LP-active regime. Such coincidence should not be expected on restricted
boxes, under nonuniform training weights, or in higher-dimensional designs that
cross LP-active regions unevenly.
\end{remark}

\bibliographystyle{amsplain}
\bibliography{main}

@article{VanderVlerk2004,
  author  = {M. H. van der Vlerk},
  title   = {Convex Approximations for Complete Integer Recourse Models},
  journal = {Mathematical Programming},
  volume  = {99},
  number  = {2},
  pages   = {297--310},
  year    = {2004}
}

@article{vanBeestenRomeijnders2022,
  author  = {{van Beesten, E. Ruben and Romeijnders, Ward}},
  title   = {Parametric Error Bounds for Convex Approximations
             of Two-Stage Mixed-Integer Recourse Models with a
             Random Second-Stage Cost Vector},
  journal = {Operations Research Letters},
  volume  = {50},
  number  = {5},
  pages   = {541--547},
  year    = {2022},
  doi     = {10.1016/j.orl.2022.07.012}
}

@article{vanDerLaanEtAl2018HigherOrder,
  author  = {van der Laan, Niels and Romeijnders, Ward
             and van der Vlerk, Maarten H.},
  title   = {Higher-Order Total Variation Bounds for Expectations
             of Periodic Functions and Simple Integer Recourse
             Approximations},
  journal = {Computational Management Science},
  volume  = {15},
  number  = {3--4},
  pages   = {325--349},
  year    = {2018},
  doi     = {10.1007/s10287-018-0315-z}
}

@article{vanBeestenRomeijnders2020,
  author  = {van Beesten, E. Ruben and Romeijnders, Ward},
  title   = {Convex Approximations for Two-Stage Mixed-Integer
             Mean-Risk Recourse Models with
             {Conditional Value-at-Risk}},
  journal = {Mathematical Programming},
  volume  = {181},
  number  = {2},
  pages   = {473--507},
  year    = {2020},
  doi     = {10.1007/s10107-019-01428-6}
}

@article{vanDerLaanEtAl2021,
  author  = {van der Laan, Niels and Romeijnders, Ward},
  title   = {A Loose Benders Decomposition Algorithm for Approximating Two-Stage Mixed-Integer Recourse Models},
  journal = {Mathematical Programming},
  volume  = {190},
  number  = {1--2},
  pages   = {761--794},
  year    = {2021},
  doi     = {10.1007/s10107-020-01559-1}
}

@book{shapiro,
  author    = {Shapiro, Alexander and Dentcheva, Darinka and Ruszczy{\'n}ski, Andrzej},
  title     = {Lectures on Stochastic Programming: Modeling and Theory},
  series    = {MOS-SIAM Series on Optimization},
  publisher = {Society for Industrial and Applied Mathematics},
  address   = {Philadelphia, PA},
  edition   = {3},
  year      = {2021},
  doi       = {10.1137/1.9781611976595}
}

@phdthesis{vlerkthesis,
  author    = {van der Vlerk, Maarten H.},
  title     = {Stochastic Programming with Integer Recourse},
  school    = {University of Groningen, SOM Research School},
  year      = {1995}
}

@book{wardthesis,
  author    = {{Romeijnders, Ward}},
  title     = {Total Variation Error Bounds for Convex Approximations of Two-Stage Mixed-Integer Recourse Models},
  publisher = {University of Groningen, SOM Research School},
  year      = {2015}
}

@article{vlerk1,
  author    = {Louveaux, F. V. and van der Vlerk, Maarten H.},
  title     = {Stochastic Programming with Simple Integer Recourse},
  journal   = {Mathematical Programming},
  volume    = {61},
  pages     = {301--325},
  year      = {1993}
}

@article{schultz2,
  author    = {Schultz, R{\"u}diger},
  title     = {Stochastic Programming with Integer Variables},
  journal   = {Mathematical Programming},
  volume    = {97},
  pages     = {285--309},
  year      = {2003}
}

@article{ward2015,
  author    = {Romeijnders, Ward and van der Vlerk, Maarten H. and Klein Haneveld, Willem K.},
  title     = {Convex Approximations for Totally Unimodular Integer Recourse Models: A Uniform Error Bound},
  journal   = {SIAM Journal on Optimization},
  volume    = {25},
  number    = {1},
  pages     = {130--158},
  year      = {2015},
  doi       = {10.1137/130945703}
}

@article{total1,
  author  = {{Romeijnders, Ward and van der Vlerk, Maarten H.
             and Klein Haneveld, Willem K.}},
  title   = {Total Variation Bounds on the Expectation of Periodic
             Functions with Applications to Recourse Approximations},
  journal = {Mathematical Programming},
  volume  = {157},
  number  = {1},
  pages   = {3--46},
  year    = {2016},
  doi     = {10.1007/s10107-014-0829-2}
}

@article{ward2015bb,
  author    = {Romeijnders, Ward and Schultz, R{\"u}diger and van der Vlerk, Maarten H. and Klein Haneveld, Willem K.},
  title     = {A Convex Approximation for Two-Stage Mixed-Integer Recourse Models with a Uniform Error Bound},
  journal   = {SIAM Journal on Optimization},
  volume    = {26},
  number    = {1},
  pages     = {426--447},
  year      = {2016},
  doi       = {10.1137/140986244}
}

@article{ward_assessing,
  author    = {Romeijnders, Ward and Morton, David P. and van der Vlerk, Maarten H.},
  title     = {Assessing the Quality of Convex Approximations for Two-Stage Totally Unimodular Integer Recourse Models},
  journal   = {INFORMS Journal on Computing},
  volume    = {29},
  number    = {2},
  pages     = {211--231},
  year      = {2017},
  doi       = {10.1287/ijoc.2016.0725}
}

@article{alban1,
  author    = {Kryeziu, Alban and Romeijnders, Ward and Ursavas, Evrim},
  title     = {Vitali Variation Error Bounds for Expected Value Functions},
  journal   = {Operations Research Letters},
  volume    = {56},
  pages     = {107157},
  year      = {2024},
  doi       = {10.1016/j.orl.2024.107157}
}

@article{MagnaniBoyd2009ConvexPWL,
  author    = {Magnani, Alessandro and Boyd, Stephen P.},
  title     = {Convex Piecewise-Linear Fitting},
  journal   = {Optimization and Engineering},
  volume    = {10},
  number    = {1},
  pages     = {1--17},
  year      = {2009},
  doi       = {10.1007/s11081-008-9045-3}
}

@article{seijo_sen,
  author    = {Seijo, Emilio and Sen, Bodhisattva},
  title     = {Nonparametric Least Squares Estimation of a Multivariate Convex Regression Function},
  journal   = {The Annals of Statistics},
  volume    = {39},
  number    = {3},
  pages     = {1633--1657},
  year      = {2011},
  doi       = {10.1214/10-AOS852}
}

@inproceedings{input_convex_neural_networks,
  author    = {Amos, Brandon and Xu, Lei and Kolter, J. Zico},
  title     = {Input Convex Neural Networks},
  booktitle = {Proceedings of the 34th International Conference on Machine Learning},
  series    = {Proceedings of Machine Learning Research},
  volume    = {70},
  pages     = {146--155},
  year      = {2017}
}

@article{KAZDA2024493,
  author    = {Kazda, Kody and Li, Xiang},
  title     = {A Linear Programming Approach to Difference-of-Convex Piecewise Linear Approximation},
  journal   = {European Journal of Operational Research},
  volume    = {312},
  number    = {2},
  pages     = {493--511},
  year      = {2024},
  doi       = {10.1016/j.ejor.2023.07.026}
}

@book{giovanni_leoni_sobolev,
  author    = {Leoni, Giovanni},
  title     = {A First Course in Sobolev Spaces},
  publisher = {American Mathematical Society},
  year      = {2017},
  edition   = {2}
}

@book{AmbrosioFuscoPallara2000,
  author    = {Ambrosio, Luigi and Fusco, Nicola and Pallara, Diego},
  title     = {Functions of Bounded Variation and Free Discontinuity Problems},
  series    = {Oxford Mathematical Monographs},
  publisher = {Clarendon Press},
  address   = {Oxford},
  year      = {2000}
}

@book{evans,
  author    = {Evans, Lawrence C.},
  title     = {Partial Differential Equations},
  publisher = {American Mathematical Society},
  address   = {Providence, RI},
  year      = {2010},
  edition   = {2}
}

@article{ChenLuedtke2022LagrangianCuts,
  author  = {Chen, Rui and Luedtke, James},
  title   = {On Generating Lagrangian Cuts for Two-Stage Stochastic Integer Programs},
  journal = {INFORMS Journal on Computing},
  volume  = {34},
  number  = {4},
  pages   = {2332--2349},
  year    = {2022},
  doi     = {10.1287/ijoc.2022.1185}
}

@article{AcostaDuranMuschietti2006,
  author  = {Acosta, Gabriel and Dur{\'a}n, Ricardo G. and Muschietti, Mar{\'i}a A.},
  title   = {Solutions of the Divergence Operator on John Domains},
  journal = {Advances in Mathematics},
  volume  = {206},
  number  = {2},
  pages   = {373--401},
  year    = {2006},
  doi     = {10.1016/j.aim.2005.09.004}
}

\end{document}